\documentclass[a4paper,11pt]{article}
\usepackage{geometry}
\usepackage{mathtools, amsthm, amsmath, amssymb, amsfonts, bbm}
\usepackage{etoolbox}
\usepackage[colorlinks=true,citecolor=red,linkcolor=blue,urlcolor=RubineRed,pdfpagetransition=Blinds,pdftoolbar=false,pdfmenubar=false]{hyperref}

\numberwithin{equation}{section}

\usepackage{graphicx, framed}
\graphicspath{{Images/}}
\usepackage{afterpage, xcolor}

\usepackage{thmtools}
\theoremstyle{plain}
\declaretheorem[title=Theorem, parent=section]{theorem}

\declaretheorem[title=Proposition,sibling=theorem]{proposition}
\declaretheorem[title=Corollary,sibling=theorem]{corollary}

\theoremstyle{definition}
\declaretheorem[title=Definition,sibling=theorem]{definition}
\declaretheorem[title=Remark,sibling=theorem]{remark}
\declaretheorem[title=Remark, numbered=no]{remark*}

\declaretheorem[title=Assumption, numbered=no]{assumption*}

\newcommand{\cs}{c^*}

\newcommand{\di}{\displaystyle}
\newcommand{\RR}{\mathbb{R}}
\newcommand{\bm}{\mathbf{m}}

\newcommand{\bu}{\mathbf{u}}
\newcommand{\bi}{\mathbf{i}}

\newcommand{\ws}{w^*}
\newcommand{\ls}{\lambda^*}
\newcommand{\e}[1]{e^{#1}}
\newcommand{\cd}{{\cdot}}
\newcommand{\varep}{\varepsilon}
\newcommand{\bnu}{\boldsymbol\nu}

\newcommand{\R}{\mathbb{R}}

\newcommand{\cI}{\mathcal{I}}
\newcommand{\cW}{\mathcal{W}}

\newcommand{\cN}{\mathcal{N}}

\newcommand{\ck}{\mathcal{K}}
\newcommand{\ci}{\mathcal{I}}

\newcommand{\cV}{\mathcal{V}}

\newcommand{\diag}{\mathrm{diag}}

\newcommand{\md}{\mathrm{d}}

\newcommand{\rn}{\mathbb{R}^N}
\newcommand{\sn}{\mathbb{S}^{N-1}}

\renewcommand{\Re}{\textnormal{Re}}

\allowdisplaybreaks

\title{\bf
	Large time behaviour in the multidimensional nonlocal Fisher-KPP equation}
\author{ Jean-Michel Roquejoffre$^{a}$~and~Mingmin Zhang$^{b}$
\thanks{This work has been supported by the French ANR  ReaCh  project (ANR-23-CE40-0023-01).  J.-M. R. acknowledges, in addition, an invitation by the School of Mathematical Sciences of the University of Science and Technology of China, which was instrumental in the completion of the work. 
		Email addresses:  jean-michel.roquejoffre@math.univ-toulouse.fr; mingmin.zhang.math@gmail.com.}}
\date{}

\begin{document}
\maketitle
\vspace{-5mm} 

\begin{center}
	\small
	$^a$ Univ Toulouse, CNRS, Institut de Math\'ematiques de Toulouse, Toulouse, France\\[1mm]
	$^b$ School of Mathematical Sciences, University of Science and Technology of China, Hefei, Anhui 230026,  China
\end{center}

\begin{abstract}
This work studies the expansion of the level sets of the solutions of a Fisher-KPP type reaction-diffusion equation, the diffusion being given by a nonnegative, compactly supported, even kernel. Because the latter is not assumed to be shperically symmetric, the level sets advance at an asymptotic speed given by a Freidlin-G\"artner type formula. The main contribution of this paper is to push the expansion up to terms that vanish as $t\to+\infty$, consisting in an asymptotically logarithmic delay corrected by an function that is constant in time, but depends on the direction of propagation.
\end{abstract}

\section{Introduction and main results}
\subsection{Model and question}
This paper is devoted to the sharp large-time asymptotics for the Cauchy problem for the reaction-diffusion equation
\begin{equation}\label{kpp}
	u_t=\ck*u-u+f(u), ~~~~~ t>0, ~~x\in \rn,
\end{equation}
starting from a smooth, nontrivial, and compactly supported initial data $u_0$  satisfying $0\le u_0(x)\le 1$. The nonlinearity $f$ is a $C^1$ function of Fisher-KPP type, i.e. 
\begin{equation}\label{f-cdn1}
f(0)=f(1)=0, \ f'(0)>0, \ f'(s)\leq f'(0)  ~~\text{ for all } s\in(0,1),
\end{equation}
and there exists $M>0$, $\alpha>0$ and $s_0\in(0,1)$ such that
\begin{equation}\label{f-cdn2}
f(s)\ge f'(0)s-Ms^2~~ \ \text{ for all } s\in(0, s_0],
\end{equation}
for which a typical example is $f(u)=u-u^2$.  More than likely, the assumption on $f'$ can be replaced by the usual assumption $f(s)\leq f'(0)s$ for all $s\in[0,1]$.

These equations arise naturally in many mathematical models in biology, ecology, chemical kinetics, etc., with $u(t,x)$ usually denoting a local population density. They appeared originally in the works of Fisher \cite{fisher} and Kolmogorov, Petrovskii, Piskunov \cite{kpp} in the context of population genetics, and have ever since been extensively studied with both PDE and probabilistic tools, due to their connection to branching Brownian motion (see \cite{mck}). The solutions of \eqref{kpp} are characterised by the invasion of the unstable steady state by the stable one, in that case $0$ and $1$ respectively, as determined by the conditions on $f$. The question in this paper is to  study precisely the geometry and speed at which this invasion occurs.
\subsection{Known results}
Model \eqref{kpp} is related to the well-known equation 
\begin{equation}
\label{kpp1}
u_t-\Delta u=f(u), ~~~~t>0, ~~x\in\rn,
\end{equation}
the function $f$ satisfying assumption \eqref{f-cdn1}.  The goal of the present section is to explain the connection of the question with the rich literature, pertaining both to models \eqref{kpp1} and \eqref{kpp}.
\subsubsection{Sharp asymptotics in one space dimension}
For \eqref{kpp1} with $N=1$, Bramson \cite{bram1,bram2} proves the following celebrated result: if $x(t)$ is the rightmost point $x$ such that $u(t,x)=\di\frac12$, then there is $x_\infty\in\RR$ such that
\begin{equation}
\label{bram}
x(t)=c^*t-\frac3{2\lambda^*}\mathrm{ln}~t+x_\infty+o_{t\to+\infty}(1).
\end{equation}
Here, $c^*=2\sqrt{f'(0)}$ is the lowest speed of a linear wave, and $\lambda^*$ the decay exponent at infinity of the bottom linear wave. A remarkable fact is that it is proved by probabilistic tools such as the Feynman-Kac formula. As there is a connection with the Branching Brownian Motion discovered by McKean in \cite{mck}, Bramson's results have led to important developments. A PDE proof was later provided by  Nolen, Ryzhik and the first author in \cite{nrr}. 

\subsubsection{Nonlocal one dimensional results}
The same type of connection as the McKean one between nonlocal equations of the type \eqref{kpp} -- albeit only with a certain type of polynomial nonlinearities -- exists with branching random walks, whose jumps are distributed according to the kernel $\ck$. This fact has been used in the probabilistic literature and has produced progress in the understanding of  \eqref{kpp}. We mention the work of A\"idekon \cite{Ai}, which gives a relation similar to that of \eqref{bram}  for the solutions of \eqref{kpp} in one space dimension. Another line of arguments, mixing probabilistic and PDE tools, is given in  Graham \cite{Gr}; it does not, however, capture the constant term in the expansion \eqref{bram}. A purely PDE proof of \eqref{bram}, still with $O(1)$ precision for $x(t)$, is proposed by Besse, Faye and the authors \cite{BFRZ}; this work may, to a certain extent, be viewed as a version of \eqref{kpp} with a kernel that is reduced to Dirac masses. The first author provided the full analogue of \eqref{bram}  in \cite{Roquejoffre2023} for the problem \eqref{kpp} in one space dimension, for a general nonlinearity satisfying the condition \eqref{f-cdn1}.
\subsubsection{Multidimensional results involving second order  elliptic operators}
For equations of the type \eqref{kpp1} with $N>1$, or, more generally,
\begin{equation}
\label{kpp2}
u_t-\Delta u=f(x,u), ~~~~t>0, ~~x\in\rn,
\end{equation}
where, for each $x\in\rn$, the function $s\in[0,1]\mapsto f(x,s)$ satisfies the conditions  \eqref{f-cdn1}, it is expected that, for any direction $e\in\sn$, there is a transition zone between $u\sim 1$ and $u\sim0$ travelling at asymptotically constant speed $w^*(e)$, and the issue is to locate this transition as precisely as possible. In the homogeneous case $f(x,s)\equiv f(s)$, the first work is due to
Aronson-Weinberger \cite{aw1}, who compute $w^*(e)=2\sqrt{f'(0)}$. G\"artner \cite{gartner} makes this estimate more precise and finds out in essence, by arguments of a similar line as those of \cite{bram1},  that  the set $\{u(t,x)\geq\di\frac12\}$  in every direction $e$ has an asymptotically spherical symmetric form:
\begin{equation}
\label{gart}
\{u(t,x)\geq\frac12\}\cap\RR e=[0,X(t)e]~~\hbox{with}~X(t)=c^*t-\frac{N+2}{2\lambda^*}\mathrm{ln}~t+O(1).
\end{equation}
This estimate is improved by Roussier-Michon, Rossi and the first author \cite{RRR}, who prove that the $O(1)$ term is asymptotic to a Lipschitz function $s_\infty(e)$ of $e$. An interesting problem  is to understand whether $s_\infty(e)$ is constant or not if $u_0$ is not spherically symmetric.  Ducrot \cite{Duc} presents extensions of Gärtner's result, that is \eqref{gart}, in settings where $f$ depends on $x$ but is asymptotically constant in $x$ with a rate of convergence. 

In the periodic setting, that is, in model of the type \eqref{kpp2}, the picture is different. The pioneering work is that of  Freidlin-G\"artner \cite{fg}, who prove
\begin{equation}
\label{fg}
w^*(e)=\min_{e\cdot e'>0}\frac{c^*(e')}{e\cdot e'},
\end{equation}
where $c_*(e')$ is the bottom speed of a planar linear wave of the linearised equation travelling in the direction $e'$. We will come back at length to this beautiful formula. It has given rise to alternative interpretations or proofs, for instance of the dynamical systems type (Weinberger \cite{weinberger}), or PDE type (Berestycki-Hamel \cite{bh1}, Berestycki-Hamel-Nadirashvili \cite{bhn1}, Guo-Hamel-Rossi \cite{GHR}, Rossi \cite{Rossi2017}), all bringing an additional element of understanding. When $f$ is not periodic in $x$ anymore, the concept of asymptotic speed of propagation may not be valid anymore; the works of Berestycki-Nadin \cite{BNad}, Rossi \cite{Rossi2017} explain by what concept it should be replaced, aided by  the concept of generalised travelling waves as elaborated, for instance, in \cite{bh2}.

For inhomogeneous equations, asymptotics of the level sets up to bounded terms -- not to speak about corrections up to vanishing terms -- are not as numerous. When $f$ is bistable, we quote a work by Matano-Mori-Nara \cite{MMN}, that studies  up to bounded terms for the level sets of the equation
$$
u_t-\mathrm{div}_p a(\nabla u)=f(u),
$$
where $f$ has an unstable intermediate zero between 0 and 1. While the flavour of the results that we are going to prove  might look similar to those of \cite{MMN}, both propagation mechanism and methods of proofs are quite different. The work that is closest to ours is  that of Shabani~\cite{Sha}, which concerns model \eqref{kpp2} in a periodic setting, and shows that in every direction $e$ the set $\{u(t,x)\geq\di\frac12\}$ has the following form:
\begin{equation}
\label{shab}
\{u(t,x)\geq\frac12\}\cap\RR e=[0,Z_e(t)e]~~\hbox{with}~Z_e(t)=w^*(e)t-\frac{N+2}{2\lambda^*(n_e)e\cdot n_e}\mathrm{ln}~t+O(1).
\end{equation}
Here, $w^*(e)$ is given by \eqref{fg} and $n_e$ denotes its unique minimiser, and the positive number $\lambda^*(n)$ is the decay exponent of a linear wave travelling in the direction $n$.

\subsubsection{Multidimensional results involving nonlocal operators}
As the kernel $\ck$ is not assumed to be spherically symmetric, it is expected that the propagation speed depends on the direction. The first work that identifies it, in the context of model \eqref{kpp},  is that of
Finkelshtein-Kondratiev-Tkachov \cite{FKT},  who prove the analogue of the Freidlin-G\"artner formula for the expansion of the level sets of the solution $u(t,x)$ of \eqref{kpp}; see an alternative proof  in \cite{Roquejoffre2023}. For a nonlocal equation of different spirit, the propagation velocity had already been identified by  Perthame-Souganidis \cite{PS}. A model of the type \eqref{kpp} on the lattice $\mathbb{Z}^d$ is studied by Hu \cite{Hu}, who proves results of the type \eqref{shab} in the directions $e$ that are parallel to the coordinate axes.
\subsection{Travelling waves}
A planar travelling wave connecting 1 and 0 is a solution of \eqref{kpp} which propagates in a given
unit direction $n\in\sn$ with a speed $c(n)$, and which can  be written as $u(t, x) = U_{c(n)}(x\cdot n-c(n)t)$, where  the wave profile $U_{c(n)}$ satisfies 
\begin{equation}
\label{TW}
	\begin{cases}
			U_{c(n)}-\ck*U_{c(n)}-c(n)U_{c(n)}'=f(U_{c(n)})~~~\text{in}~~\R,\\
		U_{c(n)}(-\infty)=1,~~U_{c(n)}(+\infty) =0,~~0<U_{c(n)}<1.
	\end{cases}
\end{equation}
We have the identities
 $$
 \ck*U_{c(n)}(z)=\int_{\rn}\ck(y)U_{c(n)}(z-y_1)\md y=\int_{\R}\ck_n(y)U_{c(n)}(z-y)\md y=\ck_n*U_{c(n)}(z),
 $$ 
 with 
 $$\ck_n(s):=\int_{n_\perp}\ck(sn+y_{n_\perp})\md y_{n_\perp},\  \hbox{and}\ \ n_{\!\perp}:=\{x\in\R^N|x\cdot n=0\}\simeq\mathbb{R}^{N-1}.
 $$ 
A planar travelling wave exists if and only if wave speed $c(n)\ge c^*(n)$, and it, if any, is unique up to translation. Moreover, its level sets are parallel hyperplanes which are orthogonal to the direction $n$, and the
solution is invariant in the moving frame with speed $c(n)$ in the direction $n$.  For later use, 
we normalise $U_{c^*(n)}(z)$ such that\footnote{Throughout this paper, we use the convention $f(t)\approx g(t)$ as $t\to+\infty$ if $\lim_{t\to+\infty}f(t)/g(t)=1$; $f(t)\asymp g(t)$ as $t\to+\infty$ if $C_1 g(t)\le f(t)\le C_2 g(t)$ as $t\to+\infty$ with some positive constants $C_1\le C_2$.}
\begin{equation}\label{normalization of mTW}
	U_{c^*(n)}(z)\approx z e^{-\lambda^*(n)z}~~~~\text{as}~~~z\to+\infty.
\end{equation}
All the above material is well known; we refer to Coville \cite{Cov} and the references therein for these questions. The asymptotic behaviour of the bottom wave is studied in Carr-Chmaj \cite{CC}, and explains why we can normalise the wave as in \eqref{normalization of mTW}. For all these questions, see also  Berestycki-Nirenberg \cite{BN} the latter, while only concerned with elliptic operators, deals with qualitative and asymptotic properties and contains arguments that can be used with great profit for nonlocal equations.  
\subsection{Results}
The main contribution of our paper is an asymptotic expansion up to vanishing terms of the level sets of the solution $u(t,x)$ of \eqref{kpp}.
\begin{theorem}\label{t1.1}
The set $\{u(t,x)\geq\di\frac12\}$  in every direction $e\in\sn$ has the following form:
\begin{equation}
\label{rz1}
\{u(t,x)\geq\frac12\}\cap\RR e=[0,Z_e(t)e]~~\hbox{with}~Z_e(t)=w^*(e)t-\frac{N+2}{2\lambda^*(n_e)e\cdot n_e}\mathrm{ln}~t+s_\infty(e)+o(1),
\end{equation}
where $w^*(e)$ is given by \eqref{fg}, and $\lambda^*(n)$ is the decay exponent of a linear wave travelling in the direction $n\in\sn$, and the vector $n_e$ is the unique minimiser in the Freidlin-G\"artner formula \eqref{fg}.  
\end{theorem}
Theorem \ref{t1.1} is an easy consequence of the following Theorem \ref{t1.2}, whose proof will be the main concern of the paper.
\begin{theorem}\label{t1.2}
 In every direction $e\in\sn$, we have the asymptotic expansion
\begin{equation}\label{rz2}
u(t,x)=U_{c_*(n_e)}\biggl(x\cdot n_e-w^*(e)t+\frac{N+2}{2\lambda^*(n_e) e\cdot n_e}\mathrm{ln}~t+s_\infty(e)\biggl)+o(1).
\end{equation}
\end{theorem}
While our work mainly concerns the nonlocal Fisher-KPP equation with a smooth kernel, we believe that extensions are possible.  In particular, it should be possible, with our line of arguments, to push~\eqref{shab} to $o(1)$ terms. The main obstacle here, also circumvented in \cite{Sha} by Nash-type estimates, is  to retrieve sufficient knowledge of a special Dirichlet heat kernel. In the present work we are dealing with $e^{-t\cI_*}$; to study it precisely everywhere is a technical challenge that is dealt with in Section 3. We also believe that our work can be extended so as to cover the case of a discrete lattice $\mathbb{Z}^d$, as well as the case where the kernel $\ck$ is a measure with no density with respect to the Lebesgue measure. One of the main modifications of our arguments that theses settings would involve is an alternative proof of the propagation of regularity for the solution of \eqref{kpp}.
\subsection{Organisation of the paper}
Our paper proposes several new ideas that have, to the best of our knowledge, no real equivalent in the works concerning the nonlocal version or parabolic version of \eqref{kpp}. Some of them are the control of the solutions of the linearised equation in the natural scale of the initial datum, the control of $u(t,x)$ at infinity, that is, away from the diffusive zone, and the design of new barrier functions. The paper comprises six sections, each of them devoted to an important part of the proof of Theorem \ref{t1.2}. Section~2 deals with the geometry of the spreading set, that is, the set enclosing the area $\{u(t,x)\sim\di\frac12\}$. The study of sharp asymptotics involves the computation of a heat kernel, that is performed in Section 3. This will {bear consequences} on the behaviour of $u(t,x)$ outside the diffusive area that will be detailed.  Barrier functions for $u(t,x)$ in the diffusive zone are constructed in Section 4, which will lead us to the identification of the correct space scale for $u(t,x)$ at large times, that is $x\sim\sqrt t$, as well as a good control of $u(t,x)$ outside the diffusive area. Section 5 refines the barriers and shows the convergence  of the tail of $u(t,x)$, that is, convergence in the diffusive zone. The convergence at the front is the final step of the proof, given in Section 6.

\bigskip

\noindent {\bf AI declaration.} No AI tools were used in this work.
\section{Spreading sets, spreading speeds}\label{TW}
\noindent  We start from the definition of the spreading set, which is, roughly speaking, the set beyond which the solution transitions from 1 to 0. In every direction $e$, its boundary is the set $\{w^*(e),\ \vert e\vert=1\}$, where $w^*(e)$ will be defined as the invasion speed in the direction $e$. The velocity $w^*(e)$ is given by the Freidlin-G\"artner formula; some deeper understanding of it will be needed. \subsection{The spreading set }
\noindent The first notion we need is that of spreading set. While it dates back to Weinberger \cite{weinberger}, the one that we will use is borrowed from Rossi \cite{Rossi2017}.
\begin{definition} 
	We say that a closed set $\mathcal{W}\subset\rn$, coinciding with the closure of its interior, is the {\it asymptotic set of spreading}  or {\it invasion set} for \eqref{kpp} if
	\begin{equation*}
		\forall ~\text{compact}~K\subset\text{int}(\mathcal{W}),~~~\inf_{x\in K}u(t,xt)\to 1~~~~\text{as}~t\to+\infty,
	\end{equation*}
\begin{equation*}
	\forall ~\text{closed}~K~\text{such that}~~K\cap\mathcal{W}=\emptyset,~~~\sup_{x\in K}u(t,xt)\to 0~~~~\text{as}~t\to+\infty.
\end{equation*}
\end{definition}

This definition essentially says  that the upper level sets of $u$ is roughly $t\mathcal{W}$ for $t$ large. Let us recall some standard material about it.
 From Finkelshtein-Kondratiev-Tkachov \cite{FKT} (see Perthame-Souganidis \cite{PS} for an earlier version involving  a model with nonlocal competition), the asymptotic set $\mathcal{W}$ of spreading for \eqref{kpp}  is given by
\begin{equation*}
	\mathcal{W}=\big\{re~|~e\in\sn, 0\le  r\le \ws(e)\big\}.
\end{equation*}
the velocity $w^*(e)$ being given by the Freidlin-G\"artner formula
\begin{equation*}
\label{FGformula}
w^*(e)=\min_{n\in\sn}\frac{c^*(n)}{n\cdot e}.
\end{equation*}
The velocity $c^*(n)$ is defined as follows:
we  look at the linearised equation
\begin{equation}\label{linearequation}
v_t=\ck* v-v+f'(0)v,~~~~~t>0,~x\in\rn.
\end{equation}
For any given direction $n\in\sn$, we apply the linear wave ansatz of the form $e^{-\lambda(x\cdot n-c(n)t)}$ for $(n,c,\lambda)\in\sn\times\R_+\times\R_+$, then  we arrive at the following dispersion relation
\begin{equation*}
	\label{dr}
D(n,c,\lambda):=-	\lambda c+\underbrace{\int_{\rn}\ck(x)e^{\lambda x\cdot n}\md x-1+f'(0)}_{=:\gamma(n,\lambda)}=0.
\end{equation*}
By defining 
\begin{equation*}\label{definition of c*} 
	\cs(n)=\inf\limits_{\lambda>0}\frac{\gamma(n,\lambda)}{\lambda},~~~\forall~n\in\sn,
\end{equation*}
there exists a unique $\lambda^*(n)>0$ such that the infimum above is achieved, namely $c^*(n)\lambda^*(n)=\gamma(n,\lambda^*(n))$, whence $D\big(n,\cs(n),\ls(n)\big)=0$. Furthermore, considering 
$$F(n,c,\lambda):=\partial_\lambda D(n,c,\lambda)=-c+\int_{\R^N}(x\cdot n)\ck(x)e^{\lambda x\cdot n}\md x
$$
 for $(n,c,\lambda)\in\sn\times\R_+\times\R_+$, we also have $F\big(n,\cs(n),\ls(n)\big)=0$. In addition, it is easily seen that $\partial^2_\lambda D\big(n,\cs(n),\ls(n)\big)>0$. 
\begin{proposition}
	\label{prop_smoothness in n}
The maps $n\in \mathbb{S}^{N-1}\mapsto c^*(n)$ and $n\in \mathbb{S}^{N-1}\mapsto \ls(n)$ are smooth.
\end{proposition}
\begin{proof}
For any fixed $n_0\in\mathbb{S}^{N-1}$, we notice that $(c^*(n_0),\lambda^*(n_0))$ satisfies the following constraints:
\begin{align*}
	D(n_0,c^*(n_0),\lambda^*(n_0))=0,~~~F(n_0,c^*(n_0),\lambda^*(n_0))=0.
\end{align*}
Moreover, 
\begin{equation*}
	\text{det}\begin{pmatrix} 
		\frac{\partial D}{\partial c}(c^*(n_0),\lambda^*(n_0)) & \frac{\partial D}{\partial \lambda}(c^*(n_0),\lambda^*(n_0)) \\[10pt]
		\frac{\partial F}{\partial c}(c^*(n_0),\lambda^*(n_0)) & \frac{\partial F}{\partial \lambda}(c^*(n_0),\lambda^*(n_0))
	\end{pmatrix}=\begin{vmatrix} 
		-\lambda^*(n_0) & 0 \\[10pt] 
		-1 & \int_{\R^N}(x\cdot n_0)^2\ck(x)e^{\lambda^*(n_0)x\cdot n_0}\md x
	\end{vmatrix}\neq 0.
\end{equation*}
The implicit function theorem along with the compactness of $\mathbb{S}^{N-1}$ then implies that for each $n\in\mathbb{S}^{N-1}$, there exist uniquely determined smooth functions $c=c^*(n)$ and $\lambda=\ls(n)$ such that $D(n,c^*(n),\lambda^*(n))=0$ and $F(n,c^*(n),\lambda^*(n))=0$.
\end{proof}

By the positivity of $\cs$, we can rewrite $\mathcal{W}$ as follows:
\begin{equation*}
	\mathcal{W}=\big\{x~|~x\cdot n\le \cs(n),~\forall n\in\sn\big\}=\displaystyle\bigcap_{n\in\sn}\big\{x~|~x\cdot n\le \cs(n)\big\}.
\end{equation*}
This implies that $\mathcal{W}$ is convex. 

\subsection{Uniqueness and Lipschitz regularity of the minimiser}

From now on, let us consider any given  $n\in\rn\backslash\{0\}$. We notice that the following ansatz
\begin{equation*}
\e{-\lambda(x\cd n-c(n)t)}=\e{-\lambda| n| \big(x\cd\frac{n}{| n| }-\frac{c(n)}{| n| }t\big)}
\end{equation*}
solves the linearised equation \eqref{linearequation}, and it gives 
\begin{equation}\label{def-gamma}
\lambda	c(n)=\underbrace{\int_{\rn}\ck(x)e^{\lambda x\cdot n}\md x-1+f'(0)}_{=:\gamma(n,\lambda)},
\end{equation}
and  the quantity
\begin{equation*}\label{c*} 
	\cs(n):=\inf\limits_{\lambda>0}\frac{\gamma(n,\lambda)}{\lambda}
\end{equation*}
is achieved at $\lambda^*(n)>0$.

Since 
\begin{equation*} 
\e{-\lambda\big(\frac{n}{| n|}\big)\big(x\cd\frac{n}{| n| }-c\big(\frac{n}{| n| }\big)t\big)}
\end{equation*}
also satisfies \eqref{linearequation}, this together with $\cs(n)=| n| \cs\left(\frac{n}{| n|}\right)$ and the uniqueness of the solution yields that 
\begin{equation*}
	\ls(n)| n|=\ls\left(\frac{n}{| n|}\right).
\end{equation*}
Moreover, for any given $ e\in\rn\backslash\{0\}$,
\begin{equation}\label{w*}
\ws\left(\frac{ e}{|  e| }\right)=\min\limits_{n\in\rn\atop n\cd e>0}\frac{\cs\big(\frac{n}{| n|}\big)}{\frac{n}{| n|}\cd\frac{ e}{| e|}}=| e|\min\limits_{n\in\rn\atop n\cd e>0}\frac{\cs(n)}{n\cd e}=| e| w^*( e).
\end{equation}

 In what follows, we would like to consider
\begin{equation*}
w^*( e)=\min\limits_{n\in\rn\atop n\cd e>0}\frac{\cs(n)}{n\cd e}, ~~~~~ e\in\rn\backslash\{0\}, 
\end{equation*}
 and discuss the ``uniqueness'' of the minimiser. Set $$~~~~~~~~~~~~~~~~~\Psi_e(n)=\frac{\cs(n)}{n\cdot e}~~~~~~~~\text{for}~~n\in\mathbb{H}_e:=\{n\in\rn | n\cdot e>0\}.$$
 Let us first point out that, suppose that $n_e\in\mathbb{H}_e$ is a minimiser for $ e$ given above, in the sense that 
$\ws( e)=\Psi_e(n_e)$,
it is easy to see that $kn_e$ with any $k>0$ is still a minimiser. Therefore, by ``uniqueness'' we actually aim to show that there exists at most one minimiser $n_e\in\{n\in\sn| n\cdot e>0\}$. This property was first noticed by Shabani \cite{Sha}. The proof that we present below is inspired by her remarks, but we add elements of understanding of our own. The Lipschitz property is, to our knowledge, new in this context.
\begin{theorem}\label{t2.1}
	For any given $ e\in\sn$,  there exists a unique minimiser $n_e\in \{n\in\sn| n\cdot e>0\}$ such that
	\begin{equation*}
		\ws( e)=\min\limits_{n\in\sn\atop n\cdot e>0} \frac{\cs(n)}{n\cdot e}=\frac{\cs(n_e)}{n_e\cdot e}.
	\end{equation*}
 The minimiser $n_e$ is exactly the unit outer normal of $\cW$ at $w^*(e)e$. Moreover, the map $e\in\mathbb{S}^{N-1}\mapsto n_e$ is Lipschitz, 
 and the invasion set $\mathcal{W}$ is $C^{1,1}$.
\end{theorem}

\begin{proof} Assume by contradiction that there exist two distinct minimisers $n_1,n_2\in \mathbb{H}_e$ with $| n_1|=| n_2|=1$ such that 
\begin{equation*}
	\ws( e)=\Psi_e(n_1)=\Psi_e(n_2), 
\end{equation*}
which gives 
\begin{equation*}\label{gamma-ei}
\cs(n_i)=\ws( e) e\cdot n_i, ~~~~i=1,2.
\end{equation*}
Moreover, we have  $$\nabla\Psi_e(n_i)=\nabla\left(\frac{\cs(n)}{n\cdot e}\right)|_{n=n_i}=\frac{1}{(n_i\cdot e)^2}\big(\nabla\cs(n_i)(n_i\cdot e)-\cs(n_i) e\big)=0,$$ namely, 
\begin{equation}\label{c*-ei}
	\nabla \cs(n_i)=\frac{\cs(n_i)}{n_i\cdot  e} e=\ws( e) e,~~~~i=1,2.
\end{equation}
On the other hand, it follows from $\cs(n)=\frac{\gamma(n,\ls(n))}{\ls(n)}$ for $n\in\mathbb{H}_e$ that $\nabla\cs(n)=\frac{\nabla\gamma(n,\lambda)|_{\lambda=\ls(n)}}{\ls(n)}$, whence
\begin{equation*}
	\nabla\cs(n_i)=\frac{\nabla\gamma(n_i,\lambda)|_{\lambda=\ls(n_i)}}{\ls(n_i)},~~~~i=1,2.
\end{equation*}
This together with \eqref{c*-ei} leads to  $\nabla\gamma(n_i,\lambda)|_{\lambda=\ls(n_i)}=\ls(n_i)\ws( e) e.$  Moreover, from the definition \eqref{def-gamma} of $\gamma$, we find that
\begin{equation*}
	\nabla\gamma(n_i,\lambda)|_{\lambda=\ls(n_i)}=\ls(n_i)\int_{\R^N} x\ck(x)e^{\ls(n_i) x\cdot n_i}\md x.
\end{equation*}
Altogether, we get 
\begin{equation}\label{parallel to xi}
	\int_{\R^N} x\ck(x)e^{\ls(n_i) n_i\cdot x}\md x=\ws( e) e.
\end{equation}

Define now 
\begin{equation*}
	\Phi(p):=\int_{\R^N} \ck(x)e^{ p\cdot x}\md x-1+f'(0),~~~~p\in\mathbb{H}_e.
\end{equation*}
It is obvious that $p\in\mathbb{H}_e\mapsto \Phi(p)$ is strictly convex and analytic, and
\begin{equation*}
	\nabla\Phi(p)=\int_{\R^N} x\ck(x)e^{ p\cdot x}\md x.
\end{equation*}
We derive from \eqref{parallel to xi} that there exist $p_i:=\ls(n_i)n_i\in\mathbb{H}_e$ with $i=1,2$ such that $p_1\neq p_2$, and 
\begin{equation}\label{contra}
\nabla\Phi(p_1)=\nabla\Phi(p_2)=\ws( e) e.
\end{equation}

Set $\varphi(t):=\Phi\big(p_1+t(p_2-p_1)\big)$ for $t\in[0,1]$, then  by the strict convexity and analyticity of $\Phi$ in $p\in\mathbb{H}_e$,  it follows that $\varphi(t)$ is strictly convex and smooth in $t\in[0,1]$,  
\begin{equation*}
	\varphi'(t)=\nabla\Phi(p_1+t(p_2-p_1)\big)\cdot(p_2-p_1),
\end{equation*}
and $\varphi'(t)$ is strictly increasing in $t\in[0,1]$.
This along with \eqref{contra} indicates that 
\begin{equation*}
	0<\varphi'(1)-\varphi'(0)=\ws( e) e\cdot(p_2-p_1)-\ws( e) e\cdot(p_2-p_1)=0,
\end{equation*}
which is a contradiction. This, together with the existence of the minimiser \cite{Roquejoffre2023}, implies that, for any given $ e\in\rn\backslash\{0\}$, there exists a unique minimiser $n_e\in \{n\in\sn| n\cdot e>0\}$  such that 
$\ws( e)=\frac{\cs(n_e)}{n_e\cdot e}$.  

Let now $n_e\in \{n\in\sn| n\cdot e>0\}$ be  the minimiser for $\ws( e)$ with a given direction $ e\in\sn$, then
\begin{equation*}
	\forall  e'\in\sn,~~~\ws( e')=\min\limits_{n\in\sn\atop n\cdot e'>0} \frac{\cs(n)}{n\cdot e'}\le \frac{\cs(n_e)}{n_e\cdot e'},
\end{equation*}
namely, $\ws( e') e'\cdot n_e\le \cs(n_e)=\ws( e) e\cdot n_e$. Geometrically, it amounts to saying the “highest” hyperplane orthogonal to $n_e$ that
touches $\mathcal{W}$ is $\{x\in\rn | x\cdot n_e= \ws( e) e\cdot n_e\}$, also known as the supporting hyperplane for $\mathcal{W}$ containing $\ws( e) e\in\partial\mathcal{W}$. Thus, $n_e$ is the outer normal to $\mathcal{W}$ at the point $w( e) e$.

The Lipschitz character of $e\in\mathbb{S}^{N-1}\mapsto n_e$ can be proved along the same lines by noticing that $e\mapsto w^*(e)$ is Lipschitz as the infimum of Lipschitz functions with a uniform Lipschitz constant, for which one just needs to be careful to avoid $n\cdot e=0$. 
To this end, we first notice from \eqref{w*} that 
$$\ws\Big(\frac{ e}{|  e|}\Big)=| e|\ws( e)=| e|\frac{\cs(n_e)}{n_e\cdot e}=\frac{\cs(n_e)}{n_e\cdot\frac{ e}{|  e|}}.$$  Together with the convexity of the invasion set $\mathcal{W}$, we get that  $\mathcal{W}$ is at least $C^{1}$. Now, we start from \eqref{contra} but we take this time
\begin{equation*}
	p_1=\lambda^*(n_e)n_e, ~~~~p_2=\lambda^*(n_{e_h})n_{e_h},~~~\text{with}~~e_h:=\di\frac{e+h}{\vert e+h\vert},~~\vert h\vert\leq\di\frac{1}{2}.
\end{equation*} 
By the strict convexity of $\Phi$, we have
$$
\lambda^*(n_{e_h})n_{e_h}-\lambda^*(n_e)n_e =O(h).
$$
Projecting this identity on $n_e$ and $n_e^\perp$ yields $\vert n_{e_h}-n_e\vert=O(h)$.
 \end{proof}

\begin{remark} \label{r2.1} We do not know if $\mathcal{W}$ is more than $C^{1,1}$, however fortunately we will not need more than that in the sequel. This is a typical feature of Hamilton-Jacobi equations: wherever the solution is $C^1$, it is in fact $C^{1,1}$, but further regularity is not granted, see for instance Fathi \cite{Fa}.
\end{remark}

\subsection{Regularity estimates  of $u$}
As no regularising effect exists in our class of equations, we need to work with barehands, and the main result of this section is the following gradient estimate.
\begin{theorem}
\label{t2.20}
There holds
\begin{equation*}
	\sup_{s>1}\Vert\nabla u(s,\cdot)\Vert_{L^\infty(\R^N)}<+\infty.
\end{equation*}
\end{theorem}
Theorem \ref{t2.20} has an immediate consequence for the higher derivatives. Let us take this opportunity to set some standard notations:  if $\alpha=(\alpha_1,\alpha_2,\cdots,\alpha_N)$ is a multi-index, the length of $\alpha$ is $\vert\alpha\vert=\alpha_1+\ldots+\alpha_N$; we also denote  $\alpha!=\alpha_1!\alpha_2!\cdots\alpha_N!$. If  $\vert\alpha\vert=m$ and $v$ is a $C^m$ function from $\RR^N$ to $\RR$, we set $\partial^\alpha v=\partial_{x_1}^{\alpha_1}\ldots\partial_{x_N}^{\alpha_N}$. If $x=(x_1,x_2,\cdots,x_N)$ ,
$x^\alpha=x_1^{\alpha_1} x_2^{\alpha_2}\cdots x_N^{\alpha_N}$ and $\beta^\alpha=\beta_1^{\alpha_1} \beta_2^{\alpha_2}\cdots \beta_N^{\alpha_N}$. Finally, if $x\in\RR^N$ and $R>0$, the Euclidian ball with centre $x$ and radius $R$ will be denoted by $B_R(x)$.

\begin{corollary}\label{cor_2.20}
For all integer $m$, and for all mult-index $\alpha$ of length $m$ we have 
$$
\sup_{s>1}\Vert \partial^\alpha u(s,\cdot)\Vert_{L^\infty(\R^N)}<+\infty.
$$
\end{corollary}
In order to make clear why Theorem \ref{t2.20} holds, let us give here a remarkably short proof in the case $f'(0)<1$. To this end, we look at the equation \eqref{kpp} satisfied by $u$. By setting $u_i(t,x)=\partial_{x_i}u(t,x)$, the equation for $u_i$ reads
$$
\partial_tu_i+u_i-f'(u)u_i=\partial_{x_i}\mathcal{K}*u,~~~~t>0,~x\in\R^N,
$$
The Duhamel formula entails
\begin{equation}
\label{e3.330}
u_i(t,x)=e^{-t+\int_0^tf'[u(s,x)]\md s}\partial_{x_i}u_0(x)+\int_0^te^{-\bigl[(t-s)+\int_s^tf'[u(\sigma,x)]\md\sigma\bigl]} \partial_{x_i}\mathcal{K}*u(s,.)\md s,~~~~t>0,~x\in\R^N,
\end{equation}
so that we have, thanks to the inequality $1-f'(u)\ge 1-f'(0)$:
$$
\vert u_i(t,x)\vert\leq e^{-(1-f'(0))t}\Vert\partial_{x_i}u_0\Vert_\infty+\int_0^te^{-(1-f'(0))(t-s)}\Vert \partial_{x_i}\mathcal{K}\Vert_{L^1(\R^N)} ds,~~~~t>0,~x\in\R^N,
$$
in which the right-hand side is uniformly bounded in $t$. 
This proves the result.

This is not the first time that this observation is made, see for instance  Shen-Shen \cite{ShSh} where  the assumption that the slope at the origin of the reaction term is $<1$ is used to establish an inconditional regularity result for transition fronts.  The  general  idea  of the proof of Theorem \ref{t2.20} relies, however, on the notions displayed in this section, as well as on 
 Corollary \ref{cor_6.2}.   We point out that this corollary only relies on the results of Section \ref{sec4} as well as Proposition \ref{prop-5.3}, that do not use any gradient estimate. 

In the sequel, we will first seize the opportunity to reformulate the problem in a way that will be useful in all the forthcoming sections. We will then draw some simple consequences of the outcomes of Sections \ref{sec4} and \ref{sec6}, which, let us repeat it here, do not need a gradient estimate. This will allow us to analyse the Duhamel formula \eqref{e3.330} in a  way that makes a more substantial use of  the structure of the solution  than in the case $f'(0)<1$, and to conclude the proof.
\subsubsection{Reformulation of Problem \eqref{kpp} in a direction of propagation} \label{sec2.3}
An important relation found in the proof of Theorem \ref{t2.1} is \eqref{parallel to xi}. For any $e\in\sn$, let  $n_e$ be the unique minimiser in the Freidlin-G\"artner formulagiven by Theorem \ref{t2.1}, associated with $e$. Then \eqref{parallel to xi} can be recast as
\begin{equation}
\label{e2.10}
~~~~~~	w^*(e)e=\int_{\R^N}x\ck_{*,n_e}(x)\md x\in\R^N,~~~~ \hbox{with}~~ \ck_{*,n_e}(x)=e^{\lambda^*(n_e)x\cdot n_e}\ck(x).
\end{equation}
Theorem \ref{t2.1} implies that
\begin{equation*}
	c^*(n_e)=\int_{\R^N}(x\cdot n_e)\ck_{*,n_e}(x)\md x,
\end{equation*}
and that the front propagation occurs along the unit outer normal $n_e$  at the boundary point $tw^*(e)e\in t\partial\mathcal{W}$ with speed $c^*(n_e)$. As $w^*(e)=\di\frac{c^*(n_e)}{n_e \cdot e}$, we denote $w^*(e)e=c^*(n_e)n_e+\bm_{n_e}^\perp$, then we have 
\begin{equation*}
	\vert\bm_{n_e}^\perp\vert =c^*(n_e)\sqrt{\frac1{(n \cdot e)^2}-1},~~~\bm_{n_e}^\perp=\int_{\R^N}\ck_{*,n_e}(x) \bigl(x-(x\cdot n_e)n_e\bigl)d x.
\end{equation*}
For commodity we now set, for all $y\in\RR^N$:  $y_e^\bot =y-(y\cdot n_e)n_e$ and, whenever this is convenient to us, we will make the abuse of notation that consists in setting $y=(y\cdot n_e,y_e^\bot)$.

 \subsubsection{A consequence of Sections \ref{sec4} and \ref{sec6}}
Let us remember that the kernel $\ck$ is supported in the ball $\{\vert x\vert\leq R\}$. Consider a point $x$, fixed until the end of the proof, with, to fix ideas,  $\vert x\vert\geq R$ and set $e=\displaystyle\frac{x}{\vert x\vert}$.  Let us introduce $\theta_-\in(\displaystyle\frac12,1)$ and $\theta_+>1$, close to 1, such that such that 
$$f'(u)\leq\displaystyle\frac{f'(1)}2~~~\hbox{if $u\in[\theta_-,\theta_+]$}.
$$
 Consider two nonlinearities $f_-\leq f_+$ that coincide with $f$ on $[0,\displaystyle\frac12]$ and such that $f_-$ (resp. $f_+)$ is of the Fisher-KPP type on $[0,\theta_-]$ (resp. on $[0,\theta_+]$. Consider the travelling waves $U^\pm_{c^*(n_e)}$  solving \eqref{TW} with the respective nonlinearities $f_-$ and $f_+$, both normalised so that we have $U_{c^*(n_e)}(0)=\displaystyle\frac12$. Let us finally set 
 $$
 X_e(t)=c^*(n_e)t-\frac{N+2}{2\lambda^*(n_e)}\mathrm{ln}~t;
 $$
 notice that the front position $Z_e(t)$ in the direction $e$ at time $t$ introduced in Theorem \ref{t1.1} is given by $Z_e(t)=\di\frac{X_e(t)}{e.n_e}$.
 The consequence of Sections \ref{sec4} and \ref{sec6} that we want to draw is the following proposition.
\begin{proposition}
\label{p2.100}
There are $t_0>0$, independent of  $x$ and $y$, and two constants $z_+>z_-$ such that  
$$
\begin{array}{rll}
\forall t\geq t_0,\forall y\in\RR^N~\hbox{with $\vert y_e^\bot\vert\leq2R$,}~~~ &u(t,y)\leq U_{c^*(n_e)}^+\bigl(y\cdot n_e-X_e(t)-z_+\bigl);\\
\forall t\geq t_0,\forall y\in\{\vert y_e^\bot\vert\leq 2R\}\cap\{y\cdot n_e\leq X_e(t)+\eta_1\sqrt t\},~~~ &u(t,y)\geq U_{c^*(n_e)}^-\bigl(y\cdot n_e-X_e(t)-z_-\bigl).
\end{array}
$$
\end{proposition}
The constant $\eta_1$ is defined at the beginning of Section \ref{sec4} below. 
\begin{proof}
Let us start with the  bound from above. On the one hand, the inequality \eqref{conclusion-super+sub} of Section \ref{sec4},  translated in the langage of this section, entails the inequality
$$
u(t,y)\lesssim e^{-\lambda^*(n_e)[y\cdot n_e-X_e(t)]}\biggl(y\cdot n_e-X_e(t)+O\bigl(t^\vartheta e^{-\frac{A[y\cdot n_e-X_e(t)]}{\sqrt t}}\chi_1[\frac{y\cdot n_e-X_e(t)]}{\sqrt t}\bigl)\biggl),
$$
for $t^\delta\leq y\cdot n_e-X_e(t)$ and for all $A>0$. The constants $\delta$ and $\vartheta$ defined in \eqref{parameters-n} of Section \ref{sec4}. This entails, the existence of $\kappa_+>0$ such that,
for $y\cdot n_2=t^\delta$, we have
$$
u(t,y\cdot n_e,0)\le	\kappa_+\bigl(y\cdot n_e-X_e(t)\bigl) t^{-\frac{N}{2}-1}.
$$
Corollary \ref{cor_6.2}, (i), then implies
$$\limsup_{t\to+\infty}\Big(u(t,y\cdot n_e,y_e^\bot)- U_{c^*(n_e)}\Big(y\cdot n_e-X_e(t)-z_+\Big)\Big)\le 0
$$ 
uniformly in $\{-\rho\le y\cdot n_e-X_e(t)\le t^\delta\}$ and $\{\vert y_e^\bot\vert\leq2R\}$, with $z_+=\di\frac{1}{\lambda^*}\ln\kappa_+$, and $\rho\geq 2R$. As $u(t,y)\leq 1$ and $U_{c^*}(z)>\di\frac{1+\theta_+}2>1$ if $z$ is sufficiently negative, the bound from above easily follows as soon as $t$ is large enough.

The bound from below needs a slightly longer development. Let us first pick any $\varepsilon_0\in(0,\di\frac14)$, $r_0\leq\min(\di\frac14,\di\frac{R}2)$,  and let us define $\underline u_0$ as a function which is equal to $\varepsilon_0$ in $B_{\frac{r_0}2}(0)$, 0 outside $B_{r_0}(0)$ and which is smooth and nonnegative. Let us introduce $T_0>0$  such that the solution $\underline u(t,y)$ of \eqref{kpp} starting from $\underline u_0$ satisfies
$$
\forall t\geq T_0,\forall \tilde y\in B_{2R}(0),~~~ \underline u(t,\tilde y)\geq\frac{1+\theta_-}2.
$$
The inequality \eqref{conclusion-super+sub} of Section \ref{sec4},  translated in the langage of this section, entails the inequality
$$
u(t,y)\gtrsim e^{-\lambda^*(n_e)[y\cdot n_e-X_e(t)]}\biggl(y\cdot n_e-X_e(t)+O\bigl(t^\vartheta\chi_1[\frac{y\cdot n_e-X_e(t)}{\sqrt t}]\bigl)\biggl),
$$
for $t$ large enough, $t^\delta\leq y\cdot n_e-X_e(t)\lesssim\sqrt t $, and $\vert y_e^\bot\vert\leq 2R$. And so, for $y\cdot n_e-X_e(t)=t^\delta$ there holds, for some $\kappa_->0$:
$$
u(t,y\cdot n_e,y_e^\bot)\ge	\kappa_-\bigl(y\cdot n_e-X_e(t)\bigl) t^{-\frac{N}{2}-1}=\kappa_-t^{-\frac{N}{2}-1+\delta}.
$$  Corollary \ref{cor_6.2}, (ii) then implies, for all $\rho\geq 2R$,
 $$\liminf_{t\to+\infty}\Big(u(t,y\cdot n_e,y_e^\bot)- U_{c^*(n_e)}\Big(y\cdot n_e+X_e(t)-z_-\Big)\Big)\ge 0
 $$  in $\{-\rho\le y\cdot n_e-X_e(t)\le t^\delta\}$ and $\{\vert y_e^\bot\vert\leq2R\}$, with $z_-=\di\frac{1}{\lambda^*}\ln\kappa_-$. Let us pick $\varepsilon_1\in (\varepsilon_0,\di\frac12)$  such that 
 $$
 \forall z\leq z_1+r_0,~~~U_{c^*(n_e)}(z)\geq\varepsilon_0,
 $$
 with $z_1=U_{c^*(n_e)}^{-1}(\varepsilon_1)$. Finally, we choose 
 $$
 \rho=1+\max\bigl(2R,c^*(n_e)T_0+\vert z_-\vert+\vert z_1\vert+1\bigl).
 $$
For $y\in\{\vert y_e^\bot\vert\leq 2R\}$, let  $t_y$ be the first $t>0$ such that $y\cdot n_e-X_e(t)-z_-=z_1$; we may assume $t_y\geq t_0$ and notice that the choice of $\rho$ implies
$$
\forall t\in[t_y,t_y+T_0], ~~~u(t,y)\geq \sup\biggl(\underline u(t-t_y,y),U_{c^*(n_e)}\bigl(y\cdot n_e+X_e(t)-z_-\bigl) \biggl).
$$
This implies in turn  
$$\forall t\geq t_y+T_0,~~~u(t,y)\geq \underline u(t-t_y,y)\geq\di\frac{1+\theta_-}2.
$$
As $U_{c^*(n_e)}^-\leq U_{c^*(n_e)}$ and $\di\lim_{z\to-\infty}U_{c^*(n_e)}^-(z)<\di\lim_{z\to-\infty}U_{c^*(n_e)}(z)$, this entails the desired lower bound for $u(t,y)$, regardless of how negative $y\cdot n_e-X_e(t)$ is.
\end{proof}
\subsubsection{Bounds for the derivatives of $u$}
Let us come back to our reference point $x$ with $e=\di\frac{x}{\vert x\vert}$. 

\begin{proof}[Proof of Theorem \ref{t2.20}] Let us first notice that $t\mapsto \vert\nabla u(t,x)\vert$ has an easy $\vert x\vert$-dependent bound. Indeed, still for our reference point $x$, there is $t_{\vert x\vert}>0$ such that $u(t,y)\geq\theta^-$ on $B_{2\vert x\vert}(0)$ as soon as $t\geq t_{\vert x\vert}$.  Let us examine the integral between 0 and $t$ in the Duhamel formula \eqref{e3.330}.  If $t\leq t_{\vert x\vert}$ we have
$$t-s-\di\int_s^tf[u(\sigma,x)]\md\sigma\geq\bigl(1-f'(0)\bigl)t_{\vert x\vert},
$$
 while, if $t>t_{\vert x\vert}$, we have, for $s<t_{\vert x\vert}$:
$$
t-s-\di\int_s^tf'[u(\sigma,x)]\md\sigma=\di t-s-\biggl(\int_s^{t_{\vert x\vert}}+\int_{t_{\vert x\vert}}^t\biggl)f'[u(\sigma,x)]\md\sigma\\
\geq\bigl(1-\di\frac{f'(1)}2\bigl)(t-t_{\vert x\vert})-f'(0)t_{\vert x\vert};
$$
so that we have
$$
\int_0^te^{-(t-s)+\int_s^t[f'[u(\sigma,x)]\md\sigma}\vert\partial_{x_i}\ck_*\vert*u(s,x)\md s\leq t_{\vert x\vert}\Vert\partial_{x_i}\ck_*\Vert_\infty e^{f'(0)t_{\vert x\vert}}.
$$
The term $e^{-t+\int_0^t[f'[u(\sigma,x)]\md\sigma}\vert\partial_{x_i}u_0\vert$ is bounded in the same fashion, which implies the $\vert x\vert$-dependent bound. 

Let us now prove a gradient bound that is uniform in $x$. Let $t_0$ be such that Proposition \ref{p2.100} is valid, and $t_1\geq t_0$ be such that $\dot X_e(t)>0$ for all $t\geq t_1$. Pick $\lambda<\lambda^*(n_e)$, close to $\lambda^*(n_e)$, such that
$$
1+\lambda^*(n_e)c^*(n_e)-f'(0)\geq\frac12\int_{\RR^N}e^{\lambda^*(n_e)y\cdot n_e}\ck(y\cdot n_e,y_e^\bot)\md y.
$$
Let us assume 
$$
x\cdot n_e\geq X_e(t_1)+z_-,
$$
a bound for $\vert\nabla u(t,x)\vert$ being granted if the opposite is true. Let $\tau_x>t_1$ be such that 
$$
x\cdot n_e=X_e(\tau_x)+z_-;
$$
the discussion is organised around two cases in which it is once again enough to examine the  the integral between 0 and $t$ in the Duhamel formula.

\noindent {\bf Case 1. $t\leq\tau_x$.} We use the bound $u(s,y)\leq Ce^{-\lambda[x-X_e(s)]}$, for $s\leq t$ and $\vert y_e^\bot\vert\leq R$, so that the integrand in the Duhamel formula \eqref{e3.330} satisfies, still for $s\leq t$:
$$
\begin{array}{rll}
\di e^{-(t-s)+\int_s^t[f'[u(\sigma,x)]\md\sigma}\vert\partial_{x_i}\ck_*\vert*u(s,x)\leq&C\di e^{-(t-s)+\int_s^t[f'[u(\sigma,x)]\md\sigma-\lambda[x-X_e(s)]}\\
=&C\di e^{-(t-s)+\int_s^t[f'[u(\sigma,x)]\md\sigma-\lambda[X_e(t)-X_e(s)]}e^{-\lambda[x-X_e(t)]}\\
\leq&C\di e^{-(t-s)[f'(0)+1+c^*(n_e)\lambda]}e^{-\lambda[x-X_e(t)]}\\
\leq&C\di e^{-\frac{(t-s)}2\int_{\RR^N}e^{\lambda^*(n_e)y\cdot n_e}\ck(y\cdot n_e,y_e^\bot)\md y}e^{-\lambda[x-X_e(t)]},
\end{array}
$$
the constant $C>0$ being universal. As a consequence we have $\vert\partial_{x_i}u(t,x)\vert\leq Ce^{-\lambda[x-X_e(t)]}$.

\noindent \noindent {\bf Case 2. $t\geq\tau_x$.}  We   write $\di \int_0^t=\int_0^{\tau_x}+\int_{\tau_x}^t$; the first integral is bounded as in Case 1, while the second is bounded using that we now have $f'[u(\sigma,x)]\leq\di\frac{f'(1)}2$ if $\sigma\geq s\geq \tau_x$.   This concludes the proof. \end{proof}

\begin{remark}
\label{r2.20}
In fact we have proved more than what is claimed by Theorem \ref{t2.20}: we have $\vert\nabla u(t,x)\vert\lesssim \inf\bigl(1,e^{-\lambda(x\cdot n_e-X_e(t)}\bigl)$, for all $\lambda<\lambda^*(n_e)$. A more precise estimate will be proved in Section \ref{sec5}. 
\end{remark}
\begin{proof}[Proof of Corollary \ref{cor_2.20}] Consider any multi-index $\alpha$ of length $m$, and pick any $\lambda<\lambda^*(n_e)$. The proof runs by induction on $m$, the induction assumption being  $\vert\partial^\alpha u(t,x)\vert\lesssim \inf\bigl(1,e^{-\lambda(x\cdot n_e-X_e(t)}\bigl)$. The equation for $u_\alpha:=\partial^\alpha u$ has the following structure:
\begin{equation}
\label{e6.211}
\partial_tu_{\alpha}+\bigl(1-f'(u)\bigl)u_{\alpha}=\partial^\alpha\ck*u+\sum_{\substack{2\leq k\leq m\\
1\leq\vert\gamma\vert,\vert\gamma'\vert\leq m-1}}a_{k,\gamma,\gamma'}f^{(k)}(u)u_{\gamma_1}^{\gamma_1'}\ldots u_{\gamma_N}^{\gamma_N'},
\end{equation}
the constants $a_{k,\gamma,\gamma'}$ having no remarkable value and the possibility to be zero. Using the induction assumption for all multi-indices of length $<m$ and arguing as in the proof of Theorem \ref{t2.20} easily yields the corollary. \end{proof}

\subsection{The operator $\mathcal{I}_*$ and its first properties}
We transfer \eqref{kpp} to the reference frame moving at speed $w^*(e)$ in the direction $e$ by setting $\bu (t,x)=u(t,x+tw^*(e)e)$,  then we get    
\begin{equation}
	\label{bu}
	\bu_t+\bu-\ck*_{*,n_e}\bu\underbrace{-c^*(n_e)\partial_{n_e}\bu-\bm_{n_e}^\perp\cdot\nabla \bu}_{=-w^*(e)\partial_{e}\bu}=f(\bu).
\end{equation}
The leading edge transformation 
$v(t,x)=e^{\lambda^*(n_e)x\cdot n_e}\bu(t,x)$ yields the following equation for $v$:
\begin{equation}
	\label{e5.101}
	v_t+\cI_{*,n_e}v
	+R(t,x;v)=0,~~~t>0,~x\in\R^N,
\end{equation}
where $\cI_{*,n_e}$ is the linear first order integro-differential operator given by\footnote{Throughout this paper, We use the following definition of the Fourier transform:
	\begin{equation*}
		\widehat{g}(\xi)\triangleq\mathcal{F}[g](\xi):=\int_{\R^N}e^{-\mathbf{i}x\cdot\xi}g(x)\md x,~~~~~~~\mathcal{F}^{-1}[g](x):=\frac{1}{(2\pi)^N}\int_{\R^N}e^{\bi x\cdot\xi}g(\xi)\md\xi.
\end{equation*} }
\begin{equation} 
	\label{e5.100}
	\cI_{*,n_e}v
	=\widehat{\ck_{*,n_e}}(0)v-\ck_{*,n_e}*v-w^*(e)\partial_{e}v
	=\widehat{\ck_{*,n_e}}(0)v-\ck_{*,n_e}*v-c^*(n_e)\partial_{n_e}v-\bm_{n_e}^\perp\cdot\nabla v,
\end{equation}
and
$$R(t,x;v)=f'(0)v-e^{\lambda^*(n_e)x\cdot n_e}f\big( e^{-\lambda^*(n_e)x\cdot n_e}v
			\big).
$$
 Obviously, the nonlinear term $R(t,x;s)$ is always nonnegative for $t>0$, $x\in\R^N$ and $s\in\R$, and it belongs to $O(s^2)$ for $s$ small.
 
 Let us notice right away that $\cI_*$ has elliptic features. Some of the following computations will be especially useful when $\cI_{*,n_e}$ acts on a function whose derivatives become increasingly small, as is the case, for instance, for functions with slow spatial variations. 

For a smooth  function $v(x)$, we have
$$
\begin{array}{rll}
\ck_{*,n_e}*v(x)=\displaystyle\int_{\R^N}\ck_{*,n_e}(y)v(x-y)\md y=&\widehat{\ck_{*,n_e}}(0)v(x)-\nabla v(x)\cdot\displaystyle\int_{\R^N}\ck_{*,n_e}(y)y\md y\\
&\di+\frac12\int_{\R^N}\ck_{*,n_e}(y)y^TD^2v(x)y \md y+O(\Vert D^3v\Vert_{L^\infty(B_R(x)}).
\end{array}
$$
This together with \eqref{e5.100} and \eqref{e2.10} gives that  
$$\cI_{*,n_e}v(x)=-\displaystyle\frac12\displaystyle\int_{\R^N}\ck_{*,n_e}(y)y^TD^2v(x)y\md y+O(\Vert D^3v\Vert_{L^\infty(B_R(x)}).
$$   
Actually, we have
\begin{equation}
	\label{e5.103}
	\displaystyle\int_{\R^N}\ck_{*,n_e}(y)y^TD^2v(x)y\md y=\mathrm{div}\bigl(H_{n_e}\nabla v\bigl)(x)=\sum_{i,j=1}^{N}h_{ij}\partial_{x_i x_j}v(x),
\end{equation}
with $H_{n_e}$ given by
\begin{equation}\label{def_H matrix}
H_{n_e}=\left(\int_{\R^N}x_ix_j \ck_{*,n_e}(x)\md x\right)_{1\le i,j\le N}=(h_{ij,n_e})_{1\leq i,j\leq N}.
\end{equation}
 Obviously, $H_{n_e}$ is symmetric. It is also positive definite: for any $\xi=(\xi_i)_{1\le i\le N}\in\R^N\backslash\{0\}$, we have
 \begin{equation}
 	\label{H_quadratic form}
 	\xi^TH_{n_e}\xi=\int_{\R^N}\ck_{*,n_e}(x)(x\cdot\xi)^2\md x.
 \end{equation}
Choose for instance $\Omega:=\big\{x\in\RR^N| \ck_{*,n_e}(x)\geq\displaystyle\frac12\sup\ck_{*,n_e}\big\}$ which has a positive measure, and does not depend on the particular direction $e$. This, together with the fact that the set $\{x\in\R^N|x\cdot\xi=0\}$ has zero measure, indicates the existence of $\mu>0$ such that
\begin{equation}
	\label{H_lower bound}
	\xi^TH_{n_e}\xi\ge \int_{\Omega}\ck_{*,n_e}(x)(x\cdot\xi)^2\md x>\mu\vert\xi\vert^2.
\end{equation}
 This amounts to saying that the minimal eigenvalue  of $H_{n_e}$ is positive, and that the operator $-\mathrm{div}(H_{n_e}\nabla)$ is uniformly elliptic. The formulation \eqref{e5.103} will appear many times in the sequel.  
 
    Another simple, and nonetheless useful fact, is a Kato-type inequality:
  \begin{proposition}
  \label{p3.200}
  If $\varphi$ is a convex function of a real variable, we have $\varphi'(u)\cI_{*,n_e}u\geq\cI_{*,n_e}(\varphi(u))$ for all functions $u\in C^1(\R^N)$.
  \end{proposition}
  \begin{proof}
We have, by the convexity inequality:
  $$
  \int_{R^N}\mathcal{K}_{*,n_e}(x-y)\bigl[\varphi\bigl(u(x)\bigl)-\varphi\bigl(u(y)\bigl)\bigl]\bigl)\md y\leq \varphi'\bigl(u(x)\bigl)\int_{\R^N}\mathcal{K}_{*,n_e}(x-y)\bigl(u(x)-u(y)\bigl)\md y.
  $$
Moreover, since $\varphi'(u)\nabla u=\nabla\varphi(u)$, the result is proved.
  \end{proof}

\section{The fundamental solution ahead of the front}
Let $e$ be a direction of the unit sphere and $n_e$ be given by Theorem \ref{t2.1}. The objective of this section is to study the linear problem
  \begin{equation}
  	\label{w-eqn-linear}
  	w_t+\cI_{*,n_e}w=0,~~~t>0,~x\in\R^N.
  \end{equation}
When endowed with a smooth initial datum $w_0(x)$ the solution of \eqref{w-eqn-linear} will be denoted by $\bigl(e^{-t\cI_*}w_0\bigl)(x)$ or, when no ambiguity is possible, $e^{-t\cI_*}w_0(x)$. From now on, in order to alleviate the notation, we will drop the subscript $e$ in $n_e$. 
  
\subsection{The linear equation in the whole space}\label{s3.2}

\begin{theorem}
	\label{thm_heat kernel estimate}
  Let $w_0$ be a bounded function satisfying $\Vert w_0\Vert_{H^m(\R^N)}<+\infty$ for all $m\in\mathbb{Z}$, and $w(t,x)=e^{-t\cI_*}w_0(x)$. Then, there is $T>0$ sufficiently large, and  a 
   kernel ${\mathcal{G}}_n(t,x)$, defined for $t\ge T$ and $x\in\R^N$, such that for any $\varpi\in(\frac{3}{8},\frac{1}{2})$, the following estimate holds true:  
	\begin{equation}
		\label{linear sol estimate}
		\Vert w(t,\cdot)-{\mathcal{G}}_n(t,\cdot)*w_0\Vert_{L^\infty(\R^N)}\le e^{-t^{1-2\varpi}}\Vert w_0\Vert_{H^m(\R^N)},~~~~t\ge T.
	\end{equation}
	Furthermore, by choosing $\delta\in(0,\frac{1}{2}-\varpi)$, the following estimates hold  for $t\ge T$.
\begin{itemize}
\item {\bf The range $\vert x\vert\le t^{\frac12+\delta}$.} 
We have
\begin{equation}\label{estimate-G-1}
{\mathcal{G}}_n(t,x)\approx\frac{e^{-\frac{x^T H_n^{-1} x}{2t}}}{(\sqrt{2\pi t})^{\frac{N}{2}}\sqrt{\det H_n}}\bigg(1+\frac{1}{\sqrt{t}}\mathcal{P}(\beta)+O\big(t^{1-4\varpi}\big)\bigg),
		\end{equation}
		where 
		$$\mathcal{P}(\beta)=\di\sum_{|\alpha|=3}\frac1{\alpha!}\di{\int_{\rn}x^\alpha\ck_{*,n}(x) \md x}\ \beta^\alpha-\frac{1}{2}\int_{\R^N} (x^T H_n^{-1}x)(\beta\cdot x)\ck_{*,n}(x) \md x,~~  \text{with }\beta:=\frac{H_n^{-1}x}{\sqrt{t}}.
		$$
\item {\bf The range $\vert x\vert\ge t^{\frac12+\delta}$.} We have		
		\begin{equation}\label{estimate-G-2}
	{\mathcal{G}}_n(t,x)\le C t^{-\varpi N} e^{-Nt^\delta}.
		\end{equation}
		\end{itemize}
\end{theorem}
\begin{remark}\label{rk_diffusion eqn}
	The main consequence of Theorem \ref{thm_heat kernel estimate} is that the solution $w$ to the linear problem \eqref{w-eqn-linear} is actually approximated by   $w^{app}(t,x)=\mathcal{G}^{app}_n(t,\cdot)*w_0(x)$, with 
	\begin{equation}
		\label{e3.4}
		\mathcal{G}^{app}_n(t,x)=\frac{1}{\sqrt{\det H_n}(2\pi t)^{\frac{N}{2}}}e^{-\frac{x^T H_n^{-1}x}{2t}}.
	\end{equation}
	The function $w^{app}(t,x)$  solves the diffusive equation
	\begin{equation}\label{e3.5}
		\partial_t w^{app}-\frac{1}{2}\mathrm{div}\bigl(H_n\nabla w^{app}\bigl)=0,~t>0,x\in\R^N,~~~~ w^{app}(0,x)=w_0(x).
	\end{equation}
\end{remark}

\begin{proof}[Proof of Theorem \ref{thm_heat kernel estimate}]

	As the set of smooth compactly supported functions of $\RR^N$ is dense in every $H^m(\RR^N)$, we may assume $w_0$ to be compactly supported when it is convenient to us. Applying the Fourier transform to \eqref{w-eqn-linear} gives that
	\begin{equation*}
		\widehat{w}(t,\xi)=e^{t\left(\widehat{\ck_{*,n}}(\xi)+\mathbf{i}w^*(e) e\cdot\xi-\widehat{\ck_{*,n}}(0)\right)}\widehat{w_0}(\xi),~~~~~~t\ge0,~\xi\in\R^N.
	\end{equation*}
	It then follows from the inverse Fourier transform that
	\begin{align}\label{w-expression}
		w(t,x)=\frac{1}{(2\pi)^N}\int_{\R^N}e^{\mathbf{i}x\cdot \xi}e^{t\left(\widehat{\ck_{*,n}}(\xi)+\mathbf{i}w^*(e) e\cdot\xi-\widehat{\ck_{*,n}}(0)\right)}\widehat{w_0}(\xi)\md\xi,~~~t>0,~x\in\R^N.
	\end{align}
	
	Since $\ck_{*,n}\in L^1(\rn)$, the Riemann-Lebesgue lemma implies that $|\widehat{\ck_{*,n}}(\xi)|\to 0$ as $|\xi|\to+\infty$. Thus, there exists $R>1$ sufficiently large such that $|\widehat{\ck_{*,n}}(\xi)|<\frac{1}{2}\widehat{\ck_{*,n}}(0)$ for $|\xi|\ge R$.
	
	Next, let us look at\footnote{For notational simplicity, we denote  the N-dimensional box of radius $r>0$ by $\mathcal{D}\big(r\big):=[-r,r]^N$.}  $t\ge T$  and $\xi\in\mathcal{D}(R)\backslash \mathcal{D}(t^{-\varpi})$, with some $T>0$ sufficiently large, and with  $\varpi\in(1/3,1/2)$ to be determined later. First of all, we notice that when $\xi$ is very close to the inner boundary $\partial \mathcal{D}(t^{-\varpi})$,
	\begin{equation*}
		\widehat{\ck_{*,n}}(0)-\Re\widehat{\ck_{*,n}}(\xi)=\int_{\R^N}(1-\cos(x\cdot \xi))\ck_{*,n}(x)\md x\ge\int_{\R^N}\frac{(x\cdot\xi)^2}{4}\ck_{*,n}(x)\md x\ge C t^{-2\varpi},
	\end{equation*}
	by noticing that $1-\cos z=\di\frac{z^2}{2}+o(z^2)$ as $z\to 0$, and the kernel $\ck_{*,n}$ is compactly supported in $\R^N$. In the remaining case, we claim that
	$\widehat\ck_{*,n}(0)-\Re\widehat\ck_{*,n}(\xi)\ge \delta>0$ for some $\delta>0$. Indeed, for each $\xi\neq 0$, we have that 
	$S_\xi:=\{x\in\textnormal{supp}(\ck_{*,n})|\cos(x\cdot \xi)<1\}$ has a positive measure. Therefore, the function 
	$$p(\xi):=\widehat{\ck_{*,n}}(0)-\Re\widehat{\ck_{*,n}}(\xi)=\int_{\R^N}(1-\cos(x\cdot\xi))\ck_{*,n}(x)\md x
	$$ is continuous and positive whenever $\xi\neq 0$. By the extreme value theorem, it attains a minimum value $\delta>0$ in this annulus, as claimed. Consequently, we conclude that 
	\begin{equation*}
		\widehat{\ck_{*,n}}(0)-\Re\widehat{\ck_{*,n}}(\xi)\ge C t^{-2\varpi}
	\end{equation*}
	for $t\ge T$ and $\xi\in\R^N\backslash \mathcal{D}(t^{-\varpi})$.
	
	Based upon the above analysis,  \eqref{w-expression} can be decomposed as $w(t,x)=\cI_1(t,x)+\cI_2(t,x)$ for $t\ge T$ and $x\in\R^N$, where
	\begin{align*}
		\cI_1(t,x)=&\displaystyle\frac{1}{(2\pi)^N}\int_{\mathcal{D}(t^{-\varpi})}e^{\mathbf{i}x\cdot \xi}e^{t\left(\widehat{\ck_{*,n}}(\xi)+\mathbf{i}w^*(e) e\cdot\xi-\widehat{\ck_{*,n}}(0)\right)}\widehat{w_0}(\xi)\md\xi\\
		\cI_2(t,x)=&\displaystyle\frac{1}{(2\pi)^N}\int_{\R^N\backslash \mathcal{D}(t^{-\varpi})}e^{\mathbf{i}x\cdot \xi}e^{t\left(\widehat{\ck_{*,n}}(\xi)+\mathbf{i}w^*(e) e\cdot\xi-\widehat{\ck_{*,n}}(0)\right)}\widehat{w_0}(\xi)\md\xi.
	\end{align*}

	We first notice that
	\begin{align*}
		|\cI_2(t,x)|\le& \frac{1}{(2\pi)^N}\int_{\R^N\backslash \mathcal{D}(t^{-\varpi})}\Big|e^{t\left(\widehat{\ck_{*,n}}(\xi)+\mathbf{i}w^*(e) e\cdot\xi-\widehat{\ck_{*,n}}(0)\right)}\Big|~|\widehat{w_0}(\xi)|\md\xi\\
		\le&\frac{e^{-t^{1-2\varpi}}}{(2\pi)^N}\left(
		\int_{\R^N\backslash \mathcal{D}(t^{-\varpi})}\frac{1}{(1+|\xi|^2)^m}\md\xi
		\right)^{\frac{1}{2}}\left(\int_{\R^N\backslash \mathcal{D}(t^{-\varpi})}(1+|\xi|^2)^m |\widehat{w_0}(\xi)|^2\md\xi\right)^\frac{1}{2}\\
		\le& C\Vert w_0\Vert_{H^m(\R^N)}e^{-t^{1-2\varpi}},~~~~t\ge T,~~~x\in\R^N.
	\end{align*}
	
	To estimate $\cI_1$, we define  ${\mathcal{G}}_n$ by
	\begin{equation*}
		{\mathcal{G}}_n(t,x):= \frac{1}{(2\pi)^N} \int_{\mathcal{D}(t^{-\varpi})}e^{\mathbf{i}x\cdot \xi}e^{t\left(\widehat{\ck_{*,n}}(\xi)+\mathbf{i}w^*(e) e\cdot\xi-\widehat{\ck_{*,n}}(0)\right)} \md\xi,~~~~t\ge T,~~~x\in\R^N,
	\end{equation*}
	then we have 
	\begin{align*}
		\cI_1(t,x)={\mathcal{G}}_n(t,\cdot)*w_0(x),~~~~t\ge T,~~~x\in\R^N.
	\end{align*}
	As a consequence, we arrive at \eqref{linear sol estimate}.
	
	Next, our goal is to investigate the asymptotics of ${\mathcal{G}}_n$. Using \eqref{e2.10}, we have the expansion
	\begin{equation*}
		\widehat{\ck_{*,n}}(\xi)=\widehat{\ck_{*,n}}(0)-\mathbf{i}w^*(e)e\cdot \xi-\frac{1}{2}\xi^TH_n\xi+\mathbf{i}P_n(\xi)+O\big(|\xi|^4\big),~~~\xi\in \mathcal{D}(t^{-\varpi}),
	\end{equation*} 
	where $H_n$ is defined by \eqref{def_H matrix}, and $P_n(\xi)$ is a polynomial with respect to $\xi$ given by
	\begin{equation}\label{def of poly}
		P_{n}(\xi):=\sum_{|\alpha|=3}\frac1{\alpha!}{\di\int_{\rn}x^\alpha\ck_{*,n}(x) \md x}\xi^\alpha.
	\end{equation}
	The symbol $O\big(|\xi|^4\big)$ above shall be understood as a
	holomorphic Taylor remainder. That is to say, there exists a holomorphic function 
	\begin{equation*}
		\mathcal{R}_{4,n}(z):=\widehat{\ck_{*,n}}(z)-\left(\widehat{\ck_{*,n}}(0)-\mathbf{i}w^*(e)e\cdot z-\frac{1}{2}z^TH_n z+\mathbf{i}P_n(z)\right)
	\end{equation*}
	defined in a complex
	neighborhood of the set swept out by the complex-deformed contour 
	satisfying $|\mathcal{R}_{4,n}(z)|=O(|z|^4)$ uniformly in that neighborhood. We shall use this convention for the remainder terms throughout this proof  without further comment.
	
	After the scaling $\zeta=\xi\sqrt{t}$, since $\di\big|t\mathcal{R}_{4,n}\big(\frac{\zeta}{\sqrt t}\big)\big|=
	O\big(\frac{|\zeta|^4}{t}\big)$, we get
	\begin{equation}
		\label{N-D kernel}
		\begin{aligned}
			{\mathcal{G}}_n(t,x)
			=&\frac{1}{(2\pi\sqrt{t})^N} \int_{\mathcal{D}( t^{\frac{1}{2}-\varpi})}\exp\bigg(\underbrace{\mathbf{i}\frac{x\cdot \zeta}{\sqrt{t}}-\frac{1}{2}\zeta^TH_n\zeta+\mathbf{i}\frac{P_n(\zeta)}{\sqrt{t}}+O\left(\frac{|\zeta|^4}{t}\right)	}_{=:\Psi(t,x,\zeta)}
			\bigg)  \md \zeta.
		\end{aligned}
	\end{equation}
	Define the complex contour\footnote{ Let us remark that
		the multidimensional contour deformation is performed coordinate by coordinate, as an iteration of one-dimensional contour deformations.}
	\begin{equation}\label{Gamma^*-N}
		\Gamma^*:=\left\{\zeta=\eta+\bi\beta~\big|~ \eta\in \mathcal{D}( t^{\frac{1}{2}-\varpi})\right\},~~~~\text{with}~~\beta:=\frac{H_n^{-1}x}{\sqrt{t}}.
	\end{equation}
	Along $\Gamma^*$, $\Psi$ can be formulated as
	\begin{equation*}
		\Psi(t,x,\zeta)=-\frac{1}{2}\eta^T H_n\eta-\frac{x^T H_n^{-1} x}{2t}+\underbrace{\bi\frac{P_n(\eta+\bi \beta)}{\sqrt{t}}+O\Bigg(\frac{|\eta+\bi\beta |^4}{t}\Bigg)}_{=:\psi(t,x,\eta)}.
	\end{equation*}
	Suppose that the remainder $\psi$ is negligible, i.e.
	\begin{equation}\label{cdn-1-N}
		\big|\psi(t,x,\eta)\big|=\bi\frac{P_n(\eta+\bi \beta)}{\sqrt{t}}+\frac{|\eta+\bi \frac{H_n^{-1}x}{\sqrt{t}} |^4}{t}\ll 1~~~~~\text{for}~~\eta\in \mathcal{D}( t^{\frac{1}{2}-\varpi}),
	\end{equation}
	then $\Psi$ can be dominated by $-\frac{1}{2}\eta^T H_n\eta-\frac{x^T H_n^{-1} x}{2t}$. Actually, \eqref{cdn-1-N} can be achieved by requiring
	\begin{equation*}
		\frac{|\eta|}{t^{\frac{1}{6}}}\ll 1,~~~~~\frac{|x|}{t^{\frac{2}{3}}}\ll 1~~~\text{for}~~\eta\in \mathcal{D}( t^{\frac{1}{2}-\varpi}).
	\end{equation*}
	To this end, it is sufficient to impose
	\begin{equation*}
		\varpi\in\left(\frac{3}{8},\frac{1}{2}\right),~~~~~~~~x\in\mathcal{D}\big(t^{\frac{1}{2}+\delta}\big), ~~~~~\text{with}~~\delta\in\left(0,\frac{1}{2}-\varpi\right).
	\end{equation*}
	In what follows,  we will therefore proceed with our analysis via dividing into two situations: either $x\in\mathcal{D}\big(t^{\frac{1}{2}+\delta}\big)$ or $x\in\R^N\backslash \mathcal{D}\big(t^{\frac{1}{2}+\delta}\big)$.
	\medskip

	\noindent
	{\bf Case 1.} $x\in \mathcal{D}\big(t^{\frac{1}{2}+\delta}\big)$. Define
	\begin{equation*}
		\Gamma:=\left\{\zeta=\eta+\bi s\beta~\big|~ \eta\in\mathcal{D}\big(t^{\frac{1}{2}-\varpi}\big),~0\le s\le 1\right\}.
	\end{equation*}
	Together with $\Gamma^*$ given in \eqref{Gamma^*-N}, we recast \eqref{N-D kernel} as
	\begin{align*}
		{\mathcal{G}}_n(t,x)=\frac{1}{(2\pi\sqrt{t})^N}\int_{\Gamma^*\cup\Gamma}\exp\big(\Psi(t,x,\zeta)\big) \md \zeta.
	\end{align*}
	
	On the one hand, we observe that on $\Gamma$,
	\begin{equation*}
		\Psi(t,x,\zeta)=\bi\frac{x\cdot\eta}{\sqrt{t}}-s\frac{ x\cdot \beta}{\sqrt{t}}-\frac{1}{2}(\eta+\bi s\beta)^T H_n (\eta+\bi s\beta) +O(t^{1-3\varpi}),
	\end{equation*}
	due to the fact that $\psi$ is in $O(t^{1-3\varpi})$.
	Therefore, 
	\begin{align*}
		\Re\Psi(t,x,\zeta)&=-s\frac{ x\cdot \beta}{\sqrt{t}}-\frac{1}{2}\Big(\eta^T H_n \eta-s^2\beta^T H_n \beta
		\Big)+O(t^{1-3\varpi})\le -\frac{1}{2}\eta^T H_n\eta(1+o(1)).
	\end{align*}
	Since the matrix $H_{n}$ is symmetric and positive definite, there exists an orthogonal matrix $Q_n$ such that
	\begin{equation*}\label{min eigen}
		Q_n^TH_nQ_n=\diag(d_{1,n},d_{2,n},\dots,d_{N,n})=:\Lambda_n,~~~\text{with}~~d_*:=\min\big(d_{1,n},d_{2,n},\dots,d_{N,n}\big)>0.
	\end{equation*}
	Then, we get
	\begin{align*}
		\frac{1}{(2\pi\sqrt{t})^N}\left|\int_{\Gamma}\exp\big(\Psi(t,x,\zeta)\big) \md \zeta\right|\le C t^{-\varpi N}e^{-\frac{1}{2}d_*t^{1-2\varpi}}.
	\end{align*}

	On the other hand, along $\Gamma^*$,  we derive  from the definition \eqref{def of poly} of the polynomial $P_n$ that
	\begin{align*}
		&P_n(\zeta)
		=P_{n}(\eta+\bi \beta)=\sum_{|\alpha|=3}\frac1{\alpha!}{\di\int_{\rn}x^\alpha\ck_{*,n}(x) \md x}\big(\eta+\bi \beta\big)^\alpha\\
		=&\frac{1}{6}\sum_{i,j,k=1}^N \bigg(\int_{\R^N} x_i x_j x_k\ck_{*,n}(x) \md x\bigg)(\eta_i+\bi \beta_i)(\eta_j+\bi \beta_j)(\eta_k+\bi \beta_k)\\
		=&P_n(\eta)-\frac{1}{2}\sum_{i,j,k=1}^N \bigg(\int_{\R^N} x_i x_j x_k\ck_{*,n}(x) \md x\bigg)\eta_i\beta_j\beta_k
		+\bi\Bigg(
		\frac{1}{2}\sum_{i,j,k=1}^N \bigg(\int_{\R^N} x_i x_j x_k\ck_{*,n}(x) \md x\bigg)\eta_i\eta_j\beta_k- P_n(\beta)\Bigg)
	\end{align*}
	with 
	\begin{equation*}
		E_n(\beta)=\Bigg(\sum_{k=1}^N \bigg(\int_{\R^N} x_i x_j x_k\ck_{*,n}(x) \md x\bigg)\beta_k\Bigg)_{1\le i,j\le N}.
	\end{equation*}
	It then follows from a Taylor expansion and the change of variable $\rho=\sqrt{H_n}\eta$  that
	\begin{align*}
		&\frac{1}{(2\pi\sqrt{t})^N}\int_{\Gamma^*}\exp\big(\Psi(t,x,\zeta)\big) \md \zeta\\
		&=\frac{1}{(2\pi\sqrt{t})^N}\int_{\mathcal{D}(t^{\frac{1}{2}-\varpi})}\exp\bigg(-\frac{1}{2}\eta^T H_n\eta-\frac{x^T H_n^{-1} x}{2t}+\psi(t,x,\eta)\bigg) \md \eta\\
		&\approx\frac{1}{(2\pi\sqrt{t})^N}\int_{\R^N}\exp\bigg(-\frac{1}{2}\eta^T H_n\eta-\frac{x^T H_n^{-1} x}{2t}+\psi(t,x,\eta)\bigg) \md \eta\\
		&=\frac{e^{-\frac{x^T H_n^{-1} x}{2t}}}{(2\pi\sqrt{t})^N}\int_{\R^N}e^{-\frac{1}{2}\eta^T H_n\eta}\bigg(1-\frac{1}{\sqrt{t}}\Big(\frac{1}{2}\eta^TE_n(\beta)\eta-P_n(\beta)\Big)+O(t^{1-4\varpi})\bigg) \md \eta\\
		&=\frac{e^{-\frac{x^T H_n^{-1} x}{2t}}}{(2\pi\sqrt{t})^N\sqrt{\det H_n}}\int_{\R^N}e^{-\frac{1}{2}|\rho|^2}\bigg(1+\frac{1}{\sqrt{t}}\Big( P_n(\beta)-\frac{1}{2}\rho^TE_n^*(\beta)\rho\Big)+O(t^{1-4\varpi})\bigg)\md\rho\\
		&=\frac{e^{-\frac{x^T H_n^{-1} x}{2t}}}{(\sqrt{2\pi t})^{\frac{N}{2}}\sqrt{\det H_n}}\bigg(1+\frac{1}{\sqrt{t}}\Big(P_n(\beta)-\frac{1}{2}\text{tr}E_n^*(\beta)\Big)+O\Big(t^{1-4\varpi}\Big)\bigg),
	\end{align*}
	where, beacuse $H_n$ is symmetric and positive definite, we have $E_n^*(\beta)=H_n^{-\frac{1}{2}} E_n(\beta)H_n^{-\frac{1}{2}}$,  and
	\begin{equation*}
		\text{tr}E_n^*(\beta)=	\text{tr}\big(E_n(\beta)H_n^{-1}\big)=\int_{\R^N} (x^T H_n^{-1}x)( \beta\cdot x)\ck_{*,n}(x) \md x,
	\end{equation*}
	whence we conclude that 
	\begin{equation}\label{estimate1-N}
			{\mathcal{G}}_n(t,x)=\frac{1}{(2\pi\sqrt{t})^N}\int_{\Gamma^*\cup\Gamma}\exp\big(\Psi(t,x,\zeta)\big) \md \zeta  \approx
			 \frac{e^{-\frac{x^T H_n^{-1} x}{2t}}}{(\sqrt{2\pi t})^{\frac{N}{2}}\sqrt{\det H_n}}\bigg(1+\frac{1}{\sqrt{t}}\mathcal{P}(\beta)+O\Big(t^{1-4\varpi}\Big)\bigg),
	\end{equation}
	with
	\begin{equation*}
		\mathcal{P}(\beta):=	P_{n}(\beta)-\frac{1}{2}\text{tr}E_n^*(\beta)=\sum_{|\alpha|=3}\frac1{\alpha!}{\di\int_{\rn}x^\alpha\ck_{*,n}(x) \md x}\beta^\alpha-\frac{1}{2}\int_{\R^N} (x^T H_n^{-1}x)(\beta\cdot x)\ck_{*,n}(x) \md x.
	\end{equation*}

	\noindent
	{\bf Case 2.} $x\in\R^N\backslash \mathcal{D}\big(t^{\frac{1}{2}+\delta}\big)$. The term $\psi(t,x,\eta)$ may not be negligible, for which we introduce 
	\begin{equation*}
		\Gamma^1:=\left\{\zeta=\eta+\bi s\mathbf{1}~\big|~\eta\in\partial  \mathcal{D}\big(t^{\frac{1}{2}-\varpi}\big), ~0\le s\le 1\right\},~~~~~\Gamma^2=\left\{\zeta=\eta+\bi\mathbf{1}~\big|~ \eta\in \mathcal{D}\big(t^{\frac{1}{2}-\varpi}\big)\right\},
	\end{equation*}
	with
	\begin{equation*}
		\mathbf{1}=(1,1,\cdots,1)\in\rn.
	\end{equation*}
	On $\Gamma^1$, we find that
	\begin{align*}
		\Psi(t,x,\zeta)=\mathbf{i}\frac{x }{\sqrt{t}}\cdot(\eta+\bi s\mathbf{1})-\frac{1}{2}(\eta+\bi s\mathbf{1})^T H_n (\eta+\bi s\mathbf{1}) +O\left(t^{1-3\varpi}\right),
	\end{align*}
	whence 
	\begin{equation*}
		\Re\Psi(t,x,\zeta)=-s\frac{x\cdot\mathbf{1}}{\sqrt{t}}-\frac{1}{2}\Big(
		\eta^T H_n \eta-\mathbf{1}^T H_n \mathbf{1}\Big) +O(t^{1-3\varpi})
		\le -\frac{1}{2}
		\eta^T H_n \eta(1+o(1)). 
	\end{equation*}
	On $\Gamma^2$,
	\begin{equation*}
		\Psi(t,x,\zeta)=	\mathbf{i}\frac{x }{\sqrt{t}}\cdot(\eta+\bi\mathbf{1})-\frac{1}{2}(\eta+\bi\mathbf{1})^T H_n (\eta+\bi\mathbf{1})+O\left(t^{1-3\varpi}\right),
	\end{equation*}
	therefore 
	\begin{equation*}
		\Re \Psi(t,x,\zeta)\le -\frac{x\cdot \mathbf{1}}{\sqrt{t}}+O(1)\le -t^{\delta}+O(1).
	\end{equation*}
	We then conclude that
	\begin{align}\label{estimate2-N}
		\big|{\mathcal{G}}_n(t,x)\big|&=\frac{1}{(2\pi\sqrt{t})^N}\bigg|\int_{\Gamma^1\cup\Gamma^2}\exp\big(\Psi(t,x,\zeta)\big) \md \zeta\bigg|\le Ct^{-\varpi N}e^{-Nt^\delta}.
	\end{align}

	\medskip
	
	\noindent
	{\bf Conclusion.} Gathering \eqref{estimate1-N} and \eqref{estimate2-N} altogether, we achieve \eqref{estimate-G-1} and \eqref{estimate-G-2}, as desired. This completes the proof of Theorem \ref{thm_heat kernel estimate}.
\end{proof}

The analysis leading to Theorem \ref{thm_heat kernel estimate} starts by the analysis of a Fourier transform. This can also be pursued to prove the following proposition: while its proof is almost immediate and, in addition, will not necessarily be useful in the sequel, it reinforces the elliptic analogy. 
\begin{proposition}
If $v\in L^2(\RR^N;\R)$, we have
\begin{equation*}
	\int_{\R^N}v(x)\cI_{*,n}v(x)\md x\geq C\int_{\R^N}\min\big(|\xi|^2,1\big)|\hat v(\xi)|^2\md \xi.
\end{equation*} 
\end{proposition}
\begin{proof}
By Plancherel's equality, it is equivalent to looking at $\int_{\RR^N}\overline{\hat v(\xi) }\widehat{\cI_{*,n}v}(\xi)\md\xi$. As it is  real, we  have
\begin{align*}
	\int_{\RR^N}\overline{\hat v(\xi) }\widehat{\cI_{*,n}v}(\xi)\md x=\Re\int_{\RR^N}\overline{\hat v(\xi) }\widehat{\cI_{*,n}v}(\xi)\md x
	&=\displaystyle\int_{\R^N}\big((\widehat{\ck_{*,n}}(0)-\Re\widehat{\ck_{*,n}}(\xi)\big)\vert \hat v(\xi)\vert^2 \md\xi.
\end{align*}
The conclusion follows by repeating the argument given at the beginning of Proof of Theorem \ref{thm_heat kernel estimate}. 
 \end{proof}

\subsection{The linear equation with a well spread initial datum}
While Theorem \ref{thm_heat kernel estimate} will be an essential brick in the construction of  barriers for the solutions of the nonlinear problem, it has the drawback of involving a time $T$, after which it is valid.  This will turn out to be harmless when we examine the solutions from the initial time onwards. It will however become a liability if we wish to examine the Cauchy problem starting from an arbitrarily large time. Nevertheless, as time goes on, the solutions of the nonlinear problem ahead of the front will eventually varies on the diffusive spatial scale $\sqrt t$. The following proposition, inspired form \cite[Chapter 2]{Roquejoffre2023}, takes advantage of this structure to give estimates  not involving a waiting time.  Hereafter, we set
\begin{equation}
\label{e3.75}
	A_n:=-\frac{1}{2}\mathrm{div}\big(H_n\nabla\big).
\end{equation} 
The main result of this section is the
 \begin{proposition}
\label{p3.1}
Consider a family of functions $(v_s)_{s\geq0}$ defined by
 \begin{equation}
\label{e3.70}
v_s(x)=\bnu\Big(s,\frac{x}{\sqrt{s}}\Big),~~~~x\in\R^N,
\end{equation}
such that, for each $s>0$ we have $\bnu(s,.)\in H^m(\R^N)\cap L^1(\R^N)$, for every integer $m$. Pick $\varsigma\in(0,\displaystyle\frac1{8(2N+5)}]$, set $q_N=2N+5$ and $m_N=5N+7$. There is a family $\bigl(Q_\alpha(\tau)\bigl)_{3\leq\vert\alpha\leq q_N}$ of polynomials in $\tau$, of degree less than $q_N$, such that,
for all $\tau\in(0,s^{\varsigma})$ we have 
\begin{equation*}
\label{e3.71}
\biggl\vert e^{-\tau\cI_{*,n}}v_s(x)-\biggl(\bigl(e^{\frac{\tau}{s}A_n}\bnu(s,.)\bigl)+\sum_{3\leq\vert\alpha\vert\leq q_N}\frac{Q_\alpha(\tau)}{s^{\frac{\vert\alpha\vert}2}}\bigl(e^{\frac{\tau}{s}A_n}\partial^\alpha \bnu(s,.)\bigl)\biggl)(\frac{x}{\sqrt s})\biggl\vert\lesssim\frac{\tau}{s^{\frac{N}{2}+1}} \Vert\bnu(s,.)\Vert_{L^1\cap H^{m_N}}.
\end{equation*}
\end{proposition}
The proof of the proposition is conceptually easy but technical, and does not bring additional elements of understanding. It is therefore deferred to the appendix. We just make two remarks: the first one is that the values of $\varsigma$ and $q_N$ may look arbitrary, we have not sought here to give the whole range of values of $q$ and $\varsigma$ for which Proposition \ref{p3.1} is valid, because the version that we have stated is what we will need in Section \ref{sec5}, where will seek improved barriers for our solutions. It is also worth, 
at this stage, pointing out the classical identity
\begin{equation}
\label{e3.2000}
\big(e^{\tau A_n}v_s\big)(x)=\big(e^{\frac\tau{s} A_n}\bnu(s,\cdot)\big)\Big(\frac{x}{\sqrt{s}}\Big).
\end{equation}
that is used in the proof of the proposition, and that will turn out to be quite useful in the computations of Section \ref{sec5}. This identity entails an estimate which is classical in diffusion equations,  but that had only been noticed in \cite{Roquejoffre2023} in the context of integral equation (Chapter 4, Theorem 4.3.1). Because it is of significant use here, we state it and give its proof.
 \begin{proposition}
\label{p3.3}
Consider a family of functions $(v_s)_{s\geq0}$ as in Proposition \ref{p3.1}, such that we have in addition, for every integer $m$:
$$
\Vert \bnu(s,.)\Vert_{H^m(\RR^N)}\leq K_m,
$$
for a constant $K_m$ that depends on $m$ but not on $s$. Also, pick  three constants $\varsigma\in(0,\displaystyle\frac1{8(2N+5)}]$, and $A>1$. Consider $x\in\RR^N$ 
and $\tau\in[0,s^\varsigma]$. We set $K=K_{m_N}$, the notation also being that of Proposition \ref{p3.1}. There exists a constant $\Lambda_{A,K}>0$ such that 
$$
\bigl(1+O(e^{-\frac{\mu s^\varsigma}4})\bigl)\inf_{\vert\xi-\frac{x}{\sqrt s}\vert\leq {s^{\varsigma-\frac12}}}\bnu(s,\xi)
-\frac{\Lambda_{A,K}}{s^{\frac32-q_N\varsigma}}\leq e^{-\tau\cI_*}v_s(x)\leq \bigl(1+O(e^{-\frac{\mu s^\varsigma}4})\bigl)\sup_{\vert\xi-\frac{x}{\sqrt s}\vert\leq
{s^{\varsigma-\frac12}}}\bnu(s,\xi)+\frac{\Lambda_{A,K}}{s^{\frac32-q_N\varsigma}},
$$
where $\mu>0$ is the coercivity constant in \eqref{H_lower bound}.
\end{proposition}
\begin{proof} The proof consists in putting in quantitative form the following simple idea: the solution $e^{-t\cI_*}v_s$ is well approximated by $e^{\frac\tau{s}A_n}\bnu(s,.)$, while the latter  is, at time $\di\frac{\tau}s=s^{\varsigma-1}$, is concentrated around $x$. 

From the assumptions of the proposition, the $H^m$ norms of $\bnu(s,.)$ are bounded independently of $s$. Therefore, from Proposition \ref{p3.1}, it is enough to examine $e^{\frac{\tau}sA_v}\bnu(s,.)$ and its successive derivatives up to the order $q_N$.

 Let us  consider the main term in the expansion of $e^{-\tau\cI_*}v_s(x)$, that is, $\bigl(e^{\frac{\tau}{s} A_{n}}\bnu(s,.)\bigl)(\zeta)$. Set  $\zeta=\di\frac{x}{\sqrt s}$;
we have $e^{\frac{\tau}{s} A_{n}}\bnu(s,.)=\mathcal{G}^{app}_{n}(\displaystyle\frac{\tau}{s},.)*_\zeta{{\bnu}}(s,.)$.
Identity \eqref{e3.4} and  estimate \eqref{H_lower bound} entail, for $\xi\in\RR^N$ such that $\vert\xi-\zeta\vert\geq s^{\varsigma-\frac12}$, and for $0<\tau\leq s^\varsigma$:
$$
\mathcal{G}^{app}_{n}(\displaystyle\frac{\tau}{s},\zeta-\xi)=(\frac{s}\tau)^{\frac{N}2}\frac{\mathrm{exp}\biggl(-\di\frac{(\xi-\zeta)^TH_n(\xi-\zeta)}{\di2\tau/{s}}\biggl)}{\sqrt{\mathrm{det}H_n}}\leq
(\frac{s}\tau)^{\frac{N}2}\frac{e^{-\mu\frac{\vert -\xi+\zeta\vert^2}{2\tau/s}}}{\sqrt{\mathrm{det}H_n}}\lesssim  (\frac{s}\tau)^{\frac{N}2}e^{- \frac{\mu s^{2\varsigma}}{2\tau}}\leq e^{-\frac{\mu s^\varsigma}{4}},
$$
if $s$ is large enough. So, we have
$$
\begin{array}{rll}
\mathcal{G}^{app}_{n}(\displaystyle\frac{\tau}{s},.)*_\zeta \bnu(s,.)(\zeta)=&\di\int_{\vert\zeta-\xi\vert\leq s^{\varsigma-\frac12}}\frac{\mathrm{exp}\biggl(-\di\frac{(\xi-\zeta)^TH_n(\xi-\zeta)}{\di2\tau/{s}}\biggl)}{\sqrt{\mathrm{det}H_n}\bigl(\di{\tau}/s\bigl)^{\frac{N}2}}\bnu(s,\xi)\md \xi+O(e^{-\frac{\mu s^{\varsigma}}4})\Vert\bnu(s,.)\Vert_\infty\\
\geq&\di\inf_{\vert\zeta-\xi\vert\leq s^{\varsigma-\frac12}}\bnu(s,\xi)\int_{\vert\xi\vert\leq s^{\varsigma-\frac12}}\frac{\mathrm{exp}\biggl(-\di\frac{\xi^TH_n\xi}{\di2\tau/{s}}\biggl)}{\sqrt{\mathrm{det}H_n}\bigl(\di{\tau}/s\bigl)^{\frac{N}2}}\md \xi+O(e^{-\frac{\mu s^{\varsigma}}4})\Vert\bnu(s,.)\Vert_\infty\\
\geq&\bigl(1+O(e^{-\frac{\mu s^\varsigma}4})\bigl)\di\inf_{\vert\zeta-\xi\vert\leq s^{\varsigma-\frac12}}\bnu(s,\xi)+O(e^{-s^{2\varsigma}})\Vert\bnu(s,.)\Vert_\infty.
\end{array}
$$
In the same way, we have
$$
\mathcal{G}^{app}_{n}(\displaystyle\frac{\tau}{s},.)*_\zeta \bnu(s,.)(\zeta)\leq\bigl(1+O(e^{-\frac{\mu s^\varsigma}4})\bigl)\di\inf_{\vert\zeta-\xi\vert\leq s^{\varsigma-\frac12}}\bnu(s,\xi)+O(e^{-s^{2\varsigma}})\Vert\bnu(s,.)\Vert_\infty.
$$
The higher derivatives are examined in the same fashion, yielding the full   proposition.
\end{proof}
\subsection{Approximation of the Dirichlet heat kernel}
\begin{proposition} \label{p3.2}
	Let $w(t,x)$ be the solution to the linear equation \eqref{w-eqn-linear} for $t>0$ and $x\in\rn$, 	associated with a nontrivial, bounded and compactly supported  smooth initial datum $w_0$ in $\R^N$ satisfying  $w_0(x)\ge0$ for $\{x\in\rn|x\cdot n\ge 0\}$ and
	$w_0\big({\mathcal{R}}_n(x)\big)=-w_0(x)$ for $x\in\R^N$ with 
	\begin{equation}
		\label{eqn_reflection}
		{\mathcal{R}}_n(x):=x-\frac{2(x\cdot n)}{n^TH_nn}H_nn.
	\end{equation}
	Then   there exists $T>0$  large enough such that
	\begin{equation*}
		w(t,x)\approx \displaystyle\frac{\big(2+O\big(t^{-\frac{1}{2}}\big)\big)x\cdot ne^{-\frac{{x^T H_n^{-1}x}}{2t}}}{n^TH_nn\sqrt{\det H_n }(2\pi)^{\frac{N}{2}}t^{\frac{N}2+1}}\int_{y\cdot n\ge0}(y\cdot n)w_0(y)\md y,~~~t\ge T,~ |x|=O(t^{\frac{1}{2}+\delta}).
	\end{equation*}
	
\end{proposition}

\begin{proof} 
	Based on our observation in  Remark \ref{rk_diffusion eqn},  we start with the diffusion equation on the half space:
	\begin{equation*}\label{eqn_D heat eqn}
		\partial_t w^{app}+A_n w^{app}=	\partial_t w^{app}-\frac12\mathrm{div}\big(H_n\nabla w^{app}\big)=0,~~~t>0,~~x\cdot n\ge 0,
	\end{equation*}
	associated with a nonnegative nontrivial and compactly supported  smooth  initial datum $w^{app}(0,x)$ for $x\cdot n\ge 0$.
	We notice that
	the function $v(t,y)=w^{app}(t,\sqrt{H_n}y)$ solves the standard heat equation 
	\begin{equation*}
		\label{e3.6}
		\partial_tv-\frac{1}{2}\Delta v=0,~~~t>0,~y\cdot \hat n\ge 0,~~~\text{with}~\hat n:=\frac{\sqrt{H_n}n}{|\sqrt{H_n}n|},
	\end{equation*}
	with the initial datum $v_0(y)=w^{app}(0,\sqrt{H_n}y)$ for $y\in\{y\in\rn|y\cdot\hat n\ge 0\}$. The function $v$ is computed by reflection, setting $v_0(\mathcal{R}_{\hat n}(y))=-v_0(y)$ for $y\in\R^N$, with
	\begin{equation*}
		\mathcal{R}_{\hat n}(y):=y-2(y\cdot \hat n)\hat n.
	\end{equation*} 
	This implies that for the original linear problem \eqref{w-eqn-linear}, the proper reflection transform shall be
	\begin{equation*}
		\mathcal{R}_n(x):=\sqrt{H_n}\mathcal{R}_{\hat n}(y)=\sqrt{H_n}\mathcal{R}_{\hat n}\Big(\sqrt{H_n^{-1}}x\Big)=x-\frac{2(x\cdot n)}{n^TH_nn}H_nn,~~~x\in\rn.
	\end{equation*}
	We therefore arrive at \eqref{eqn_reflection}.
	
	Next, we turn back to  $w$, from Theorem \ref{thm_heat kernel estimate} and \eqref{eqn_reflection} there exists  $T>1$ sufficiently large such that 
	\begin{align*}
		w(t,x)\approx\int_{\R^N}\mathcal{G}_n(t,x-y)w_0(y)\md y=\int_{y\cdot n\ge0}\mathcal{G}_n(t,x-y)w_0(y)\md y+\int_{y\cdot n\le0}\mathcal{G}_n(t,x-y)w_0(y)\md y,
	\end{align*}
	for $t\ge T$ and $x\in\R^N$, where
	\begin{align*}
		\int_{y\cdot n\le0}\mathcal{G}_n(t,x-y)w_0(y)\md y=-\int_{y\cdot n\le0}\mathcal{G}_n(t,x-y)w_0(\mathcal{R}_ny)\md y
		=-\int_{z\cdot n\ge0}\mathcal{G}_n(t,x-\mathcal{R}_n(z))w_0(z)dz.
	\end{align*}
	Therefore,  it follows that
	\begin{align*}
		w(t,x)\approx&\int_{y\cdot n\ge0}\Big(\mathcal{G}_n(t,x-y)-\mathcal{G}_n(t,x-\mathcal{R}_n(y))\Big)w_0(y)\md y\\
		=&\frac{1+O\big(t^{-\frac{1}{2}}\big)}{\sqrt{\det H_n }(2\pi t)^{\frac{N}{2}}}\int_{y\cdot n\ge0}\Big(e^{-\frac{(x-y)^T H_n^{-1} (x-y)}{2t}}-e^{-\frac{(x-\mathcal{R}_n(y))^T H_n^{-1} (x-\mathcal{R}_n(y))}{2t}}\Big)w_0(y)\md y\\
		=&\frac{1+O\big(t^{-\frac{1}{2}}\big)}{\sqrt{\det H_n }(2\pi t)^{\frac{N}{2}}}e^{-\frac{x^T H_n^{-1} x}{2t}}\int_{y\cdot n\ge0}\Big(\underbrace{e^{\frac{x^T H_n^{-1} y}{t}-\frac{y^T H_n^{-1} y}{2t}}-e^{\frac{x^T H_n^{-1} \mathcal{R}_n(y)}{t}-\frac{\mathcal{R}_n(y)^T H_n^{-1} \mathcal{R}_n(y)}{2t}}}_{=:\mathcal{E}_n(t,x,y)}\Big)w_0(y)\md y
	\end{align*}
	for all $t\ge T$ and $| x|=O(t^{\frac{1}{2}+\delta})$, in which we derive from a Taylor expansion that
	\begin{align*}
		\mathcal{E}_n(t,x,y)=\frac{2}{n^TH_n n}\frac{x\cdot n}{t} y\cdot n+O\big(t^{-1}\big).
	\end{align*}
	As a consequence, we get
	\begin{equation*}
		\label{e3.3000}
		w(t,x)\approx\frac{2+O\big(t^{-\frac{1}{2}}\big)}{n^TH_n n \sqrt{\det H_n }(2\pi )^{\frac{N}{2}}}\frac{x\cdot n }{ t^{\frac{N}{2}+1}}e^{-\frac{x^T H_n^{-1} x}{2t}}\int_{y\cdot n\ge0}(y\cdot n)w_0(y)\md y
	\end{equation*}
	for  $t\ge T$ and $| x|=O(t^{\frac{1}{2}+\delta})$. \end{proof}  
\begin{remark}
	\label{rmk_Dirichlet heat kernel}
	Under the same initial condition assumption as in Proposition \ref{p3.2}, it  follows that 
	\begin{equation*}
		e^{tA_n}w_0(x)\approx \displaystyle\frac{2x\cdot ne^{-\frac{{x^T H_n^{-1}x}}{2t}}}{n^TH_nn\sqrt{\det H_n }(2\pi)^{\frac{N}{2}}t^{\frac{N}2+1}}\int_{y\cdot n\ge0}(y\cdot n)w_0(y)\md y~~~\text{for}~~t\ge T~~\text{and}~~ | x|=O(t^{\frac{1}{2}}).
	\end{equation*}
	On the other hand,
	let us rephrase Proposition \ref{p3.2}, for later use in Section \ref{sec5},  by introducing the  notation 
	\begin{equation}
		\label{e3.3001}
		\Gamma_n(\zeta):=\frac{2(\zeta\cdot n)e^{-\frac{\zeta^T H_n^{-1} \zeta}{2}}}{n^TH_n n \sqrt{\det H_n }(2\pi )^{\frac{N}{2}}},~~~\zeta\in\RR^N,
	\end{equation}
	and
	\begin{equation}
		\label{e3.3002}
		\mu_n[w_0]:=\int_{\zeta\cdot n>0}(\zeta\cdot n)w_0(\zeta)\md \zeta
	\end{equation}
	for a bounded function $w_0(\zeta)$ such that $\vert\zeta\vert w_0\in L^1(\R^N)$.
	Proposition \ref{p3.2} asserts that
	\begin{equation}
		\label{e3.3003}
		w(t,x)\approx\frac{\mu_n[w_0]}{t^{\frac{N+1}2}}\Gamma_n\Big(\frac{x}{\sqrt t}\Big)~~\hbox{as $t\to+\infty$, \text{uniformly for} $\displaystyle\frac{\vert x\vert}{\sqrt t}\leq t^\delta$.}
	\end{equation}
\end{remark}

\subsection{Control beyond the diffusive scale}

While Theorem  \ref{thm_heat kernel estimate} and its consequences leads to fairly precise estimates of the dynamics of the solution within the diffusive zone, that is, $\vert x\vert=O(\sqrt t)$, we need to compare the solution of the full problem \eqref{kpp} to its linear approximations outside the diffusive zone, in order to assess the quality of the approximation inside. The first theorem concerns the solution $v(t,x)$ of \eqref{e5.101} obtained by the leading edge transform.
\begin{theorem}
	\label{t3.20}
	Let $w$ be the solution to the linear equation \eqref{w-eqn-linear} in $\R_+\times\R^N$ associated with a nonnegative, bounded and compactly supported {smooth} initial datum $w_0$ in $\R^N$. Then, for all $B\ge 1$ and $\vartheta,\varepsilon\in(0,\frac{1}{2})$, there exist  $T>0$ large enough (depending on $B$, $\vartheta$ and $\varepsilon$), and $\eta_B>0$ depending on $B$ and the matrix $H_n$ given in \eqref{def_H matrix}, such that 
	\begin{equation*}
		0\leq w(t-T,x)\leq \frac{e^{-A\big(\frac{|x|}{\sqrt{t}}-\eta_B\big)}}{t^{\frac{N}{2}-\vartheta}}.
	\end{equation*}
	for $t\ge T$ and $x\in\R^N$. 
\end{theorem}
The proof of this theorem is similar to that of Theorem \ref{thm-control beyond diffusive scale-n} below, therefore we refer to its proof. An obvious consequence is the following.
\begin{corollary}\label{c3.10}
Let $v$ be the solution to the nonlinear equation \eqref{e5.101} in $\R_+\times\R^N$ associated with a nonnegative, bounded and compactly supported {smooth} initial datum $w_0$ in $\R^N$. Then, for all $B\ge 1$ and $\vartheta,\varepsilon\in(0,\frac{1}{2})$, there exist  $T>0$ large enough (depending on $B$, $\vartheta$ and $\varepsilon$), and $\eta_B>0$ depending on $B$ and the matrix $H_n$ given in \eqref{def_H matrix}, such that 
	\begin{equation*}
		0\leq v(t-T,x)\le \frac{e^{-A\big(\frac{|x|}{\sqrt{t}}-\eta_B\big)}}{t^{\frac{N}{2}-\vartheta}}.
	\end{equation*}
	for $t\ge T$ and $x\in\R^N$. 
\end{corollary}
The reason is simply that $v(t,x)\leq w(t,x)$ for all $t\geq0$ and $x\in\RR^N$. Corollary \ref{c3.10} is what we will use in the sequel.
\begin{theorem}
	\label{thm-control beyond diffusive scale-n}
	Let $w^*$ be the solution to the linear equation \eqref{w-eqn-linear} in $\R_+\times\R^N$ associated with a nontrivial, bounded and compactly supported  smooth initial datum $w_0^*$ in $\R^N$ satisfying  $w_0^*(x)\ge0$ for $\{x\in\rn|x\cdot n\ge 0\}$ and
	$w_0^*\big({\mathcal{R}}_n(x)\big)=-w_0^*(x)$ for $x\in\R^N$ with $\mathcal{R}_n$ given by \eqref{eqn_reflection}. Then, for all $A\ge 1$ and $\vartheta,\varepsilon\in(0,\frac{1}{2})$, there exist\footnote{It is worth pointing out that whenever the initial datum $w_0^*$ is not compactly supported anymore, the large time $T$ needs to be chosen depending also on $w_0^*$. } $T>0$ large enough $($depending on $A$, $\vartheta$ and $\varepsilon$$)$, and $\eta_A>N+4+2A^2\bar h$ with $\bar h:=N^2\max_{1\leq i,j\leq N}\vert h_{ij,n}\vert$ and $h_{ij,n}$ given in \eqref{def_H matrix}, such that 
	\begin{equation*}
		e^{tA_n}w_0^*(x)-t^{-\frac{N}{2}-1+\vartheta}e^{-A\big(\frac{|x|}{\sqrt{t}}-\eta_A\big)} \leq w^*(t-T,x)\le e^{tA_n}w_0^*(x)+t^{-\frac{N}{2}-1+\vartheta}e^{-A\big(\frac{|x|}{\sqrt{t}}-\eta_A\big)}
	\end{equation*}
	for $t\ge T$ and $x\in\R^N$. 
\end{theorem}
\begin{proof}
	Choose $T>0$ sufficiently large and define
	\begin{equation*}
		\overline w(t,x)=e^{tA_n}w_0^*(x)+\varphi(t,x),~~~\text{with}~~\varphi(t,x):=t^{-\frac{N}{2}-1+\vartheta}e^{-A\big(\frac{|x|}{\sqrt{t}}-\eta_A\big)},
	\end{equation*}
	for $t\ge T$ and $x\in\R^N$.
	We claim that $\overline w$ satisfies $\overline w_t(t,x)+\mathcal{I}_{*,n}\overline w\ge 0$ for $t\ge T$ and $|x|\ge \eta_A\sqrt{t}$. In this range, we first make use of Remark \ref{rmk_Dirichlet heat kernel} together with \eqref{e5.103} and derive  that
	\begin{equation}\label{part1}
		\big(\partial_t+\cI_{*,n}\big)\big(e^{tA_n}w_0^*(x)\big)=O\Big(\big\Vert D^3\big(e^{tA_n}w_0^*\big)\big\Vert_{L^\infty(B_R)(x)}\Big)=O\Big(t^{-\frac{N}{2}-2}e^{-\frac{x^T H_n^{-1}x}{4t}}\Big);
	\end{equation}
we have indeed $e^{-\frac{x^T H_n^{-1}x}{4t}}\gtrsim e^{-\frac{y^T H_n^{-1}y}{2t}}$  for $y\in B_R(x)$.	On the other hand,
	\begin{align*}
		\partial_t\varphi(t,x)=\left(\frac{A|x|}{2t^{\frac{3}{2}}}-\frac{N+2(1-\vartheta)}{2t}\right)\varphi(t,x),
	\end{align*}
	and  the second order spatial derivatives are 
	\begin{equation}
		\label{e5.104}
		\partial_{x_ix_j}\varphi(t,x)=\frac{A^2}t\frac{x_ix_j}{\vert x\vert^2}\varphi(t,x)-\frac{A}{\sqrt t}\biggl(\frac{\delta_{ij}}{\vert x\vert}-\frac{x_ix_j}{\vert x\vert^3}\biggl)\varphi(t,x),
	\end{equation}
	where $\delta_{ij}$ is the Kronecker symbol. As $\vert x\vert\geq\eta_A\sqrt t$, we have
	$$
	\big| \partial_{x_ix_j}\varphi(t,x)\big| \leq\frac{1}{t}\bigg(A^2+\frac{A}{\eta_A}\bigg)\varphi(t,x).
	$$
	Differentiating \eqref{e5.104} once more and using again that $\vert x\vert \geq\eta_A\sqrt t$, we obtain $\big| D^3\varphi\big|=O\big(t^{-\frac{3}{2}}\varphi(t,x)\big)$.
	Consequently,  it follows from \eqref{e5.103} that
	\begin{align}\label{part2}
		\varphi_t(t,x)+\cI_{*,n}\varphi(t,x)\ge \left(\frac{A\eta_A-(N+2-2\vartheta)-2A^2\bar h}{2t}-\frac{C}{t^{\frac{3}{2}}}\right)\varphi(t,x)\ge Ct^{-\frac{N}{2}-2+\vartheta}e^{-A\big(\frac{|x|}{\sqrt{t}}-\eta_A\big)}
	\end{align}
	for $t\ge T$ and $\vert x\vert\ge \eta_A\sqrt{t}$, up to increasing $T$ if necessary.  Combining \eqref{part1} and \eqref{part2} yields that
	\begin{align*}
		\overline w_t(t,x)+\cI_{*,n}\overline w(t,x)\ge O\Big(t^{-\frac{N}{2}-2}e^{-\frac{x^T H_n^{-1}x}{4t}}\Big)+		
		Ct^{-\frac{N}{2}-2+\vartheta} e^{-A\big(\frac{|x|}{\sqrt{t}}-\eta_A\big)}>0
	\end{align*}
	for $t\ge T$ and $\vert x\vert\ge \eta_A\sqrt{t}$.  Our claim is therefore achieved.
	
	In addition, it follows from Proposition \ref{p3.2} and Remark \ref{rmk_Dirichlet heat kernel} that
	\begin{equation*}
		\big|w^*(t-T,x)-e^{tA_n}w_0^*(x)\big|\le C t^{-\frac{N}{2}-1}<t^{-\frac{N}{2}-1+\vartheta}\le \varphi(t,x)
	\end{equation*}
	for $t\ge 2T$ and $|x|\le \eta_A\sqrt{t}$,	 up to increasing $T$. For $t\in[T,2T]$ and $|x|\le \eta_A\sqrt{t}$, we also have $	\big|w^*(t-T,x)-e^{tA_n}w_0^*(x)\big|< \varphi(t,x)$, up to increasing $\eta_A$ if needed. Finally, by virtue of Remark \ref{rmk_Dirichlet heat kernel}, it is seen that $\varphi(T,x)$ is the leading term in the supersolution $\overline w(T,x)$  for $x\ge \eta_A\sqrt{T}$. 
	Given that $w_0^*$ is compactly supported, it easily follows that  $w_0^*(x)< \overline w(T,x)$ for $x\ge \eta_A\sqrt{T}$, up to increasing $\eta_A$.
	The maximum principle implies therefore the second inequality. The first inequality can be proved along the same lines. 
\end{proof}
As already mentioned, Theorem \ref{thm-control beyond diffusive scale-n} is valid only when the initial datum is compactly supported, while need to understand what happens when it is of a more general type. For well spread initial data, the answer is given by the following counterpart of Theorem \ref{thm-control beyond diffusive scale-n}. Note that the scaling of the initial datum is slightly different than in Proposition \ref{p3.1} because, there, we wished to observe the evolution of $e^{-t\cI_*}$ on a time scale $\ll s$. Here, we want here to estimate its evolution over the whole range of times.
\begin{theorem}\label{t3.10}
For all $s\geq0$, let $\bnu^*(s,.)$ be a smooth function of the variable $\zeta\in\R^N$ satisfying  $\bnu^*(\zeta)\ge0$ for $\{\zeta\in\rn|x\cdot n\ge 0\}$ and
	$\bnu^*\big({\mathcal{R}}_n(\zeta)\big)=-\bnu^*(\zeta)$ for $\zeta\in\R^N$. Assume in addition the existence of $\vartheta\in(0,\di\frac14)$ and $A>0$ such that $\bnu^*(s,.)$ has the following decomposition: 
	$$
	\bnu^*(s,\zeta)=\bnu_1^*(s,\zeta)+\frac{\chi^*(\zeta)e^{-A\vert\zeta\vert}}{s^{\frac12-\vartheta}},
	$$
	the function $\bnu_1^*(s,.)$ satisfying $\bnu_1^*\big({\mathcal{R}}_n(\zeta)\big)=-\bnu_1^*(\zeta)$ for $\zeta\in\R^N$ and being, as well as all its derivatives,  by a multiple of $e^{-\mu\vert\zeta\vert^2}$ that is uniformly bounded in $s$. The function $\chi$ is smooth and supported in the interval $[\eta,+\infty)$ with $\eta>0$. 	Set
	$$
	v_s^*(x)=\frac1{s^{\frac{N+1}2}}\bnu^*(s,\frac{x}{\sqrt s}).
	$$
	Let $w^*$ be the solution to the linear equation \eqref{w-eqn-linear} in $\R_+\times\R^N$ associated with the initial datum $v_s$. Then we have
	\begin{equation*}
		e^{tA_n}v_s^*(x)-t^{-\frac{N}{2}-1+\vartheta}e^{-\frac{A|x|}{\sqrt{t}}} \leq w^*(t,x)\le e^{tA_n}v_s^*(x)+t^{-\frac{N}{2}-1+\vartheta}e^{-\frac{A|x|}{\sqrt{t}}}
	\end{equation*}
	for $t\ge 0$ and $x\in\R^N$.
\end{theorem}
\begin{proof} The argument follows the line of Theorem \ref{thm-control beyond diffusive scale-n} and uses Proposition \ref{p3.1} to control $w^*(t,x)$ for $t\leq s^\varsigma$.
\end{proof}


\section{Upper and lower barriers along the minimizing direction}\label{sec4}
In this section, we aim to establish refined super- and subsolutions in order to capture the correct order of magnitude of  $v(t,x)$ ahead of the front. Given any direction $e\in\sn$, let $w^*(e)$ be the invasion speed along the direction $e$. Hereafter,  we assume without loss of generality that $n_e=(1,0,\ldots,0)$. Then, the vector $\bm_{n_e}^\perp$ takes the form $\bm_{n_e}^\perp=(0,m_2,\cdots,m_N)=:\bm'$. For notational simplicity, we also set
\begin{equation*}
	c^*(n_e)=c^*, ~~~\ck_{*,n_e}=\ck_{*},~~~\lambda^*(n_e)=\lambda^*,~~~\cI_{*,n_e}=\cI_{*,n}.
\end{equation*}
For $x\in\R^N$, we denote it as $x=x_1(1,0,\ldots,0)+x'$ with $x'=(0,x_2,\cdots,x_N)\in\R^{N}$. 

We then recall from Section \ref{sec2.3} that 
\begin{equation*}
	c^*=\int_{\R^N} x_1\ck_{*}(x)\md x,~~~~\bm'=\int_{\rn} x' \ck_{*}(x)\md x.
\end{equation*}
The matrix $H_n$ introduced in \eqref{def_H matrix} is now denoted by $H_{e_1}$, and its entries $h_{ij,e_1}$. While this is not the lightest notation, it has the merit of recalling where this matrix comes from.
The function $v$ satisfies
\begin{equation}
	\label{v-eqn}
	\mathcal{N} v:=	v_t+\cI_{*}v
	+R(t,x;v)=0,~~~t>0,~x\in\R^N,
\end{equation}
with  $\cI_{*}$ now written as
\begin{equation*}
	\begin{aligned}
		\cI_{*}v=\widehat{\ck_{*}}(0)v-\ck_{*}*v-c^*(n_e)\partial_{x_1}v-\bm'\cdot\nabla v=-\frac{1}{2}\sum_{i,j=1}^{N}h_{ij,e_1}\partial_{x_ix_j}v+O\big(\Vert D^3v\Vert_{L^\infty}\big)
	\end{aligned}
\end{equation*}
thanks to \eqref{e5.103}, and
\begin{equation*}
	R(t,x;v)=f'(0)v-e^{\lambda^*x_1}f\big( e^{-\lambda^*x_1}v
	\big),
\end{equation*}
starting from $v(0,x)=e^{\lambda^*x_1}u(0,x)$.

\vskip 3mm

We shall construct upper and lower barriers for the function $v(t,x)$ ahead of $x_1\approx 0$ for $t$ sufficiently large. To this end, we begin with introducing some parameters.  Let $T$, $A$ and $\eta_A$ be as stated in Theorem \ref{thm-control beyond diffusive scale-n}. 
Moreover, we fix positive parameters $\delta$, $\beta$, $\alpha$, $\gamma$ and  $\vartheta$   such that
\begin{equation}
	\label{parameters-n}
	0<\delta<\gamma<\beta<\frac{4}{25}<\frac{7}{15}<\alpha<\frac{1}{2},~~~~~\beta<\vartheta<1-2\alpha+\beta.
\end{equation}
Fix $0<\eta_A<\eta_1<\eta_2<\eta_3<\eta_4<+\infty$ 
, and define smooth functions $\chi_1$ and $\chi_2$ on $[0,+\infty)$ such that
\begin{itemize}
	\item $r\mapsto \chi_1(r)$ is nondecreasing such that $\chi_1(r)\equiv 0$ for $r\in[0,\eta_1]$, and $\chi_1(r)\equiv 1$ for $r\in[\eta_2,+\infty)$, with bounded derivatives of all orders for $r\in[\eta_1,\eta_2]$;
	
	\item $r\mapsto \chi_2(r)$ is nonincreasing such that $\chi_2(r)\equiv 1$ for $r\in[0,\eta_3]$, and $\chi_2(r)\equiv 0$ for $r\in[\eta_4,+\infty)$, with bounded derivatives of all orders for $r\in[\eta_3,\eta_4]$;
	
\end{itemize}

Moreover, we denote by $w^*$ the solution to \eqref{w-eqn-linear} in $\R_+\times\R^N$
associated with a nontrivial, bounded and compactly supported  smooth initial datum $w^*_0$ in $\R^N$ satisfying 
$w^*_0(x)\ge0$ for $\{x\in\rn|x_1\ge 0\}$ and 
$w^*_0(\mathcal{R}_{e_1}{e_1}(x))=-w^*_0(x)$ for $x\in\R^N$, with $\mathcal{R}_{e_1}(x):=x-2(x_1/h_{11,e_1})H_{e_1}$.  Then, up to increasing $T>0$, it is worth observing from Theorem \ref{thm-control beyond diffusive scale-n} that  the function $w^*$ has the following asymptotics 
\begin{equation}
	\label{w-asympt-n}
	w^*(t,x)\asymp \frac{x_1}{t^{\frac{N}{2}+1}}
\end{equation}
 for $t\ge T$ and for $|x_1|= O(\sqrt{t})$
and $|x'|= O(\sqrt{t})$.

\subsection{The upper barrier} 
For $t\ge T$ and $x_1\ge-t^\delta-R$, we define
\begin{equation}
	\label{upper-n}
	\overline v(t,x)=
	\overline\xi(t)w^*(t,x)+\mathcal{V}_{1}(t,x)+\mathcal{V}_{2}(t+T,x),
\end{equation}
in which $\overline \xi(t)=1-\frac{1}{t^{\gamma}}$,
\begin{equation*}
	\begin{aligned}
	\cV_{1}(t,x)&=\chi_2\left(\frac{|x'|}{\sqrt{t}}\right)\frac{1}{t^{\frac{N}{2}+1-\beta}}\cos\left(\frac{x_1}{t^\alpha}\right)\mathbbm{1}_{\big\{x\in\R^N\big|-t^{\delta}-R\le x_1\le \frac{3\pi}{2}t^{\alpha}\big\}},\\
	\cV_{2}(t,x)&=\chi_1\left(\frac{	|x| }{\sqrt{t}}\right)\frac{1}{t^{\frac{N}{2}+1-\vartheta}}e^{-A\big(\frac{	|x| }{\sqrt{t}}-\eta_A\big)}.~~~~~~~~~~~~~
	\end{aligned}
\end{equation*}

In what follows, we shall check that $\overline v$ is a  supersolution to equation \eqref{v-eqn} for $t\ge T$ and $x_1\ge -t^\delta$, for which will be sufficient if  $\overline v$ satisfies $\mathcal{N}\overline v\ge 0$
for   $t\ge T$ and $x_1 \ge -t^\delta$, thanks to  the comparison principle.

Before proceeding, let us explore each term in our target $\overline v$. First,  the term $\overline\xi(t) w^*(t,x)$, as expected,  is a good candidate for the supersolution within the diffusive regime. As a matter of fact, one finds that 
$	\displaystyle\mathcal{N}\big(\overline \xi(t)w^*(t,x)\big)=\overline\xi'(t)w^*(t,x)	
	\approx \frac{C}{t^{1+\gamma}}\frac{x_1}{t^{\frac{N}{2}+1}}
$for $t\ge T$ and for $|x_1|= O(\sqrt{t})$  and $|x'|=O(\sqrt{t})$. Regarding the more involved correction terms $\mathcal{V}_1$ and $\mathcal{V}_2$,  explicit computations are listed below. 

For the sake of notational simplification,  we adopt the following set of notation
\begin{equation*}\label{short notation-n}
	\zeta_\alpha:=\frac{x_1}{t^\alpha},~~~~\zeta'=\frac{x'}{\sqrt{t}},~~~~\zeta:=\frac{x}{\sqrt{t+T}},
\end{equation*}
and
$	\chi_1:=\chi_1(\vert\zeta\vert),~~\chi_2:=\chi_2(\vert\zeta'\vert).
$
Whenever $\mathcal{V}_1(t,x)$ is nontrivial,  a straightforward calculation yields
\begin{equation*}\label{V1-diff-t}
	t^{\frac{N}{2}+1-\beta}\partial_t\cV_{1}(t,x)=-\frac1t\big(\frac{N}{2}+1-\beta\big)\chi_2\cos(\zeta_\alpha)+\frac{\alpha\zeta_\alpha}t\chi_2\sin(\zeta_\alpha)-\frac{\vert\zeta'\vert}{2t}\chi_2'\cos(\zeta_\alpha)
\end{equation*}
and
\begin{equation*}
	\label{V1-diff-x}
\begin{aligned}
	t^{\frac{N}{2}+1-\beta}\partial^2_{x_1x_1}\cV_{1}(t,x)
	&=-\frac{1}{t^{2\alpha}}\chi_2\cos(\zeta_\alpha),\\
	t^{\frac{N}{2}+1-\beta}\partial^2_{x_1x_j}\cV_{1}(t,x)
	&=-\frac{1}{t^{\alpha+\frac{1}{2}}}\frac{x_j}{|x'|}\chi_2'\sin(\zeta_\alpha),~~~~j\neq 1,\\
	t^{\frac{N}{2}+1-\beta}\partial^2_{x_ix_j}\cV_{1}(t,x)
	&=\frac{1}{t^{\frac{1}{2}}}\Big(\frac{\delta_{ij}}{|x'|}-\frac{x_ix_j}{|x'|^3}\Big)\chi_2'\cos(\zeta_\alpha)+\frac{1}{t}\frac{x_ix_j}{|x'|^2}\chi_2''\cos(\zeta_\alpha),~~~i\neq 1,~j\neq 1.
\end{aligned}
\end{equation*}
Let us also notice  that, in the support of $\chi_2'$, the terms $\delta_{ij}/{\vert x'\vert}$ and ${x_i'x_j'}/{\vert x'\vert^3}$ are of the order $1/t^{1/2}$.
 Furthermore, the computation of the third order spatial derivatives 
 leads to
\begin{equation*}
	\Vert D^3\cV_1\Vert_{L^\infty}
	=O(t^{-3\alpha}).
\end{equation*}
Consequently, whenever $\mathcal{V}_1$ is nontrivial, we arrive at
\begin{equation}
	\label{eqn-V1}
	t^{\frac{N}2+1-\beta}\cN\cV_1\geq \bigg(\frac{h_{11,e_1}}{2t^{2\alpha}}-\frac{\frac{N}{2}+1-\beta}{t}\bigg)\chi_2\cos(\zeta_\alpha)+O\bigg(\frac1{t^{\alpha+\frac{1}{2}}}\bigg).
\end{equation}

Let us now turn to the computation for $\cV_2$ by focusing on the region where $\cV_2$ is supported, that is when $\vert\zeta\vert> \eta_1>\eta_A$. We have
\begin{align*}
	(t+T)^{\frac{N}2+1-\vartheta}e^{A(\vert\zeta\vert-\eta_A)} 	\partial_t\cV_2(t+T,x)=	
	\frac{A\vert\zeta\vert-N-2(1-\vartheta)}{2(t+T)}\chi_1 - \frac{\vert\zeta\vert}{2(t+T)}\chi_1'.
\end{align*}
As for the spatial derivatives,  we deduce that
\begin{align*}
	&(t+T)^{\frac{N}{2}+1-\vartheta}e^{A(\vert\zeta\vert-\eta_A)}	\partial_{x_ix_j}\cV_{2}(t+T,x)\\
	&~~~~~~~~~~~~~~~~~~=\frac{1}{(t+T)^{\frac{1}{2}}} \Big(\frac{\delta_{ij}}{|x|}-\frac{x_ix_j}{|x|^3}\Big)\big(\chi_1'-A\chi_1\big)+\frac{1}{t+T}\frac{x_ix_j}{|x|^2}\big(A^2\chi_1-2A\chi_1'+\chi_1''\big)\\
	&~~~~~~~~~~~~~~~~~~\le \frac{1}{t+T}\Big(\frac{2}{\eta_1}A+A^2\Big)\chi_1+\frac{1}{t+T}\Big(\frac{2}{\eta_1}|\chi_1'|+2A|\chi_1'|+|\chi_1''|\Big)
\end{align*}
and  $\Vert D^3\cV_2\Vert_{L^\infty}=O\big((t+T)^{-\frac{3}{2}}\big)$.
Noticing that $\chi_1'\neq 0$ or $\chi_1''\neq 0$ implies  $\eta_1\le \vert\zeta\vert\le \eta_2$, we get
\begin{equation}
\label{eqn-V2}
\begin{aligned}
&	(t+T)^{\frac{N}{2}+1-\vartheta}e^{A(\vert\zeta\vert-\eta_A)}	\mathcal{N}\cV_{2}(t+T,x)\\
\ge &\frac{1}{2(t+T)}\Bigg(\bigg[A\eta_A\!-\!N\!-\!2\!-\!\Big(\frac{2}{\eta_1}A+A^2\Big)\bar{h}\bigg]
	\chi_1\-\vert\zeta\vert\chi_1' \!-\!\bar{h}\bigg[\frac{2}{\eta_1}|\chi_1'|+2A|\chi_1'|+|\chi_1''|\bigg]\Bigg)\!\!+O\Big(\frac{1}{(t+T)^{\frac{3}{2}}}\Big)
	\ge0,
\end{aligned}
\end{equation}
up to increasing $\eta_A$ if needed.

\medskip

We are now in a position to look at our target $\overline v$ for $t\ge T$ and $x_1\ge -t^\delta-R$.
Based upon the definition of $\overline v$, we will proceed with our analysis by dividing the region into several parts for $t\ge T$:

\medskip

\noindent
\textbf{Step 1}. The range $-t^\delta-R\le x_1\le\di \frac{3\pi}{2}t^\alpha$. 

\medskip

\noindent
\textbf{Step 1.1}.  $-t^\delta-R\le x_1\le t^\delta$. In this region, $\cos(\zeta_\alpha)$ is almost 1. 
\begin{itemize}
	\item  $0\le |x'|\le \eta_3\sqrt{t}$. We have $0\le \chi_1\le 1$ and $\chi_2\equiv 1$. The term $\cV_{2}$ is nonnegative, and the cosine term $\cV_{1}$ itself can dominate the sign of the function $\overline v$ up to increasing $T$, thanks to 
	\eqref{w-asympt-n} and $\delta<\beta$, therefore  $\overline v$ is positive. Moreover, we infer from    \eqref{w-asympt-n}, \eqref{eqn-V1} and \eqref{eqn-V2} that, up to increasing $T$, 
	\begin{equation*}
		\overline\xi'(t)w^*(t,x)+\cN\cV_{1}(t,x)+\cN\cV_{2}(t+T,x)\ge -\frac{C}{t^{\frac{N}{2}+2-\delta+\gamma}}+\frac{C}{t^{\frac{N}{2}+1-\beta+2\alpha}}-\frac{C}{(t+T)^{\frac{N}{2}+2-\vartheta}}>0.
	\end{equation*}
	
	\item  $\eta_3\sqrt{t}\le |x'|\le \eta_4\sqrt{t}$.  We observe that $\chi_1\equiv 1$, $0\le \chi_2\le 1$, and the cosine term is nonnegative, thus  $\cV_{2}(t+T,x)\ge C(t+T)^{\frac{N}{2}+1-\vartheta}$ with $\vartheta>\delta$, which together with \eqref{w-asympt-n} implies that $\overline v$ is positive.
	Since
	$\mathcal{N}\cV_{1}(t,x)\ge -Ct^{-(\frac{N}{2}+1-\beta)-1}$, we get 
	\begin{equation*}
		\overline\xi'(t)w^*(t,x)+\cN\cV_{1}(t,x)+\cN\cV_{2}(t+T,x)\ge -\frac{C}{t^{\frac{N}{2}+2-\delta+\gamma}}-\frac{C}{t^{\frac{N}{2}+2-\beta}}+\frac{C}{(t+T)^{\frac{N}{2}+2-\vartheta}}>0.
	\end{equation*}
	\item  $ |x'|\ge \eta_4\sqrt{t}$.  We have $\chi_1\equiv 1$ and $\chi_2\equiv 0$.  It follows from Theorem \ref{thm-control beyond diffusive scale-n} that the function   $\overline v$ is positive as well, and, up to increasing $T$,
 \begin{equation*}
 	\overline\xi'(t)\big(e^{-t\ci_*}w_0^*\big)(x)+\cN\cV_2(t+T,x)\ge \frac{C}{t^{1+\gamma}}\big(e^{-t\ci_*}w_0^*\big)(x)+\frac{C}{t+T}\cV_2(t+T,x)>0.
 \end{equation*}	
\end{itemize}

\medskip 

\noindent
\textbf{Step 1.2}.  $t^\delta\le x_1\le \frac{\pi}{4}t^\alpha$.  In this region,   $\overline v>0$, since each term is either positive or nonnegative. Furthermore,
\begin{itemize}
	\item  $0\le |x'|\le \eta_3\sqrt{t}$. We find that $0\le \chi_1\le 1$ and $\chi_2\equiv 1$. Combining $\overline\xi'(t)w^*(t,x)>0$ and  \eqref{w-asympt-n}, one has that, up to increasing $T$, 
	\begin{equation*}
		\overline\xi'(t)w^*(t,x)+\cN\cV_{1}(t,x)+\cN\cV_{2}(t+T,x)\ge \frac{C}{t^{\frac{N}{2}+1-\beta+2\alpha}}-\frac{C}{(t+T)^{\frac{N}{2}+2-\vartheta}}>0.
	\end{equation*}
	
	\item  $\eta_3\sqrt{t}\le |x'|\le \eta_4\sqrt{t}$.  Here we have $\chi_1\equiv 1$, $0\le \chi_2\le 1$, and $\mathcal{N}\cV_{1}(t,x)\ge -Ct^{-(\frac{N}{2}+1-\beta)-1}$. Together with  $\overline\xi'(t)w^*(t,x)>0$ and $\vartheta>\beta$, we have 
	\begin{equation*}
		\overline\xi'(t)w^*(t,x)+\cN\cV_{1}(t,x)+\cN\cV_{2}(t+T,x)\ge -\frac{C}{t^{\frac{N}{2}+2-\beta}}+\frac{C}{(t+T)^{\frac{N}{2}+2-\vartheta}}>0.
	\end{equation*}
	
	\item  $ |x'|\ge \eta_4\sqrt{t}$.  From $\chi_1\equiv 1$ and $\chi_2\equiv 0$ (namely, $\cV_{1}\equiv 0$), we  deduce from Theorem \ref{thm-control beyond diffusive scale-n} that, up to increasing $T$,
	\begin{equation*}
		\overline\xi'(t)w^*(t,x)+\cN\cV_{2}(t+T,x)\ge \frac{C}{t^{1+\gamma}}w^*(t,x)+\frac{C}{t+T}\cV_{2}(t+T,x)>0.
	\end{equation*}
	
\end{itemize}

\medskip

\noindent
\textbf{Step 1.3}.  $\frac{\pi}{4}t^\alpha \le x_1\le \frac{3\pi}{2}t^{\alpha}$.  Note that the cosine term $\cV_{1}$ may have negative sign in this regime, and the term $\overline\xi(t)w^*(t,x)$ starts playing a role.
\begin{itemize}
	
	\item $0\le |x'|\le \eta_4\sqrt{t}$. We derive from \eqref{w-asympt-n} that $\overline\xi(t)w^*(t,x)\ge Ct^{-(\frac{N}{2}+1)+\alpha}> Ct^{-(\frac{N}{2}+1)+\beta}\ge \cV_{1}$, thanks to \eqref{parameters-n}. This, together with the fact that $\cV_{2}$ is nonnegative, implies that $\overline v$ is positive. Moreover, since $0\le \chi_1,\chi_2\le 1$, along with \eqref{w-asympt-n}, as well as  $\mathcal{N}\cV_{1}(t,x)\ge -Ct^{-(\frac{N}{2}+1-\beta)-2\alpha}$ from \eqref{eqn-V1}, and $\vartheta<1-2\alpha+\beta<\alpha-\gamma$, we have
	\begin{align*}
		\overline\xi'(t)w^*(t,x)+\cN\cV_{1}(t,x)+\cN\cV_{2}(t+T,x)\ge \frac{C}{t^{\frac{N}{2}+2-\alpha+\gamma}}-\frac{C}{t^{\frac{N}{2}+1-\beta+2\alpha}}-\frac{C}{(t+T)^{\frac{N}{2}+2-\vartheta}}>0.
	\end{align*}

	\item $\eta_4\sqrt{t}\le |x'|$.   We have
	\begin{equation}
		\label{bar v>0}
		\begin{aligned}
			\overline v(t,x)&=\overline\xi(t)w^*(t,x)+\cV_{2}(t+T,x)\\
			&=\overline\xi(t)e^{(t+T)A_n}w_0^*(x)+\overline\xi(t)\big(w^*(t,x)-e^{(t+T)A_n}w_0^*(x)\big)+\cV_{2}(t+T,x)>0
		\end{aligned}
	\end{equation} 
	 by Theorem \ref{thm-control beyond diffusive scale-n}, this also confirms that $\overline v$ is positive.  In addition, we  have
	  \begin{equation}
	 	\label{eqn>0}
	 	\begin{aligned}
	 		&\overline\xi'(t)w^*(t,x)+\cN\cV_{2}(t+T,x)\\
	 		&~~~~~~~\ge \overline\xi'(t)e^{(t+T)A_n}w_0^*(x)+\frac{\gamma}{t^{1+\gamma}}\big(w^*(t,x)-e^{(t+T)A_n}w_0^*(x)\big)  +\frac{C}{t+T}\cV_{2}(t+T,x)>0.
	 	\end{aligned}
	 \end{equation}
\end{itemize}

\medskip

\noindent
\textbf{Step 2}. The range  $\frac{3\pi}{2}t^{\alpha} \le x_1\le \eta_2\sqrt{t}$. 
In this region, it is observed that $\cV_{1}(t,x)\equiv 0$. 
\begin{itemize}
	
	\item $0\le |x'|\le \eta_2\sqrt{t}$.  It is easily seen from \eqref{w-asympt-n} and the nonnegativity of $\cV_{2}$ that $\overline v$ is positive. Since $0\le \chi_1\le 1$, it immediately follows from \eqref{w-asympt-n} and $\vartheta<\alpha-\gamma$ that
	\begin{align*}
		\overline\xi'(t)w^*(t,x)+\cN\cV_{2}(t+T,x)\ge \frac{C}{t^{\frac{N}{2}+2-\alpha+\gamma}}-\frac{C}{(t+T)^{\frac{N}{2}+2-\vartheta}}>0.
	\end{align*}
	
	\item $\eta_2\sqrt{t}\le |x'|$. Then $\chi_1\equiv 1$.  From Theorem \ref{thm-control beyond diffusive scale-n}, we have $\overline v\geq0$, and up to increasing $T$,
	\begin{align*}
		\overline\xi'(t)w^*(t,x)+\cN\cV_{2}(t+T,x)\ge \frac{C}{t^{1+\gamma}}w^*(t,x) +\frac{C}{t+T}\cV_{2}(t+T,x)>0.
	\end{align*}
	
\end{itemize}

\medskip

\noindent
\textbf{Step 3}.  The range  $ x_1\ge \eta_2\sqrt{t}$. 

\medskip

Here, $\chi_1\equiv 1$. The same reasoning as in \eqref{bar v>0} and \eqref{eqn>0} gives that $ \overline v(t,x)=\overline\xi(t)w^*(t,x)+\cV_{2}(t+T,x)>0$,
and $\overline\xi'(t)w^*(t,x)+\cN\cV_{2}(t+T,x)>0$ up to increasing $T$.

\medskip

\noindent  
\textbf{Conclusion}.
We conclude that the function $\overline v(t,x)$ defined in \eqref{upper-n} is a positive function for $t\ge T$ and for $x_1\ge -t^\delta-R$ and satisfies $\cN\overline v\ge 0$. Thus, it is a supersolution to the nonlinear  problem \eqref{v-eqn} for $t\ge T$ and $x_1\ge -t^\delta-R$.

Since $v(0,x)$ is compactly supported for $x\in\R^N$, one can choose $\kappa_+>\Vert e^{\lambda^*x_1}u_0(x)\Vert_{L^\infty}$ large such that $v(0,x)\le\kappa_+ \overline v(T,x)$ for all $x_1\ge -T^\delta-R$. 
Pick $\varpi>0$ small; we will now compare $\kappa_+\overline v(t,x)$ to $v(t-T,x)$ on the set 
 $$
\{-t^\delta-R\leq x_1,~\vert x'\vert\geq t^{\frac12+\varpi}\}\cup\{-t^\delta-R\leq x_1\leq-t^\delta,~x'\in\RR^{N-1}\}.
 $$
 If $v(t-T,x)\leq\kappa_+\overline v(t,x)$ on that set, then the two functions compare everywhere on the set $\{x_1-t^\delta-R\}$. On the one hand, notice that  $\cV_2$ plays a  dominant role in  $\overline v(t,x)$   for  $t\ge T$ and for $x_1 \in[ -t^\delta-R, -t^\delta]$ and $|x'|\ge \eta_2\sqrt{t}$. Hence, by picking $A>1$ in this theorem and choosing $B>A$ in Corollary \ref{c3.10}, we deduce from these two results that $\kappa_+\overline v(t,x)\ge v(t-T,x)$  for  $t\ge T$,  $x_1 \in[ -t^\delta-R, -t^\delta]$ and $|x'|\ge t^{\frac12+\varpi}$. 
On the other hand,
 we derive from the definition of $v$ and from $0\le \bu(t,x)\le 1$ for $t\ge 0$ and $x\in\R^N$ that, up to increasing $T$,
\begin{equation*}
	v(t-T,x)=e^{\lambda^*x_1}\bu(t-T,x)\le  e^{-\lambda^*t^\delta}, ~~~~t\ge T,~x_1 \in[ -t^\delta-R, -t^\delta],
\end{equation*}
As it is  seen from the definition \eqref{upper-n} of $\overline v$ that $	\overline v(t,x)\gtrsim t^{-(\frac{N}{2}+1-\beta)}e^{-At^{\varpi}}$  for  $t\ge T$ and for $x_1 \in[ -t^\delta-R, -t^\delta]$ and $|x'|\le t^{\frac12+\varpi}.$
Consequently, up to increasing $T$, we have $v(t-T,x) \le\kappa_+ \overline v(t,x)$ for  $t\ge T$ and for $x_1 \in[ -t^\delta-R, -t^\delta]$ and $\vert x'\vert\leq t^{\frac12+\varpi}$ if $\varpi$ is small enough. As we also have $v(t-T,x)\lesssim t^{-\frac{N}2+\varsigma}e^{-\frac{B\vert x\vert}{\sqrt t}}$  from Corollary \ref{c3.10}, we have $v(t-T,x)\leq\kappa_+\overline v(t,x)$ for $-t^\delta-R\leq x_1\leq-t^\delta$ and $\vert x'\vert\geq t^{\frac12+\varpi}$, as soon once again as we have chosen $A>B$. It then follows from the comparison principle  that
\begin{equation*}\label{upper-final-2}
	v(t-T,x)\le \kappa_+ \overline v(t,x)
\end{equation*}
for $t\ge T$ and $x_1\ge -t^\delta$.

\subsection{The lower barrier}

For $t\ge T$ and $x_1\ge t^\delta-R$, we define
\begin{equation}  
	\label{lower-n}
	\underline{v}(t,x)=\max \bigg(
	\underline\xi(t)w^*(t,x)-\mathcal{V}_{3}(t,x)-\mathcal{V}_{2}(t+T,x), 0\bigg)
\end{equation}
in which $\underline \xi(t)=1+\frac{1}{t^{\gamma}}$,  $\cV_{2}(t,x)$ is given as in the definition \eqref{upper-n} of the supersolution $\overline v$, and
\begin{equation*}
	\cV_{3}(t,x)=\chi_2\left(\frac{|x'|}{\sqrt{t}}\right)\frac{1}{t^{\frac{N}{2}+1-\beta}}\cos\left(\frac{x_1}{t^\alpha}\right)\mathbbm{1}_{\big\{x\in\R^N\big| t^{\delta}-R\le x_1\le \frac{3\pi}{2}t^{\alpha}\big\}}.
\end{equation*}

Our goal now is to check that $\underline v$ is a  subsolution to equation \eqref{v-eqn} for $t\ge T$ and $x_1\ge t^\delta-R$, namely, $\cN\underline v+R(t,x;\underline v)\le 0$ for $t\ge T$ and $x_1\ge t^\delta$.  The nonnegative term 
$R(t,x;\underline v)$ shall be taken into account this time, and it satisfies
\begin{equation}\label{R-cdn-n}
	R(t,x;\underline v)=f'(0)\underline v-e^{\lambda^*x_1}f\big(e^{-\lambda^*x_1}\underline v\big)\le e^{-\lambda^*x_1}M\underline v^2
\end{equation}
whenever $\underline v$ is not trivial so that $e^{-\lambda^*(n)(x_1)}\underline v(t,x)\in(0,s_0]$, thanks to \eqref{f-cdn2}. 

Based upon the construction of $\underline v$, an easy observation from Theorem \ref{thm-control beyond diffusive scale-n} and the fact that $\beta<\vartheta$ is that $$\underline \xi(t)w^*(t,x)-\mathcal{V}_{3}(t,x)-\mathcal{V}_{2}(t+T,x)<0$$
for $t\ge T$ and for $|x|\ge \eta_3\sqrt{t}$, up to increasing $T$. Therefore, it will be sufficient for us to focus on the region where $t\ge T$,  $t^\delta-R\le x_1\le \eta_3\sqrt{t}$, and $0\le |x'|\le \eta_3\sqrt{t}$.

To proceed with our analysis, we distinguish again several subregions for $t\ge T$:

\medskip

\noindent
\textbf{Step 1}. The range $t^\delta-R\le x_1\le \frac{3\pi}{2}t^\alpha$. 

\medskip

\noindent
\textbf{Step 1.1}.  $t^\delta-R\le x_1\le t^\gamma$. 
\begin{itemize}
	\item  $0\le |x'|\le \eta_3\sqrt{t}$.
	Since $\cV_{2}(t+T,x)$ is nonnegative in this region, this, together with
	\eqref{w-asympt-n} and $\gamma<\beta$, gives that
	\begin{align*}
		\underline \xi(t)w^*(t,x)-\mathcal{V}_{3}(t,x)-\mathcal{V}_{2}(t+T,x)\le \frac{C}{t^{\frac{N}{2}+1-\gamma}}-\frac{C}{t^{\frac{N}{2}+1-\beta}}\le 0,
	\end{align*}
	which implies that $\underline v(t,x)\equiv 0$ in this regime. Therefore, it is also trivial that $\mathcal{N}\underline v+R(t,x;\underline v)=0$.
\end{itemize}

\medskip

\noindent
\textbf{Step 1.2}.  $t^\gamma\le x_1\le \frac{3\pi}{2}t^\alpha$. First of all, we deduce that whenever $\underline v(t,x)=0$, the conclusion is trivial. In the following, we look at the subregion where $\underline v(t,x)>0$. We then have
$\underline v(t,x)\le Ct^{-(\frac{N}{2}+1)+\alpha}$,
thanks to \eqref{w-asympt-n} and $\beta<\alpha$. Furthermore, up to increasing $T$, we also conclude that $e^{-\lambda^*x_1}\underline v(t,x)\in(0,s_0]$, and   that $R(t,x;\underline v)\le Ce^{-\lambda^* t^{\gamma}}t^{-(N+2)+2\alpha}$, by using \eqref{R-cdn-n}. Moreover,
\begin{itemize} 
	\item $0\le |x'|\le \eta_1\sqrt{t}$. Here, $\chi_1\equiv0$, i.e. $\cV_{2}(t,x)\equiv0$, thus 	up to increasing $T$,
	\begin{align*}
		\underline\xi'(t)w^*(t,x)-\cN\cV_{3}(t,x)+R(t,x;\underline v)\le -\frac{Cx_1}{t^{\frac{N}{2}+2+\gamma}}-\frac{C}{t^{\frac{N}{2}+1-\beta+2\alpha}}+Ce^{-\lambda^* t^{\gamma}}t^{-(N+2)+2\alpha}\le 0.
	\end{align*}

	\item $\eta_1\sqrt{t}\le |x'|\le \eta_3\sqrt{t}$.  We infer from $\vartheta<1-2\alpha+\beta$ that, up to increasing $T$,
	\begin{align*}
		&\underline\xi'(t)w^*(t,x)-\cN\cV_{3}(t,x)-\cN\cV_{2}(t+T,x)+R(t,x;\underline v)\\
		&~~~~~~~~~~\le -\frac{Cx_1}{t^{\frac{N}{2}+2+\gamma}}-\frac{C}{t^{\frac{N}{2}+1-\beta+2\alpha}}+\frac{C}{(t+T)^{\frac{N}{2}+2-\vartheta}}+Ce^{-\lambda^* t^{\gamma}}t^{-(N+2)+2\alpha}\le 0.
	\end{align*}
\end{itemize}

\medskip

\noindent
\textbf{Step 2}. We now look at the region where $\frac{3\pi}{2}t^\alpha\le x_1\le \eta_3\sqrt{t}$. Here, $\cV_{3}(t,x)\equiv 0$. We follow exactly the same strategy as above. Whenever $\underline v(t,x)\neq0$, we  have
$\underline v(t,x)\le Ct^{-(\frac{N}{2}+1)+\frac{1}{2}}$,
thanks to \eqref{w-asympt-n}. Furthermore, up to increasing $T$, we also deduce that $e^{-\lambda^*x_1}\underline v(t,x)\in(0,s_0]$, and   that $R(t,x;\underline v)\le Ce^{-\lambda^* t^{\alpha}}t^{-(N+2)+1}$, by using \eqref{R-cdn-n}. Moreover,
\begin{itemize} 
	\item $0\le |x'|\le \eta_1\sqrt{t}$. In this region, $\cV_{2}(t+T,x)\equiv0$. Up to increasing $T$,
	\begin{align*}
		\underline\xi'(t)w^*(t,x)+R(t,x;\underline v)\le -\frac{C}{t^{\frac{N}{2}+2-\alpha+\gamma}}+Ce^{-\lambda^* t^{\alpha}}t^{-(N+2)+1}\le 0.
	\end{align*}

	\item $\eta_1\sqrt{t}\le |x'|\le \eta_3\sqrt{t}$.  We infer from $\vartheta<\alpha-\gamma$ that, up to increasing $T$,
	\begin{align*}
		\underline\xi'(t)w^*(t,x)-\cN\cV_{2}(t+T,x)+R(t,x;\underline v)\le -\frac{C}{t^{\frac{N}{2}+2-\alpha+\gamma}}+\frac{C}{(t+T)^{\frac{N}{2}+2-\vartheta}}\!+\!Ce^{-\lambda^* t^{\alpha}}t^{-(N+2)+1}\le 0.
	\end{align*}
\end{itemize}

\vspace{2mm}

\noindent  
\textbf{Conclusion}.
We conclude that the function $\underline v(t,x)$ defined in \eqref{lower-n} satisfies $\cN\underline v+R(t,x;\underline v)\le 0$ for $t\ge T$ and $x_1\ge t^\delta-R$. Thus, it is a subsolution to the nonlinear  problem \eqref{v-eqn} for $t\ge T$ and $x_1\ge t^\delta-R$.

\medskip
On the other hand, since $v(T,x)$ is positive everywhere for $x\in\R^N$ whereas $\underline v(T,x)$ has compact support in the domain $x_1\ge T^\delta-R$. Therefore, there is $\kappa_->0$ small such that  $\kappa_-\underline v(0,x)\leq v(T,x)$ for $x_1\ge T^\delta-R$.
Moreover,  it is easily seen from Step 1.1 that $\kappa_-\underline v(t,x)=0<v(t,x)$ for $t\ge T$ and $t^\delta-R\le x\le t^\delta$,
up to increasing $T$. It then follows from the comparison principle  that
\begin{equation*}\label{lower-final-n}
	\kappa_-\underline v(t,x)\le  v(t,x)
\end{equation*}
for $t\ge T$ and $x_1\ge  t^\delta$.

\subsection{Conclusion}

Based on the upper and lower barriers previously established, we conclude that  for $t\ge T$,
\begin{equation}\label{conclusion-super+sub}
	\begin{aligned}
	 v(t,x)\ge \kappa_-\underline v(t,x)\ge &\kappa_-	\Bigg\{\left(1+\frac{1}{t^{\gamma}}\right)w^*(t,x)-\chi_2\left(\frac{|x'|}{\sqrt{t}}\right)\frac{1}{t^{\frac{N}{2}+1-\beta}}\cos\left(\frac{x_1}{t^\alpha}\right)\mathbbm{1}_{\big\{t^{\delta}\le x_1\le \frac{3\pi}{2}t^{\alpha}\big\}}\\
	&~~~~~~~~~~~-\chi_1\left(\frac{	|x| }{\sqrt{t+T}}\right)\frac{1}{(t+T)^{\frac{N}{2}+1-\vartheta}}e^{-A\big(\frac{	|x| }{\sqrt{t+T}}-\eta_A\big)}\Bigg\},~~~~x_1\ge t^\delta;\\
	v(t-T,x)\le \kappa_+\overline v(t,x)\le &\kappa_+\Bigg\{
\left(1-\frac{1}{t^{\gamma}}\right)w^*(t,x)+\chi_2\left(\frac{|x'|}{\sqrt{t}}\right)\frac{1}{t^{\frac{N}{2}+1-\beta}}\cos\left(\frac{x_1}{t^\alpha}\right)\mathbbm{1}_{\big\{-t^{\delta}\le x_1\le \frac{3\pi}{2}t^{\alpha}\big\}}\\
	&~~~~~~~~~~~+\chi_1\left(\frac{	|x| }{\sqrt{t+T}}\right)\frac{1}{(t+T)^{\frac{N}{2}+1-\vartheta}}e^{-A\big(\frac{	|x| }{\sqrt{t+T}}-\eta_A\big)}
	\Bigg\},~~~~x_1\ge -t^\delta.
			\end{aligned}
\end{equation}

\section{Convergence ahead of the front}\label{sec5}
In this section, we continue to use the notation and setup introduced in the preceding section.  The consequence of \eqref{conclusion-super+sub} is that there are two positive constants $\kappa_-<\kappa_+$, a small constant $\vartheta\in(0,\frac{1}{2})$  and two positive large constants $A$ and $\Lambda$ such that for all $t$ sufficiently large and for $x_1\geq t^\delta$, 
\begin{equation}
\label{e6.2}
\kappa_-\frac{x_1e^{-\frac{x^TH_{e_1}x}{2t}}}{t^{\frac{N}2+1}}-\frac{\Lambda \chi_1\Big(\frac{\vert x\vert}{\sqrt{t}}\Big)    e^{-\frac{A\vert x\vert}{\sqrt{t}}}}{t^{\frac{N}2+1-\vartheta}}   \leq v(t,x)\leq\kappa_+\frac{x_1e^{-\frac{x^TH_{e_1}x}{2t}}}{t^{\frac{N}2+1}}+  \frac{\Lambda \chi_1\Big(\frac{\vert x\vert}{\sqrt{t}}\Big) e^{-\frac{A\vert x\vert}{\sqrt{t}}}}{t^{\frac{N}2+1-\vartheta}}. 
\end{equation}
The goal of this section is to upgrade this inequality into $\kappa_+=\kappa_-$. 

To state the version of the result that we shall prove, we introduce the following objects and notations.
In agreement with  \eqref{e3.4} and \eqref{e3.3001}-\eqref{e3.3002},  let us set 
\begin{equation}
\label{e5.2000}
\mathcal{G}^{app}_{e_1}(t,x)=\frac{e^{-\frac{x^T H_{e_1}^{-1}x}{2t}}}{\sqrt{\det H_{e_1}}(2\pi t)^{\frac{N}{2}}},~~~~
\Gamma_{e_1}(\zeta)=\frac{2\zeta_1e^{-\frac{\zeta^T H_{e_1}^{-1} \zeta}{2}}}{e_1^TH_{e_1} e_1 \sqrt{\det H_{e_1} }(2\pi )^{\frac{N}{2}}}.
\end{equation}

Given a continuous compactly supported function  $w_0$,
we derive from Remark \ref{rmk_Dirichlet heat kernel} that for  $t$ sufficiently large and $x\in\R^N$,
\begin{equation}
\label{e5.2001}
e^{tA_{e_1}}w_0(x)=\frac{\mu_{e_1}[w_0]}{t^{\frac{N+1}2}}\Gamma_{e_1}\Big(\frac{x}{\sqrt t}\Big)+
O\bigg(\chi_1\Big(\frac{\vert x\vert}{\sqrt{t}}\Big)\frac{ e^{-A\frac{\vert x\vert}{\sqrt t}}}{t^{\frac{N}2+1-\vartheta}}\bigg),~~~\text{with}~~\mu_{e_1}[w_0]=\int_{x_1>0}x_1w_0(x)\md x.
\end{equation}
 Although conservative, this estimate is sufficient for our purposes.
The aim of this section is the following theorem.
\begin{theorem}\label{t6.1}
There exists $\kappa_\infty>0$ and two constants $\vartheta>0$ and $A>0$  such that
\begin{equation}
\label{e6.1}
v(t,x)=\frac{\kappa_\infty+o_{t\to+\infty}(1)}{t^{\frac{N+1}2}}\Gamma_{e_1}\Big(\frac{x}{\sqrt t}\Big)+
O\bigg(\chi_1\Big(\frac{\vert x\vert}{\sqrt{t}}\Big)\frac{ e^{-A\frac{\vert x\vert}{\sqrt t}}}{t^{\frac{N}2+1-\vartheta}}\bigg)
\end{equation}
for all $t$ sufficiently large and $x_1\geq t^\delta$.
\end{theorem}
\subsection{An intuitive overview of the proof of Theorem \ref{t6.1}}
Since the proof is rather technical, it is useful first to gain some intuition by taking a closer look at \eqref{conclusion-super+sub}. A first observation is that, in the range $\{t^\delta\leq x_1\leq t^\alpha\}$   when $t$ sufficiently large, 
\begin{equation*}
	v(t,x)\le C\max\bigg(\frac{x_1}{t^{\frac{N}{2}+1}},\frac1{t^{\frac{N}2+1-\beta}}\bigg).
\end{equation*} 
For $x_1\geq t^\delta$, let us recall from  \eqref{v-eqn} that $v$ satisfies 
$$
0=v_t+\cI_{*,n}v+O(e^{-\lambda^*x_1}v^2)=v_t+\cI_{*,n}v+O(e^{-\lambda^*t^\delta}v^2)\approx v_t+\cI_{*,n}v.
$$
Instead of taking $t_0=0$ as the starting time, we take $t_0=s$ with $s>0$ sufficiently large. As $e^{-(t-s)\cI_*}$ closely resembles  $e^{(t-s)A_{e_1}}$ as soon as $t-s$ is large, it is  tempting to expect that in the regime
 $s\to+\infty$ and $t-s\to+\infty$, $v(t,x)$ is well approximated by the solution $w^{app}_s(t,x)$ of the Dirichlet heat equation
 \begin{equation*}
 	\begin{aligned}
 		\begin{cases}
 			\displaystyle	\partial_t w^{app}_s-\displaystyle\frac12\mathrm{div}\big(H_{e_1}\nabla w^{app}_s\big)=0,~~~&t>s,~~x\in \R^N_+=\{x~|~x_1\ge 0,~~x'\in\R^{N-1}\},\\
 		\displaystyle	w^{app}_s(t,x)=0,~~~&t>s,~~x\in\partial\R^{N}_+,\\
 		\displaystyle	w^{app}_s(s,x)=v(s,x),~~~&x\in \R^N_+.
 		\end{cases}
 	\end{aligned}
 \end{equation*}
This intuition is further supported by a second observation, suggested by \eqref{conclusion-super+sub} or \eqref{e6.2}. Heuristically, the characteristic scale on which $v(t,\cdot)$ varies is the diffusive scale  $x_1\sim\sqrt t$. If this were rigorous, then the computation \eqref{e5.103} would provide even stronger evidence that $v(t,x)$ is close to $w^{app}_s(t,x)$ for $t\geq s$. In fact, if the operator appearing in the landscape were $A_{e_1}$ instead of $\cI_*$, these considerations, added to the result of Section 4, would almost amount to a proof that we can upgrade $\kappa_-$ and $\kappa_+$ to the same constant $\kappa$. This is essentially the approach taken, for instance, in \cite{RRR}.

In order to follow this path, we
consider any large $s>0$ and modify the super- and sub-solutions defined in \eqref{upper-n} and \eqref{lower-n} by replacing the initial datum $v_0$ with $v(s,\cdot)$.  We hope that this procedure yields a family of new barriers that become increasingly accurate as $s\to+\infty$.
However, it is important to keep in mind that 
 $v_0$ is not replaced  by $v(s,\cdot)$ itself, 
  but  by the odd reflection  $v^*(s,\cdot)$ of the restriction of $v(s,\cdot)$ to $\R^N_+$. 
   As we have no explicit expression of $e^{-t\cI_*}$ under Dirichlet boundary conditions but only an approximation, we face an issue. 
    Indeed, although we have a fairly precise picture of the structure of $v(s,\cdot)$ for $x_1$ larger than a small power of $s$, its behaviour in the region where $\vert x_1\vert$ is smaller  than such a power  remains unclear. There, we only have the estimate that its size is slightly larger than $s^{-\frac{N}{2}-1}$. As a consequence, the suitable hyperplane around which  one should perform the odd extension of $v(s,\cdot)$ into $v^*(s,\cdot)$ is not obvious. An even more serious question is whether one can compare the new   super and sub-solutions to $v(t,\cdot)$ for $t\geq s$. One might argue that    $e^{-(t-s)\cI_*}v(s,\cdot)$ can be well approximated provided $t-s$ large enough; this is true, but the meaning of ``large'' is itself unclear. It depends on $s$, whereas one needs an estimate uniform in $s$.

The help will come from   the estimates in  \eqref{conclusion-super+sub}  which  in particular imply that the characteristic scale on which $v(s,\cdot)$ varies is of order $\sqrt s$. This suggests that one might apply Proposition \ref{p3.2} to approximate $e^{-(t-s)\cI_*}v(s,\cdot)$ by initial data that are spread over the scale $\sqrt s$. There is, however, a caveat: as noted above, the initial datum should really vary over the scale $\sqrt s$, in the sense that    its derivatives at that scale must remain bounded, which is actually not guaranteed. However, in order to get the barrier property, we do not need the smoothness of $v(s,\cdot)$, but rather the smoothness of some suitably defined upper or  lower envelopes. 
 This is why we modify $v(s,\cdot)$ for $x_1$ around 0 (more precisely, for $\vert x_1\vert$ less than a small power of $s$) and estimate carefully the action of the semigroup $e^{-(t-s)\ci_*}$ for $t$ large but not exceedingly larger than $s$ so as to derive a good approximation of $e^{-(t-s)\cI_*}v(s,\cdot)$ in this time range. Once we are out of this range, things become easier as we rely on Theorem \ref{thm_heat kernel estimate}. 

\subsection{The improved barriers}
We start from a given large time $s$ and try to reproduce the analysis of Section \ref{sec4} from that time. If we were working with a diffusive equation this would  amount to replacing $w^*(t,x)=e^{-t\cI_*}w_0^*(x)$ by $e^{-(t-s)\cI_*}v^*(s,\cdot)(x)$; as we work with a nonlocal equation with no regularising effect we will need to devise a careful modification of $v(s,\cdot)$.  It will be convenient to work with the variable $\zeta=x/\sqrt{s}$, however we eventually would like to give our conclusions in the variable $x$, for which we need some care with the notation\footnote{ At the end of the section   the variable $\zeta$ will have another definition, this will be clearly indicated.}.
Set
\begin{equation}
\label{e5.20}
v(s,x)=\frac1{s^{\frac{N+1}2}}{\boldsymbol\nu}(s,\frac{x}{\sqrt s})\overset{\zeta=\frac{x}{\sqrt{s}}}{=}\frac1{s^{\frac{N+1}2}}{\boldsymbol\nu}(s,\zeta).
\end{equation}
Rephrasing \eqref{e6.2} for the function ${\boldsymbol\nu}(s,\zeta)$ gives that
\begin{equation}
\label{e6.3}
\kappa_-\zeta_1 e^{-\frac{\zeta^TH_{e_1}\zeta}{2}}-\frac{\Lambda e^{-A\vert\zeta\vert}\chi_1(\vert\zeta\vert)}{s^{\frac{1}2-\vartheta}}\leq {\boldsymbol\nu}(s,\zeta)\leq\kappa_+\zeta_1 e^{-\frac{\zeta^TH_{e_1}\zeta}{2}}+\frac{\Lambda e^{-A\vert\zeta\vert}\chi_1(\vert\zeta\vert)}{s^{\frac{1}2-\vartheta}},~~~\zeta_1\ge \frac{1}{t^{\frac{1}{2}-\delta}}.
\end{equation}

The main modification of ${\boldsymbol\nu}(s,\cdot)$ that we propose concerns   the vicinity of the origin, as scarce information is available in this region. Let us first introduce some ingredients.
For $1/s\ll\varepsilon\ll1$,  define a smooth  function $\chi_{3,\varepsilon}$ on $\R$ such that
\begin{equation}
	\label{chi_3}
	r\mapsto \chi_{3,\varepsilon}(r)~\text{is nonincreasing such that}~\chi_{3,\varepsilon}(r)\equiv 0~\text{on}~[2\varepsilon,+\infty)~\text{and}~\chi_{3,\varepsilon}(r)\equiv 1~\text{on}~(-\infty,\varepsilon].
\end{equation}  
 Choose positive constants $\tilde\kappa_\pm$ and $\mu_\pm$ so that 
\begin{equation*}
	\label{e6.5}
	2\tilde\kappa_-\leq\mu_-\leq2\mu_+\leq\tilde\kappa_+,
\end{equation*}
the $\kappa_\pm$ being given in \eqref{e6.2}, and\footnote{With a slight abuse of notation, we identify $\xi'=(\xi_2,\ldots,\xi_N)\in\mathbb R^{N-1}$ with the vector 
	$\xi'=(0,\xi_2,\cdots,\xi_N)\in\R^N$.}
\begin{equation}
	\label{e6.334}
	2\mu_+\vert\xi'\vert^2\leq\xi'^TH_{e_1}\xi'\leq 2 \mu _-\vert\xi'\vert^2.
\end{equation}
Then, 
 define ${\boldsymbol\nu}_{s,\varepsilon}^\pm$ as follows
\begin{equation}
\label{e6.331}
{\boldsymbol\nu}_{s,\varepsilon}^+(\zeta)=\tilde\kappa_+\chi_{3,\varepsilon}(\zeta_1+\varepsilon)e^{-\mu_+\vert\zeta'\vert^2}(\zeta_1+\varepsilon)+\big(1-\chi_{3,\varepsilon}(\zeta_1+\varepsilon)\big){\boldsymbol\nu}(s,\zeta),~~~\zeta_1\geq -\varepsilon,
\end{equation}
and \begin{equation}
	\label{e6.331-lower}
	{\boldsymbol\nu}_{s,\varepsilon}^-(\zeta)=\tilde\kappa_-\chi_{3,\varepsilon}(\zeta_1-\varepsilon)e^{-\mu_-\vert\zeta'\vert^2}(\zeta_1-\varepsilon)+\big(1-\chi_{3,\varepsilon}(\zeta_1-\varepsilon)\big){\boldsymbol\nu}(s,\zeta),~~~\zeta_1\geq \varepsilon.
\end{equation}
As in the dictionary \eqref{e5.20}, we also introduce the functions $v_{s,\varepsilon}^\pm(x)$ as
\begin{equation}
\label{e6.700}
v_{s,\varepsilon}^\pm(x)=\frac1{s^{\frac{N+1}2}}{\boldsymbol\nu}_{s,\varepsilon}^\pm\Big(\frac{x}{\sqrt{s}}\Big)\overset{\zeta=\frac{x}{\sqrt{s}}}{=}\frac1{s^{\frac{N+1}2}}{\boldsymbol\nu}_{s,\varepsilon}^\pm(\zeta).
\end{equation}
Then, we denote by ${\boldsymbol\nu}_{s,\varepsilon}^{\pm,*}(\zeta)$ for $\zeta\in\R$ and $v_{s,\varepsilon}^{\pm,*}(x)$ for $x\in\R$ the extended functions of  ${\boldsymbol\nu}_{s,\varepsilon}^{\pm}$ and $v_{s,\varepsilon}^{\pm}$ respectively, such that 
\begin{equation}\label{extension}
{\boldsymbol\nu}_{s,\varepsilon}^{\pm,*}(\zeta)={\boldsymbol\nu}_{s,\varepsilon}^{\pm}(\zeta)~~\text{for}~\zeta_1\ge\mp \varep,~~~~~{\boldsymbol\nu}_{s,\varepsilon}^{\pm,*}(\zeta)=-{\boldsymbol\nu}_{s,\varepsilon}^+\big(\mathcal{R}_{e_1}(\zeta\pm\varepsilon e_1)\big)~~\text{for}~\zeta\in\R^N,
\end{equation}
and
\begin{equation}
\label{exten_x}
v_{s,\varepsilon}^{\pm,*}(x)=v_{s,\varepsilon}^{\pm}(x)~~\text{for}~x_1\ge \mp \varep\sqrt{s},~~~~~v_{s,\varepsilon}^{\pm,*}(x)=-v_{s,\varepsilon}^+\big(\mathcal{R}_{e_1}(x\pm\varepsilon \sqrt se_1)\big)~~\text{for}~x\in\R^N.
\end{equation}
Finally, let $\cV_1$, $\cV_2$ and $\cV_3$ be as defined in \eqref{upper-n} and \eqref{lower-n}. 

We set
\begin{equation}
\label{e6.7}
\overline{v}_{s,\varepsilon}(t,x)=\Big(1+\frac1{s^\gamma}-\frac1{t^\gamma}\Big)e^{-(t-s)\cI_*}v_{s,\varepsilon}^{+,*}(x)+\cV_1(t,x+\varepsilon\sqrt se_1)+\cV_2(t,x),
\end{equation}
and
\begin{equation}
\label{e6.8}
\underline{v}_{s,\varepsilon}(t,x)=\max\biggl(\Big(1-\frac1{s^\gamma}+\frac1{t^\gamma}\Big)e^{-(t-s)\cI_*}v_{s,\varepsilon}^{-,*}(x)-\cV_2(t,x-\varepsilon\sqrt s e_1)-\cV_3(t,x), 0\biggl).
\end{equation}
From  \eqref{e6.334} and \eqref{e6.331}, we have $\underline v_{s,\varepsilon}(x)\leq v(s,x)$ if $x_1\geq \varepsilon\sqrt s$, and $v(s,x)\leq \overline v_{s,\varepsilon}(x)$ if $x_1\geq -\varepsilon\sqrt s$. The main ingredient for the proof of Theorem \ref{t6.1} is the following theorem.
\begin{theorem}
\label{t6.2}
There is $s_\varepsilon>0$ such that, for $s\geq s_\varepsilon$, the  function $\overline{v}_{s,\varepsilon}$ is a super-solution to \eqref{v-eqn} for $t\geq s$, $x_1\geq -(t-s)^\delta-\varepsilon\sqrt s$ and $x'\in\R^{N-1}$. Still for $s\geq s_\varepsilon$, the function $\underline{v}_{s,\varepsilon}$ is a sub-solution to \eqref{v-eqn} for $t\geq s$, $x_1\geq (t-s)^\delta+\varepsilon\sqrt s$ and $x'\in\R^{N-1}$.
\end{theorem}
As we have already explored in an extensive manner the functions $\cV_i$ ($i=1,2,3$), the main part of the effort will lie in understanding the quantities  $e^{-(t-s)\cI_*}v_{s,\varepsilon}^{\pm,*}$, for $t\geq s$. Our goal is to show that they are as close as possible to $e^{(t-s)A_{e_1}}v_{s,\varepsilon}^{\pm,*}$, which are objects that are amenable to explicit estimates. The issue here is that the main tool at our disposal  for carrying out this task, namely Theorem \ref{thm_heat kernel estimate},  involves the $H^m$-norm of the initial datum. This is innocent if the initial datum is fixed once and for all, and Proposition \ref{p6.1} shows that it remains harmless in the long time range. The time range close to $s$, however, requires a separate treatment.

To prove Theorem \ref{t6.2}, we will not repeat the heavy computations of Section \ref{sec4}, we will rather insist on the new features, which also entail heavy notations and computations. As the correcting functions $\cV_i$ $(i=1,2,3)$ are unchanged, it follows that $\cN \cV_i$ remains  the same. What one must check is that either of $\cN \cV_i$ $(i=1,2,3)$ shall dominate the terms $\cN \bigl(e^{-(t-s)\cI_*}v_{s,\varepsilon}^{\pm,*}\bigl)$ in the regions where roughly speaking  $\zeta_1$ is close to zero or $|\zeta|$ is very large, and that 
 $\cN \bigl(e^{-(t-s)\cI_*}v_{s,\varepsilon}^{\pm,*}\bigl)$ is supposed to
 dominate $\cN \cV_i$ $(i=1,2,3)$ in all other regions. In view of the computations in the upper and lower barriers in Section \ref{sec4}, the features to be checked are  that $e^{-(t-s)\cI_*}v_{s,\varepsilon}^{\pm,*}$ (i). has a linear behaviour that is controlled from below inside the diffusive zone, at least when $x_1$ is larger than a small power of $t$, (ii). is dominated by $\cV_2$ in the region where $x_1$ is less than a small power of $t$, and (iii). is dominated by $\cV_3$ outside the diffusive zone.  The precise formulation is given in the three propositions below.
\begin{proposition}[The diffusive zone]\label{p5.6} 
	For all $\rho >0$,  there exist $\overline\kappa_\rho >0$ and $\underline\kappa_\rho >0$ such that 
	\begin{equation}\label{e6.85}
		e^{-(t-s)\cI_*}v_{s,\varepsilon}^{+,*}(x)\geq\frac{\overline\kappa_\rho (x_1+\varepsilon\sqrt s)}{t^{\frac{N}2+1}},~~~t\geq s,~~-\varepsilon\sqrt s-t^{\delta}\leq x_1\leq \rho \sqrt t,~~\vert x'\vert\leq \rho \sqrt t,
	\end{equation}
and
	\begin{equation}\label{e6.81}
		e^{-(t-s)\cI_*}v_{s,\varepsilon}^{-,*}(x)\geq \frac{\underline\kappa_\rho (x_1-\varepsilon\sqrt s)}{t^{\frac{N}2+1}},~~~t\geq s,~~-\varepsilon\sqrt s+t^{\delta}\leq x_1\leq \rho \sqrt t,~~\vert x'\vert\leq \rho \sqrt t.
	\end{equation}
\end{proposition}
\begin{proposition}[The far field]\label{p5.7} 
For all $\vartheta\in(0,\di\frac12)$, there are $\Lambda,\Lambda',>0$ depending on $\vartheta$ such that
$$\displaystyle
-\frac{\Lambda e^{-\frac{\Lambda'\vert x\vert}{\sqrt t}}}{t^{\frac{N}2+1-\vartheta}}\leq e^{-(t-s)\cI_*}v_{s,\varepsilon}^{*,\pm}(x)\leq  \frac{\Lambda e^{-\frac{\Lambda'\vert x\vert}{\sqrt t}}}{t^{\frac{N}2+1-\vartheta}}.
$$
\end{proposition}
\begin{proposition}[The region $x_1\pm\varepsilon\sqrt s=O(t^\delta)$]\label{p5.8} 
For all $\rho >0$ there is $\Lambda_\rho >0$ such that, for $t\geq s$, $\vert\zeta'\vert\leq \rho $,  and $-t^\delta\leq x_1 \pm \varepsilon\sqrt s\leq t^\delta$ we have
$$
\vert e^{-(t-s)\cI_*}v_{s,\varepsilon}^{*,\pm}(x)\vert\leq \frac{\Lambda_\rho }{t^{\frac{N}2+1-\delta}}.
$$
\end{proposition}
These propositions are sufficient to carry out the computations of Section \ref{sec4}, and, eventually, prove Theorem \ref{t6.2}. As Theorem \ref{thm_heat kernel estimate} or proposition \ref{p3.1} involve the $H^m$-norm of the datum to which $e^{-(t-s)\cI_*}$ is applied, we  first should   focus on  analysing the regularity of $v(s,\cdot)$. This is done in the next section. Proposition \ref{p5.6} is proved in the subsequent section, the emphasis being laid on the time range $t-s\ll s$. Then, we prove Propositions \ref{p5.7} and \ref{p5.8}, and we finally conclude.
\subsection{Sobolev estimates for $v$}
Recall that Theorem \ref{t2.20} and Corollary \ref{cor_2.20} in Section \ref{TW} ensure the boundedness of all derivatives of $u$.
The objective is here a polynomial bound for the $H^m$ norm of $v(s,\cdot)$.
\begin{proposition}
\label{p6.1}
For each $m\in\mathbb{N}^*$, we have
\begin{equation*}
\sup_{s>1} s^{-(m-1)}	\Vert v(s,\cdot)\Vert_{H^m(\R^N)}<+\infty.
\end{equation*}
\end{proposition}
\begin{proof}
  Let $v_i(t,x)=e^{\lambda^*x_1}\partial_{x_i}\bu(t,x)$, where we recall that $\bu$ satisfies \eqref{bu}. Then, $v_i$ solves the equation
 \begin{equation}
\label{e6.210}
  \partial_tv_i+\cI_*v_i+\bigl(f'(0)-f'(u)\bigl)v_i=0.
 \end{equation} 
 Multiplying \eqref{e6.210} by $\mathrm{sgn}~v_{i}$ and using  Kato's inequality in Proposition \ref{p3.200}  we obtain
$$
	\partial_t |v_{a,i}|+\mathcal I_* |v_{a,i}| \leq0.
	$$
 Integrating this inequality over $\RR^N$ yields $\Vert v_i(t,\cdot)\Vert_{L^1}\leq\Vert \partial_{x_i}v_0\Vert_{L^1}$, while the maximum principle yields $\Vert v_i(t,\cdot)\Vert_\infty\leq\Vert  \partial_{x_i}v_0\Vert_\infty$.  This entails the boundedness of $\Vert v_i(t,\cdot)\Vert_{L^2}$. As $v_i=\partial_{x_i}v$ if $i\neq 1$ and $v_1=\partial_{x_1}v-\lambda^*v$ if $i=1$ we infer the boundedness of $\Vert \partial_{x_1}v(t,\cdot)\Vert_{L^2}$.
 
 Now, set $v_{\alpha,i}:=\partial^\alpha v_i$, then the equation for $v_{\alpha,i}$ has the following structure as \eqref{e3.330}:
\begin{equation}
\label{e6.211}
\partial_tv_{\alpha,i}+\cI_*v_{\alpha,i}+\big(f'(0)-f'(u)\big)v_{\alpha,i}=\sum_{\vert\beta\vert\leq\vert\alpha\vert-1} v_{\beta,i}\sum_{k,\gamma,\gamma'}b_{k,\gamma,\gamma'}f^{(k)}(u)(\partial_{x_1}^{\gamma_1}u)^{\gamma_1'}\ldots(\partial^{\gamma_N}_{x_N}u)^{\gamma_N'},
\end{equation}
where $k$ runs between $2$ and $\vert\alpha\vert$, and the lengths of the multi-indices $\gamma$ and $\gamma'$ are also $\leq\vert\alpha\vert$. The constants $b_{k,\gamma,\gamma'}$ have, as in Corollary \ref{cor_2.20}, no remarkable value and can be zero.

We finish the proof of the proposition by induction. For any nonnegative integer $m$, we set $V_m(t,x)=\displaystyle\sum_{\vert\alpha\vert=m}\vert v_{\alpha,i}\vert$, with the convention $V_{0,i}=\vert v_i\vert$. We multiply \eqref{e6.211} by $\mathrm{sgn}~v_{\alpha,i}$, and sum over all $\alpha$ with lengths $m$ the resulting inequalities. For $m\geq1$ we obtain, from Corollary \ref{cor_2.20}, an inequality of the form
$$
\partial_tV_m+\cI_*V_m\leq C_m\sum_{k=0}^{m-1}V_k,
$$
where $C_m$ is a constant independent of $t$. And so, we infer
$$
\Vert V_m(t,\cdot)\Vert_{L^\infty(\R^N)}
\leq C_m\sum_{k=0}^{n-1}\int_0^t \Vert V_k(s,\cdot)\Vert_{L^\infty(\R^N)}\md s,
~~~~ \Vert V_m(t,\cdot)\Vert_{L^1}\leq C_m\int_0^t\Vert V_k(s,\cdot)\Vert_{L^1}\md s,
$$
This entails, by induction, a growth of order at most $t^{{m}}$ for the $L^\infty$ and $L^1$ norms of $V_m$, which in turn entails a  growth for all the $\Vert v_{\alpha,i}\Vert_{L^2}$, of order $\leq m$.
To end the proof, notice that every derivative of $v$ is a linear combination of $v$ and the $v_{a,i}$.
\end{proof}

\subsection{ The region where $e^{-(t-s)\cI_*}v_{s,\varepsilon}^{+,*}$ dominate: proofs of Proposition \ref{p5.6}}
We will mainly deal with $e^{-(t-s)\cI_*}v_{s,\varepsilon}^{+,*}(x)$, the function $e^{-(t-s)\cI_*}v_{s,\varepsilon}^{-,*}(x)$ can be treated in exactly the same way. 
Fix $\varsigma\in(0,\displaystyle\frac{\delta}2)$, to be chosen smaller if necessary. We split the analysis into two regimes: the short-time regime $t-s\leq s^\varsigma$ and the long-time regime  $t-s\geq s^\varsigma$. Here it will be convenient to work with the rescaled variable $\zeta=x/{\sqrt s}$.

\medskip

\noindent
{\bf Case of $t-s\leq s^\varsigma$}.
We would like to apply Proposition \ref{p3.1} to compare $ e^{-(t-s)\cI_*}v_{s,\varepsilon}^{+,*}(x)$ to $$ e^{(t-s) A_{e_1}}v_{s,\varepsilon}^{+,*}(x)=\displaystyle\frac1{s^{\frac{N+1}2}}e^{\frac{t-s}{s} A_{e_1}}{\boldsymbol\nu}_{s,\varepsilon}^{+,*}(\zeta).$$ 
However, applying   Proposition \ref{p3.1} directly to ${\boldsymbol\nu}_{s,\varepsilon}^{+,*}(\zeta)$ would produce a large error term: we only have, indeed, Proposition \ref{p6.1} at our disposition, which yields polynomial estimates for $\Vert {\boldsymbol\nu}_{s,\varepsilon}^{+,*}\Vert_{H^m}$.  
We bypass the difficulty as follows:
choose $\rho >10$, and let $\chi_{3,\varepsilon}$ be as given in \eqref{chi_3}. Let the 
function $\tilde{{\boldsymbol\nu}}_{s,\varepsilon}(\zeta)$ be defined as 
$$
\tilde{{\boldsymbol\nu}}_{s,\varepsilon}(\zeta)=\tilde\kappa_+\chi_{3,\varepsilon}(\zeta_1+\varepsilon)e^{-\mu_+\vert\zeta'\vert^2}(\zeta_1+\varepsilon)+
\bigl(1-\chi_{3,\varepsilon}(\zeta_1+\varepsilon)\bigl)\tilde\kappa_-e^{-\mu_-\vert\zeta'\vert^2}(\zeta_1+\varepsilon), ~~\zeta_1\geq -\varepsilon.
$$
 Obviously we have $\tilde{{\boldsymbol\nu}}_{s,\varepsilon}(\zeta)\leq {\boldsymbol\nu}_{s,\varepsilon}^{+,*}={\boldsymbol\nu}_{s,\varepsilon}^{+}$  in the region $\{\zeta_1\ge -\varepsilon\}$. A natural idea would be to do the same extension for $\tilde{{\boldsymbol\nu}}_{s,\varepsilon}(\zeta)$ to the region $\{\zeta_1\leq-\varepsilon\}$ as in \eqref{extension}, but this is not going to work:  in this region, the extension
 is below ${\boldsymbol\nu}_{s,\varepsilon}^{+,*}$, so that one cannot compare $e^{-(t-s)\cI_*}{\boldsymbol\nu}_{s,\varepsilon}^{+,*}$ and $e^{-(t-s)\cI_*}\tilde{{\boldsymbol\nu}}_{s,\varepsilon}^*$.  To overcome this additional inconvenient, let us recall, from \eqref{e6.2}, that we have
$$
 \hbox{for $\zeta_1,\vert\zeta'\vert\geq \rho ,$}\ {\boldsymbol\nu}_{s,\varepsilon}^{+,*}(\zeta)\lesssim e^{-A\vert\zeta\vert}.
$$
for all $A>1$. Pick one, and choose $a<{A}/{\vert\mathcal{R}_{e_1}\vert}$; we set, for $\zeta_1\leq-\varepsilon$:
$$
\tilde{{\boldsymbol\nu}}_{s,\varepsilon}^*(\zeta)=-\tilde\kappa_+\chi_{3,\varepsilon}(-\zeta_1-\varepsilon)e^{-\mu_+\vert\bigl(\mathcal{R}_{e_1}\zeta\bigl)'\vert^2}(\zeta_1+\varepsilon)-\Lambda\bigl(1-\chi_{3,\varepsilon}(-\zeta_1-\varepsilon)\bigl)e^{-a\vert(\mathcal{R}_{e_1}\zeta)'\vert}.
$$
If $\Lambda>0$ is large enough, for $\zeta_1\leq-\varepsilon$,  we have $\tilde{{\boldsymbol\nu}}_{s,\varepsilon}^*(\zeta)\leq {\boldsymbol\nu}_{s,\varepsilon}^{+,*}(\zeta)$. As a consequence we have $e^{-(t-s)\cI_*}\tilde{{\boldsymbol\nu}}_{s,\varepsilon}^*(\displaystyle\frac{.}{\sqrt s})\leq {\boldsymbol\nu}_{s,\varepsilon}^{+,*}(\displaystyle\frac{.}{\sqrt s})$.
Moreover,  $\tilde{{\boldsymbol\nu}}_{s,\varepsilon}^*$ is obviously in all the spaces $H^m(\R^N)$, with derivatives of order $m$ bounded independently of $s$ -- but of course not of $\varepsilon$. A direct application of Proposition \ref{p3.3} yields  estimate \eqref{e6.8}, which of course deteriorates   when $\rho $ increases.

\medskip

\noindent
{\bf Case of $t-s\geq s^\varsigma$}.
 Then we have, by Theorem \ref{thm_heat kernel estimate}, for some $\gamma\in(\displaystyle\frac38,\frac12)$: 
$$
e^{-(t-s)\cI_*}{\boldsymbol\nu}_{s,\varepsilon}^{+,*}(\displaystyle\frac{.}{\sqrt s})=e^{(t-s) A_{e_1}} {\boldsymbol\nu}_{s,\varepsilon}^{+,*}(\displaystyle\frac{.}{\sqrt s})+O\bigl(\Vert {\boldsymbol\nu}_{s,\varepsilon}^{+,*}(\displaystyle\frac{.}{\sqrt s})\Vert_{H^m_x}e^{-(t-s)^{1-2\gamma}} \bigl),
$$
for some suitably large $m$. By Proposition \ref{p6.1}, the $H^m$ norm of $v(t,\cdot)$ grows at most polynomially. So, one may neglect the term $O\bigl(\Vert {\boldsymbol\nu}_{s,\varepsilon}^{+,*}(\displaystyle\frac{.}{\sqrt s})\Vert_{H^m_x}e^{-(t-s)^{1-2\gamma}} \bigl)$ from now on, and work on $e^{(t-s) A_{e_1}} {\boldsymbol\nu}_{s,\varepsilon}^{+,*}(\displaystyle\frac{.}{\sqrt s})$, a much more tractable quantity.
 Set 
 \begin{equation}
\label{e6.21}
\tau=\displaystyle\frac{t-s}s,~~~{\boldsymbol\omega}^{app}(\tau,\zeta):=e^{\tau A_{e_1}} {\boldsymbol\nu}_{s,\varepsilon}^{+,*}(\zeta).
\end{equation}
We should dress ${\boldsymbol\omega}^{app}$ with $+$, $s$, $*$, $\varepsilon$; let us not do so because the notation are already heavy.
Recall that, from the observation \eqref{e3.5}, ${\boldsymbol\omega}^{app}(\tau,\zeta)$ satisfies the equation
\begin{equation}\label{e6.20}
\begin{array}{rll}
\displaystyle \partial_\tau {\boldsymbol\omega}^{app}-\frac12\mathrm{div}_\zeta(H_{e_1}\nabla_\zeta {\boldsymbol\omega}^{app})=&0~~(\tau>0,\zeta_1>-\varepsilon,\zeta'\in\R^{N-1})\\
{\boldsymbol\omega}^{app}(\tau,-\varepsilon,\zeta')=&0,~~~~{\boldsymbol\omega}^{app}(0,\zeta)={\boldsymbol\nu}_{s,\varepsilon}^{+,*}(\zeta)
\end{array}
\end{equation}
Once again, several ranges of $\tau$ should be envisaged. 
\begin{itemize}
\item [--] {\it The range $s^\varsigma\leq \tau\leq\varrho_\varepsilon s$,} where $\varrho_\varepsilon$ is a small number. We use the same idea as in the previous regime, but we are now working with ${\boldsymbol\omega}^{app}$ defined by \eqref{e6.20} rather than with the solution of a nonlocal equation. And so, while ${\boldsymbol\nu}_{s,\varepsilon}^{+,*}(\zeta)$ may not be a smooth initial datum, we may easily place below it a function $\underline {\boldsymbol\omega}_\varepsilon(\zeta)$ that coincides with ${\boldsymbol\nu}_{s,\varepsilon}^{+,*}(\zeta)$ for $\zeta_1\leq\varepsilon$, that is below $\tilde\kappa_-(\zeta_1+\varepsilon)e^{-\mu_+\vert\zeta\vert}$ for $\zeta_1\geq 2\varepsilon$,  such that $\partial_{\zeta_1}\underline {\boldsymbol\omega}_\varepsilon(\zeta)\geq\displaystyle\frac{\tilde\kappa_-}{2}e^{-A\vert\zeta\vert}$ for a large constant $A>0$,  which has  all its derivatives bounded, and which is finally dominated by $e^{-\mu_+\vert\zeta\vert}$. Then, we have ${\boldsymbol\omega}^{app}(\tau,\zeta)\geq e^{\tau A_{e_1}}\underline {\boldsymbol\omega}_\varepsilon(\zeta)$ and, moreover, we have, form classical parabolic theory:
\begin{equation}
\label{e6.22}
\Vert \nabla e^{\tau A_{e_1}}\underline {\boldsymbol\omega}_\varepsilon-\nabla\underline {\boldsymbol\omega}_\varepsilon\Vert_\infty\lesssim\tau\Vert D^2\underline {\boldsymbol\omega}_\varepsilon\Vert_\infty\lesssim\tau\Vert\underline {\boldsymbol\omega}_\varepsilon\Vert_{H^{m_N}},
\end{equation} 
where $m_N$ is chosen so that $H^{m_N}(\RR^N)$ is imbedded in $W^{2,\infty}(\R^N)$. The construction of $\underline {\boldsymbol\omega}_\varepsilon$ entails, for every $\rho >0$, the existence of $\underline\kappa_\rho $ such that 
$\partial_{\zeta_1}\underline {\boldsymbol\omega}_\varepsilon(\zeta)\geq\kappa_\rho$ as soon as $\vert\zeta\vert\leq \rho $.

This implies the desired linear behaviour for $e^{\tau A_{e_1}}\underline {\boldsymbol\omega}_\varepsilon$, as soon as $\varrho_\varepsilon$ is small enough. This also implies the estimate  for $w^{app}$ because it is above  $e^{\tau A_{e_1}}\underline {\boldsymbol\omega}_\varepsilon$.
\item [--] {\it The range $\tau\geq\varrho_\varepsilon $.} This time we prefer to write \eqref{e6.20} in the true self-similar coordinates starting from $\tau=\varrho_\varepsilon$. So, for $\tau\geq\varrho_\varepsilon$ we set 
$$\xi=\displaystyle\frac{\zeta}{\sqrt\tau},~\sigma=\mathrm{ln}~\tau,~~~ \tilde{\boldsymbol\omega}^{app}(\sigma,\xi)=\tau^{\frac{N+1}2}{\boldsymbol\omega}^{app}(\tau,\zeta)=e^{\frac{N+1}2\sigma}{\boldsymbol\omega}^{app}(e^\sigma,e^{\frac\sigma2}\xi).
$$
Equation \eqref{e6.20} becomes
\begin{equation}\label{e6.21}
\begin{array}{rll}
\displaystyle \partial_\sigma \tilde{\boldsymbol\omega}^{app}-\frac12\biggl(\mathrm{div}_\xi H_{e_1}\nabla_\xi +\xi.\nabla+(N+1)\biggl)\tilde{\boldsymbol\omega}^{app}=&0~~(\sigma>\mathrm{ln}~\varrho_\varepsilon,\xi_1>-\displaystyle\frac{\varepsilon}{\sqrt\varrho_\varepsilon},\xi'\in\R^{N-1})\\
\tilde{\boldsymbol\omega}^{app}(\tau,-\varepsilon,\zeta')=&0,~~~~\tilde{\boldsymbol\omega}^{app}(\mathrm{ln}~\varrho_\varepsilon,\zeta)={\boldsymbol\nu}_{s,\varepsilon}^{+,*}(\zeta)
\end{array}
\end{equation}
We may apply the classical parabolic theory: for $\tau\leq T$, with $T>0$ a large constant, the strong maximum principle implies the positivity of $\tilde{\boldsymbol\omega}^{app}$ away from the boundary $\{\zeta_1=-\displaystyle\frac{\varepsilon}{\sqrt\varrho_\varepsilon}\}$, while  the Hopf lemma imply the positivity of $\partial_{\zeta_1}\tilde{\boldsymbol\omega}^{app}$ in the vicinity of the boundary. This once again entails estimate \eqref{e6.8}. For $\tau$ going to $+\infty$, we use the convergence, according to \eqref{e5.2001}, of $\tilde{\boldsymbol\omega}^{app}$ to a positive multiple of $\Gamma_{e_1}(\xi)$. Reverting to ${\boldsymbol\omega}^{app}$ we have ${\boldsymbol\omega}^{app}(\tau,\zeta)\gtrsim\displaystyle\frac {\zeta_1+\varepsilon}{\tau^{\frac{N+1}2}}$ in every cube $\{0\leq\zeta_1+\varepsilon\leq \rho ,~\vert\zeta'\vert\leq \rho \}$. This, with the definition \eqref{e6.21} of $\tau$, and the relation \eqref{e5.20} between the $v$'s and $w$'s, implies 
\eqref{e6.8}.
\end{itemize}
The sub-solution property is shown in a similar fashion, taking into account that it is compactly supported.  So, in order to estimate $e^{-(t-s)\cI_*}v_{s,\varepsilon}^{-,*}$ from below we may use the function
$$
\tilde{{\boldsymbol\nu}}_{s,\varepsilon}(\zeta)=\chi_4(\vert\zeta'\vert)\biggl(\tilde\kappa_-\chi_{3,\varepsilon}(\zeta_1+\varepsilon)e^{-\mu_+\vert\zeta'\vert^2}(\zeta_1+\varepsilon)+
\bigl(1-\chi_{3,\varepsilon}(\zeta_1+\varepsilon)\bigl)\frac{\tilde\kappa_-}2e^{-\mu_-\vert\zeta'\vert^2}(\zeta_1+\varepsilon)\biggl),
$$
where $\chi_4$ is a function of the real variable $h$, equal to $1$ if $h\leq \rho -1$ and 0 if $h\geq \rho +1$.

\subsection{The regions where $\cV_2$ and $\cV_3$ dominate: Proofs of Propositions \ref{p5.7} and \ref{p5.8}}
Proposition \ref{p5.7} is now a direct consequence of Theorem \ref{t3.10}. As for Proposition \ref{p5.8},
the principle is the same as that of Proposition \ref{p5.6}, the emphasis being laid on the stages $t-s\leq s^\varsigma$. One needs to bound $\vert e^{-(t-s)\cI_*}v_{s,\varepsilon}^{*,\pm}(x)\vert$ from above; to bound $v_{s,\varepsilon}^{*,\pm}(x)$ from above we may use $\kappa_+(\zeta_1\pm\varepsilon)$ on $\{\vert \zeta_1\pm\varepsilon\vert\leq\varepsilon\}$, and a smooth function transitionning from $\kappa_+(\zeta_1\pm\varepsilon)$ to 0 for $\zeta_1\pm\varepsilon\leq-\varepsilon$.

\subsection{Proof of Theorem \ref{t6.1}}
Consider the functions $\overline{v}_{s,\varepsilon}$ and $\underline{v}_{s,\varepsilon}$ defined by \eqref{e6.7} and \eqref{e6.8} respectively, and let us first show that, when $s$ is large, we have
\begin{equation}
\label{e6.100}
\forall t\geq s,~~x_1\geq-\varepsilon\sqrt s-(t-s)^\delta:~~~\overline{v}_{s,\varepsilon}(t,x)\geq v(t,x).
\end{equation}
     Propositions \ref{p5.7} and \ref{p5.8} show that we have, in this strip, the estimate   $\overline{v}_{s,\varepsilon}(t,x)\gtrsim t^{-\frac{N}2-1+\vartheta}e^{-\frac{A\vert x\vert}{\sqrt t}} $, for all $A>1$, while   $v(t,x)\lesssim\inf\bigl(e^{-\lambda^*t^\delta},t^{-\frac{N}2-\varsigma}e^{-\frac{B\vert x\vert}{\sqrt t}}\bigl)$ as soon as $t$ is large enough. This is enough to show \eqref{e6.100}.
In the same spirit we have $v(t,x)\geq \underline{v}_{s,\varepsilon}(t,x)$ for $t\geq s$, $\vert x'\vert\leq\rho$  and $\varepsilon\sqrt s+(t-s)^\delta\leq x_1\leq -\varepsilon\sqrt s+(t-s)^\delta+1$, still  if $s$ is large. So, there is $s_\varepsilon>0$ such that $\underline{v}_{s,\varepsilon}(t,x)\leq v(t,x)\leq\overline{v}_{s,\varepsilon}(t,x)$ as soon as $s\geq s_\varepsilon$, $t\geq s$, $\vert x'\vert\leq\rho$ and $x_1\geq-\varepsilon\sqrt s-(t-s)^\delta.$
 
 We may now conclude; for this we are going to show that, for $t\geq s$, and $s$ larger than a time that depends of course on $\varepsilon$, we have
 \begin{equation}
\label{e6.2018}
 \bigl\vert \mu_{e_1}[{\boldsymbol\nu}(t,\cdot)]-\mu_{e_1}[{\boldsymbol\nu}(s,\cdot)]\bigl\vert=O(\varepsilon)+
O\bigl((t-s)^{-\frac{1}2+\vartheta}\bigl),
\end{equation} 
Our main effort will be devoted to comparing  the functions  $v(t,\cdot)$ and $e^{(t-s)A_{e_1}}v^*(s,\cdot)$. As the latter quantity is the solution of a nice linear heat equation, the conclusion will be easy. Let us introduce the functions $\overline{v}_{s,\varepsilon}$ and $\underline{v}_{s,\varepsilon}$ in the picture, so that we have
$$
2\bigl\vert v(t,\cdot)-e^{(t-s)A_{e_1}}v^*(s,\cdot)\bigl\vert\leq \bigl(\overline v_{s,\varepsilon}-v(t,\cdot)\bigl)+\bigl\vert\overline v_{s,\varepsilon}-e^{(t-s)A_{e_1}}v^*(s,\cdot)\bigl\vert
+ \bigl(v(t,\cdot)-\underline v_{s,\varepsilon}\bigl)+\bigl\vert\underline v_{s,\varepsilon}-e^{(t-s)A_{e_1}}v^*(s,\cdot)\bigl\vert
$$
From  the definitions \eqref{e6.7} and \eqref{e6.8}  we have:
 $$
 \begin{array}{rll}
 &2\bigl\vert v(t,\cdot)-e^{(t-s)A_{e_1}}v^*(s,\cdot)\bigl\vert\\
\leq&\displaystyle\bigl\vert \frac1{s^{\frac{N+2}2}}e^{-(t-s)\cI_*}{\boldsymbol\nu}_{s,\varepsilon}^{+,*}(\frac{.}{\sqrt s})-e^{(t-s)A_{e_1}}v_{s,\varepsilon}^{+,*}\bigl\vert+\bigl\vert \frac1{s^{\frac{N+2}2}}e^{-(t-s)\cI_*}{\boldsymbol\nu}_{s,\varepsilon}^{-,*}(\frac{.}{\sqrt s})-e^{(t-s)A_{e_1}}v_{s,\varepsilon}^{-,*}\bigl\vert\\
 &+\displaystyle\bigl\vert e^{(t-s)A_{e_1}}\bigl[\frac1{s^{\frac{N+1}2}}w^*(s,\frac{.}{\sqrt s})-v_{s,\varepsilon}^{+,*}\bigl]\bigl\vert+\frac1{s^{\frac{N+1}2}}\bigl\vert e^{(t-s)A_{e_1}}\bigl[\frac1{s^{\frac{N+1}2}}w^*(s,\frac{.}{\sqrt s})-v_{s,\varepsilon}^{-,*}\bigl]\bigl\vert\\
 &+\displaystyle\frac{\mathbbm{1}_{-t^\delta\leq x_1\leq \frac{3\pi t^\alpha}2}(x_1)\mathbbm{1}_{\vert x'\vert\leq\eta_2\sqrt t}(x')}{t^{\frac{N}2+1-\beta}}+\frac{\Lambda e^{-A\frac{\vert x\vert}{\sqrt t}}\mathbbm{1}_{\vert x\vert\geq\eta_1\sqrt t}(x)}{t^{\frac{N}2+1-\vartheta}}
 \end{array}
 $$
 The first two terms are estimated by Theorem \ref{thm_heat kernel estimate}, so that we have
 $$ 
 \begin{array}{rll}
&\bigl\vert v(t,\cdot)-e^{(t-s)A_{e_1}}v^*(s,\cdot)\bigl\vert\\
 \leq&\displaystyle\frac1{s^{\frac{N+1}2}}\biggl\vert e^{(t-s)A_{e_1}}\bigl[w^*(s,\frac{.}{\sqrt s})-v_{s,\varepsilon}^{+,*}\bigl]\biggl\vert+\frac1{s^{\frac{N+1}2}}\biggl\vert e^{(t-s)A_{e_1}}\bigl[w^*(s,\frac{.}{\sqrt s})-v_{s,\varepsilon}^{-,*}\bigl]\biggl\vert\\
 &+\Lambda \max\bigl(\displaystyle \frac{e^{-(t-s)^{1-2\gamma}}}{s^{\frac{N}2-m_N+1}},\frac{e^{-A\frac{\vert x\vert}{\sqrt t}}}{t^{\frac{N}2+1-\vartheta}}\bigl)+\displaystyle\frac{\Lambda\mathbbm{1}_{-t^\delta\leq x_1\leq \frac{3\pi t^\alpha}2}(x_1)\mathbbm{1}_{\vert x'\vert\leq\eta_2\sqrt t}(x')}{t^{\frac{N}2+1-\beta}}+\frac{\Lambda e^{-A\frac{\vert x\vert}{\sqrt t}}\mathbbm{1}_{\vert x\vert\geq\eta_1\sqrt t}(x)}{t^{\frac{N}2+1-\vartheta}}
 \end{array}
 $$
 The above considerations have allowed us to replace the  operator $\cI_*$ by the   operator $A_{e_1}$, so that the remainder of the argument is a classical consideration on the heat equation. As it is short, we detail it nevertheless. We send $t$ to $+\infty$, so that $t$ and $t-s$ are comparable so that, in particular,   the term $e^{-A\frac{\vert x\vert}{\sqrt{t}}}$ may be replaced by $e^{-A\frac{\vert x\vert}{\sqrt{t-s}}}$ with an error of size $t^{-1}$ at most. From \eqref{e5.2001}, we have
$$e^{(t-s)A_{e_1}}v_{s,\varepsilon}^{*,\pm}(x)=\frac1{s^{\frac{N+1}2}}e^{\frac{t-s}sA_{e_1}}{\boldsymbol\nu}_{s,\varepsilon}^{*,\pm}(.)(\zeta)=\frac{\mu_{e_1}\bigl[{\boldsymbol\nu}^\pm_{s,\varepsilon}(.)\bigl]\Gamma_{e_1}(\zeta\sqrt{\frac{s}{t-s}})}{(t-s)^{\frac{N+1}2}}+
O(\frac{e^{-A\frac{\vert \sqrt s\zeta\vert}{\sqrt {t-s}}}}{(t-s)^{\frac{N}2+1}}),
$$
with
 $$
 \begin{array}{rll}
\displaystyle \mu_{e_1}\bigl[{\boldsymbol\nu}^\pm_{s,\varepsilon}(.))\bigl]=&\displaystyle\int_{\zeta_1\pm\varepsilon>0}(\zeta_1\pm\varepsilon){\boldsymbol\nu}_{s,\varepsilon}^\pm(\zeta)\md \zeta=\displaystyle\int_{\mp\varepsilon\leq\zeta_1\leq2\varepsilon}(\zeta_1\pm\varepsilon){\boldsymbol\nu}_{s,\varepsilon}^\pm(\zeta)\md \zeta +\int_{\zeta_1>2\varepsilon}(\zeta_1\pm\varepsilon){\boldsymbol\nu}(s,\zeta)\md \zeta\\
 =&\mu_{e_1}[{\boldsymbol\nu}(s,\cdot)]+O(\varepsilon).
 \end{array}
 $$
 Reverting to the variable $x=\sqrt s\zeta$, we all in all have, for $s\geq s_\varepsilon$ and $t\gg s$:
\begin{equation}
\label{e6.2017}
 \bigl\vert v(t,x)-\frac{\mu_{e_1}[{\boldsymbol\nu}(s,\cdot)]\Gamma_{e_1}(\frac{x}{\sqrt t-s})}{(t-s)^{\frac{N+1}2}}\bigl\vert=\frac{O(\varepsilon)\Gamma_{e_1}(\frac{x}{\sqrt t-s})}{(t-s)^{\frac{N+1}2}}+
O(\frac{e^{-A\frac{\vert x\vert}{\sqrt {t-s}}}}{(t-s)^{\frac{N}2+1-\vartheta}}).
\end{equation}
Let us define a new self-similar variable, that we still denote $\zeta$: $\zeta=\displaystyle\frac{x}{\sqrt{t-s}}$, and  $v(t,x)=\displaystyle\frac{{\boldsymbol\nu}(t,\zeta)}{(t-s)^{\frac{N+1}2}}$. Taking into account that $\zeta_1\Gamma_{e_1}(\zeta)$ has unit mass, we deduce from equation \eqref{e6.2017} the estimate \eqref{e6.2018} by multiplying the inequality by $\zeta_1$ and integrating on $\{x_1\geq-\varepsilon\sqrt s-(t-s)^\delta\}$.
As $\varepsilon$ is as small as we wish, we obtain the convergence of the function $t\mapsto\mu_{e_1}[{\boldsymbol\nu}(t,\cdot)]$ as $t\to+\infty$:  
estimate \eqref{e6.2018} indeed entails, for every $\varepsilon>0$, and for a constant $C>0$:
$$
\limsup_{s\to+\infty}\mu_{e_1}[{\boldsymbol\nu}(s,\cdot)]-C\varepsilon\leq\liminf_{t\to+\infty}\mu_{e_1}[{\boldsymbol\nu}(t,\cdot)]\leq\limsup_{t\to+\infty}\mu_{e_1}[{\boldsymbol\nu}(t,\cdot)]\leq\liminf_{t\to+\infty}\mu_{e_1}[{\boldsymbol\nu}(s,\cdot)]+C\varepsilon.
$$
Let us denote
\begin{equation*}
\kappa_\infty=\lim_{t\to+\infty}\mu_{e_1}[{\boldsymbol\nu}(t,\cdot)];
\end{equation*}
 a  last application of \eqref{e6.2017} implies \eqref{e6.1}, hence Theorem \ref{t6.1}.

 \section{ The final argument}\label{sec6}

 Fix any $\mu\in(4/25,1/4)$ 
 and any $\varep>0$ small enough, then it follows from Section \ref{sec5} that there exists $T_\varep>0$ sufficiently large such that  for $t\ge T_\varep$ and $x_1=t^\mu+o(t^\mu)$,
 \begin{equation}
 	\label{v-est1}
 	\begin{aligned}
 			(\kappa-\varep) x_1 e^{-\frac{x^T H_{e_1} x}{2t}}t^{-\frac{N}{2}-1}	\le v(t,x)\le	(\kappa+\varep) x_1 e^{-\frac{x^T H_{e_1} x}{2t}}t^{-\frac{N}{2}-1}
 	\end{aligned}
 \end{equation}
 when $|x'|=O(\sqrt{t})$,  and 
 	\begin{equation}
 		\label{v-est2}
 		v(t,x)=O\Big(x_1 e^{-\frac{x^T H_{e_1} x}{2t}}t^{-\frac{N}{2}-1}\Big)+O\Big( (t+T)^{-\frac{N}{2}-1+\vartheta}e^{-A\big(\frac{|x'|}{\sqrt{t+T}}-\eta_A\big)}	\Big)
 	\end{equation}   
 	when $|x'|\ge \eta_A\sqrt{t}$, where $\kappa>0$ is given in Theorem \ref{t6.1}.
 
 For any $\sigma\in[\kappa-\varep, \kappa+\varep]$, we introduce     
 \begin{equation}
 	\label{1-TW shift}
 	\psi_\sigma(t,x):=e^{\lambda^*(x_1+\frac{N+2}{2\lambda^*}\ln t)}     U_{c^*}\Big(x_1+\frac{N+2}{2\lambda^*}\ln t+\zeta_\sigma(t)\Big),~~~~~t\ge  T_\varep,~x_1\in\R.
 \end{equation}
 Here, the function $\zeta_\sigma(t)$  is chosen through the following constraint 
 \begin{equation}
 	\label{constraint}
 	\psi_\sigma(t,x)=\sigma t^{\mu}e^{-\frac{h_{11,e_1}}{2} t^{2\mu-1}},~~~~~~t\ge  T_\varep,~x_1=\mathcal{Y}^++R,
 \end{equation}
 where $h_{11,e_1}>0$ given in \eqref{def_H matrix}, and we have introduced for convenience 
 $$\mathcal{Y}^\pm(t):=-\frac{N+2}{2\lambda^*}\ln t\pm  t^\mu,~~~~~t\ge  T_\varep.$$
 Recalling that $ U_{c^*}$ satisfies the normalization $ U_{c^*}(s)\approx s e^{-\lambda^*s}$ as $s\to+\infty$, we find that for $t\ge  T_\varep$, 
 \begin{equation}\label{zeta_k<-3}
 	\zeta_\sigma(t)=-\frac{1}{\lambda^*}\ln\sigma+\mathcal{O}( t^{2\mu-1}),
 	~~~~~~~	|\dot\zeta_\sigma(t)|\le C t^{2\mu-2},
 \end{equation}
 with some $C>0$ independent of $\sigma$.

 On the other hand, by defining $
 V(t,x)=t^{\frac{N}{2}+1}v(t,x)$  for $t\ge T_\varep$ and $x\in\R^N$, then  \eqref{v-eqn} can be recast as 
 \begin{equation}
 	\label{eqn-V}
 	V_t+\mathcal{I}_{*}V-\frac{N+2}{2t}V+\underbrace{f'(0)V-e^{\lambda^*(x_1+\frac{N+2}{2\lambda^*}\ln t)}f\big(e^{-\lambda^*(x_1+\frac{N+2}{2\lambda^*}\ln t)} V\big)}_{=:Q(t,x;V)}=0,~~t\ge  T_\varep,~x\in\R^N.
 \end{equation}

 Substituting $\psi_\sigma$ into  \eqref{eqn-V}, together with \eqref{zeta_k<-3}, one has
 \begin{align*}
 	&\bigg|\partial_t \psi_\sigma+\mathcal{I}_{*}\psi_\sigma-\frac{N+2}{2t}
 	\psi_\sigma+Q(t,x;\psi_\sigma)\bigg|\\
 	&~~~~~~~~~~~~~~=\bigg|e^{\lambda^*(x_1+\frac{N+2}{2\lambda^*}\ln t)} U_{c^*}'\bigg(x_1+\frac{N+2}{2\lambda^*}\ln t+\zeta_\sigma(t)\bigg)\bigg(\dot\zeta_\sigma(t)+\frac{N+2}{2\lambda^*t}\bigg)\bigg|\le C t^{\mu-1}
 \end{align*}
 for $t\ge T_\varep$ and $x\in\{x\in\R^N|\mathcal{Y}^-(t)\le x_1\le\mathcal{Y}^+(t)\}$.

 By virtue of  \eqref{v-est1} and \eqref{constraint}, it is obvious to see that  for $x\in\{x\in\R^N|\mathcal{Y}^+(t)\le x_1\le \mathcal{Y}^+(t)+R\}$, 
 \begin{equation*}
 	\psi_{\kappa+\varep}\big(t,x_1\big)\ge V(t,x),~~~~~~\text{for}~t\ge T_\varep.
 \end{equation*}
 
 \begin{proposition}
 	\label{prop-5.3}
 	For $\varep>0$ small enough, there holds
 	\begin{equation*}
 		\lim_{t\to+\infty}\big(\psi_{\kappa\pm\varep}(t,x)-V(t,x) \big)=0, 
 	\end{equation*}
 	uniformly for $\mathcal{Y}^-(t)\le x_1 \le \mathcal{Y}^+(t)$ and $x'$ in any compact.
 \end{proposition}  
 \begin{proof} We divide into two steps.
 	
 	\noindent
 	\textbf{Step 1}.
 	We start with proving
 	\begin{equation*}
 		\label{prop6.1-1}
 		\limsup_{t\to+\infty}\big(V(t,x)-\psi_{\kappa+\varep}(t,x) \big)\le 0,
 	\end{equation*}  
 	uniformly in $\mathcal{Y}^-(t)\le x_1 \le \mathcal{Y}^+(t)$ and $x'$ in any compact. To do so, 
 	define
 	$\mathcal{S}(t,x):=(V-\psi_{\kappa+\varep})^+(t,x)$ for $t\ge  T_\varep$ and $x\in\{x\in\R^N|\mathcal{Y}^-(t)\le x_1 \le \mathcal{Y}^+(t)\}$. We are  led to the problem
 	\begin{equation}
 		\label{1-s_varep}
 		\begin{aligned}
 			\begin{cases}
 				\displaystyle	\Big| \mathcal{S}_t +\mathcal{I}_{*}\mathcal{S}-\frac{N+2}{2t}
 				\mathcal{S}+\mathcal{H}(t,x;\mathcal{S})\Big|<C t^{\mu-1},~~~&t\ge T_\varep,~\mathcal{Y}^-(t)\le x_1 \le \mathcal{Y}^+(t),\vspace{3pt}\\
 				\displaystyle	\mathcal{S}(t,x) \le   e^{-\lambda^* t^\mu}, ~~~~~~ ~~~~~~~~~~~~~~~~~ &t\ge T_\varep,~\mathcal{Y}^-(t)-R\le x_1\le  \mathcal{Y}^-(t),\vspace{3pt}\\
 				\mathcal{S}(t,x)  =0, ~~~~~~~~~~~~~~~~~~~~~~~~~~~~~~~~~~~~~~~~~&t\ge T_\varep,~\mathcal{Y}^+(t)\le x_1\le  \mathcal{Y}^+(t)+R,\vspace{3pt}\\
 				\displaystyle	\mathcal{S}( T_\varep,x)\le V( T_\varep,x), &\mathcal{Y}^-( T_\varep)-R\le x_1 \le \mathcal{Y}^+(T_\varep)+R.
 			\end{cases}
 		\end{aligned}
 	\end{equation}
 	Here, 
 	\begin{align*}
 		\mathcal{H}(t,x;\mathcal{S}):=Q(t,x;V)-Q(t,x;\psi_{\kappa+\varep})=f'(0)\mathcal{S}-d(t,x)\mathcal{S}\ge0, ~~~~~~\mathcal{S}\ge 0,
 	\end{align*} 
 	uniformly for $t\ge  T_\varep$ and $x\in\{x\in\R^N|\mathcal{Y}^-(t)\le x_1 \le \mathcal{Y}^+(t)\}$,
 	in which $d(t,x)$ is a continuous and  bounded function satisfying  $\Vert d(t,x)\Vert_{L^\infty}\le f'(0)$
 	since $0<f(s)\le f'(0)s$ for $s\in(0,1)$ and $f$ has linear extension outside $[0,1]$. It then suffices for us to show that $\mathcal{S}(t,x)\to 0$ as $t\to+\infty$, uniformly in $x\in\{x\in\R^N|\mathcal{Y}^-(t)\le x_1 \le \mathcal{Y}^+(t)\}$. 
 	
 	Remember that $\mu\in(4/25,1/4)$, one can then choose $\rho\in(\mu,1/2)$ such that $2\rho+\mu<1$, and finally fix $\upsilon\in(0,1-2\rho-\mu).$ 
 	Up to increasing $T_\varep$, let us assume that 
 	$\cos\big( t^{\mu-\rho}\big)>\frac{1}{2}$ for $t\ge  T_\varep$. Then fix $\mathcal{B}>0$ so large that $\frac{\mathcal{B}}{2} T_\varep^{-\upsilon}\ge V( T_\varep,x)$ uniformly in $x\in\{x\in\R^N|\mathcal{Y}^-(t)\le x_1 \le \mathcal{Y}^+(t)\}$, which is achievable thanks to \eqref{v-est1} and \eqref{v-est2}.
 	Define
 	\begin{equation}\label{bar S}
 		\overline{\mathcal{S}}(t,x)=\frac{\mathcal{B}}{ t^\upsilon}\cos\left(\frac{x_1+\frac{N+2}{2\lambda^*}\ln t}{ t^\rho}\right),~~~~~~~~~t\ge  T_\varep, ~~\mathcal{Y}^-(t)-R\le x_1 \le \mathcal{Y}^+(t)+R.
 	\end{equation}

 	At time $t= T_\varep$, we observe that $\overline{\mathcal{S}}( T_\varep,x)>\frac{\mathcal{B}}{2} T_\varep^{-\upsilon}\ge V(\overline{T}_\varep,x)\ge \mathcal{S}(T_\varep,x)$ for $\mathcal{Y}^-( T_\varep)-R\le x_1 \le \mathcal{Y}^+(T_\varep)+R$. For $x\in[\mathcal{Y}^-(t)-R,\mathcal{Y}^-(t)]\cup [\mathcal{Y}^+(t),\mathcal{Y}^+(t)+R]$,  up to further increasing $T_\varep$ if necessary, there holds $\overline{\mathcal{S}}(t,x)>\frac{\mathcal{B}}{2} t^{-\upsilon}>Ce^{-\lambda^* t^\mu}\ge \mathcal{S}(t,x)$ for $t\ge T_\varep$. Eventually,  a direct computation gives that
 	\begin{align*}
 		\displaystyle	\overline{\mathcal{S}}_t +\mathcal{I}_{*}\overline{\mathcal{S}}-\frac{N+2}{2t}
 		\overline{\mathcal{S}}&=\Big(\frac{-\upsilon}{ t}-\frac{N+2}{2t}+\frac{h_{11}}{2 t^{2\rho}}\Big)\overline{\mathcal{S}}-\frac{\mathcal{B}(N+2)}{2\lambda^*t^{\upsilon+\rho+1}}
 		\sin\left(\frac{x_1+\frac{N+2}{2\lambda^*}\ln t}{ t^\rho}\right)+O(t^{-3\rho})\\
 		&\ge \frac{C}{ t^{2\rho+\upsilon}}\gg \frac{C}{ t^{1-\mu}}, ~~~~~~~~~t\ge T_\varep, ~~\mathcal{Y}^-(t)\le x_1 \le \mathcal{Y}^+(t).
 	\end{align*} 
 	Together with $\mathcal{H}(t,x;\overline{\mathcal{S}})\ge 0$ uniformly for $t\ge T_\varep$ and $x\in\{x\in\R^N|\mathcal{Y}^-(t)\le x_1 \le \mathcal{Y}^+(t)\}$, we then conclude that $\overline{\mathcal{S}}(t,x)$ is a supersolution of \eqref{1-s_varep} for $t\ge T_\varep$ and $x\in\{x\in\R^N|\mathcal{Y}^-(t)\le x_1 \le \mathcal{Y}^+(t)\}$. The comparison principle implies that $\mathcal{S}(t,x)\le \overline{\mathcal{S}}(t,x)$ for $t\ge T_\varep$ and $x\in\{x\in\R^N|\mathcal{Y}^-(t)\le x_1 \le \mathcal{Y}^+(t)\}$. Thus,
 	\begin{equation*}
 		V(t,x)-	\psi_{\kappa+\varep}(t,x)\le 	\mathcal{S}(t,x)\le \overline{\mathcal{S}}(t,x)=o_{t\to+\infty}(1), 
 	\end{equation*}
 	uniformly in $x\in\{x\in\R^N|\mathcal{Y}^-(t)\le x_1 \le \mathcal{Y}^+(t)\}$.
 	
 	\medskip
 	
 	\noindent
 	\textbf{Step 2}.
 	It remains to show $\limsup_{t\to+\infty}\big(\psi_{\kappa-\varep}(t,x)-V(t,x) \big)\le 0$,  uniformly in $\mathcal{Y}^-(t)\le x_1 \le \mathcal{Y}^+(t)$ and $x'$ in any compact, for which we follow the same lines as above and introduce indispensible modifications.   With a slight abuse of notation, we now define
 	$\mathcal{S}(t,x):=(\psi_{\kappa-\varep}-V)^+(t,x)$ for $t\ge  T_\varep$ and $x\in\{x\in\R^N|\mathcal{Y}^-(t)\le x_1 \le \mathcal{Y}^+(t)\}$. We are  led to the problem
 	\begin{equation*}
 		\begin{aligned}
 			\begin{cases}
 				\displaystyle	\Big| \mathcal{S}_t +\mathcal{I}_{*}\mathcal{S}-\frac{N+2}{2t}
 				\mathcal{S}+\mathcal{H}(t,x;\mathcal{S})\Big|<C t^{\mu-1},~~~&t\ge  T_\varep,~\mathcal{Y}^-(t)\le x_1 \le \mathcal{Y}^+(t),\vspace{3pt}\\
 				\displaystyle	\mathcal{S}(t,x) \le   e^{-\lambda^* t^\mu}, ~~~~~~ ~~~~~~~~~~~~~~~~~ &t\ge  T_\varep,~\mathcal{Y}^-(t)-R\le x_1\le  \mathcal{Y}^-(t),\vspace{3pt}\\
 				\mathcal{S}(t,x)  =t^\mu\mathbbm{1}_{\{|x'|\ge \eta_2\sqrt{t}\}}, ~~~~~~~~~~~~~~~~~~~~~~~~~~~~~~~~~~~~~~~~~&t\ge  T_\varep,~\mathcal{Y}^+(t)\le x_1\le  \mathcal{Y}^+(t)+R,\vspace{3pt}\\
 				\displaystyle	\mathcal{S}( T_\varep,x)\le\psi_{\kappa-\varep}( T_\varep,x), &\mathcal{Y}^-( T_\varep)-R\le x_1 \le \mathcal{Y}^+( T_\varep)+R.
 			\end{cases}
 		\end{aligned}
 	\end{equation*}
 	Here, parameter $\eta_2>0$ is given after \eqref{parameters-n}, and $\mathcal{H}(t,x;\mathcal{S})$ now is given by
 	\begin{align*}
 		\mathcal{H}(t,x;\mathcal{S}):=Q(t,x;\psi_{\kappa-\varep})-Q(t,x;V)=f'(0)\mathcal{S}-d(t,x)\mathcal{S}\ge0, ~~~~~~\mathcal{S}\ge 0,
 	\end{align*} 
 	uniformly for $t\ge  T_\varep$ and $x\in\{x\in\R^N|\mathcal{Y}^-(t)\le x_1 \le \mathcal{Y}^+(t)\}$.

 	Define
 	\begin{equation*}
 		\overline{\mathcal{S}}^\star (t,x)=\overline{\mathcal{S}}(t,x)
 		+t^\mu\chi_1\Big(\frac{|x'|}{\sqrt{t}}\Big)
 		,~~~~~~~~~t\ge  T_\varep, ~~\mathcal{Y}^-(t)-R\le x_1 \le \mathcal{Y}^+(t)+R.
 	\end{equation*}
 	where $\overline{\mathcal{S}}$ is given in \eqref{bar S}, and we recall that the nondecreasing function $r\in\R_+\mapsto\chi_1(r)$ is given at the beginning of Section \ref{sec4}.
 	
 	At time $t= T_\varep$, we observe that $\overline{\mathcal{S}}^\star ( T_\varep,x)>\frac{\mathcal{B}}{2} T_\varep^{-\upsilon}\ge\max_{x_1\in[\mathcal{Y}^-( T_\varep)-R,\mathcal{Y}^+( T_\varep)+R]}\psi_{\kappa+\varep}( T_\varep,x)\ge \mathcal{S}( T_\varep,x)$ for $\mathcal{Y}^-( T_\varep)-R\le x_1 \le \mathcal{Y}^+( T_\varep)+R$. For $x_1\in[\mathcal{Y}^-(t)-R,\mathcal{Y}^-(t)]$,  up to further increasing $ T_\varep$ if necessary, there holds $\overline{\mathcal{S}}^\star(t,x)>\frac{\mathcal{B}}{2} t^{-\upsilon}>Ce^{-\lambda^* t^\mu}\ge \mathcal{S}(t,x)$ for $t\ge  T_\varep$. For $x_1\in [\mathcal{Y}^+(t),\mathcal{Y}^+(t)+R]$, we have $\overline{\mathcal{S}}^\star(t,x)>\frac{\mathcal{B}}{2} t^{-\upsilon}>0= \mathcal{S}(t,x)$ for $t\ge  T_\varep$ and $|x'|\le \eta_2\sqrt{t}$, and $\overline{\mathcal{S}}^\star(t,x)>\frac{\mathcal{B}}{2} t^{-\upsilon}+t^\mu\ge \mathcal{S}(t,x)$ for $t\ge  T_\varep$ and $|x'|\ge \eta_2\sqrt{t}$. Moreover, one can check that
 	\begin{align*}
 		\displaystyle	\overline{\mathcal{S}}^\star_t +\mathcal{I}_{*}\overline{\mathcal{S}}^\star-\frac{N+2}{2t}
 		\overline{\mathcal{S}}^\star
 		&\ge \frac{C}{ t^{2\rho+\upsilon}}+ \frac{C}{ t^{1-\mu}}+ \frac{C}{ t^{\frac{3}{2}-\mu}}\gg \frac{C}{ t^{1-\mu}}, ~~~~~~~~~t\ge  T_\varep, ~~\mathcal{Y}^-(t)\le x_1 \le \mathcal{Y}^+(t).
 	\end{align*} 
 	Together with $\mathcal{H}(t,x;\overline{\mathcal{S}}^\star)\ge 0$ uniformly for $t\ge  T_\varep$ and $x\in\{x\in\R^N|\mathcal{Y}^-(t)\le x_1 \le \mathcal{Y}^+(t)\}$, we then conclude that $\overline{\mathcal{S}}^\star(t,x)$ is a supersolution of \eqref{1-s_varep} for $t\ge  T_\varep$ and $x\in\{x\in\R^N|\mathcal{Y}^-(t)\le x_1 \le \mathcal{Y}^+(t)\}$. The comparison principle implies that $\mathcal{S}(t,x)\le \overline{\mathcal{S}}^\star(t,x)$ for $t\ge  T_\varep$ and $x\in\{x\in\R^N|\mathcal{Y}^-(t)\le x_1 \le \mathcal{Y}^+(t)\}$. Thus,
 	\begin{equation*}
 		\psi_{\kappa-\varep}(t,x)-V(t,x)\le 	\mathcal{S}(t,x)\le \overline{\mathcal{S}}^\star(t,x)=o_{t\to+\infty}(1), 
 	\end{equation*}
 	uniformly for $\mathcal{Y}^-(t)\le x_1 \le \mathcal{Y}^+(t)$ and $x'$ in any compact. This finishes the proof. 
 \end{proof}

 Proposition \ref{prop-5.3}, along with the definition  \eqref{1-TW shift} of $\psi_{\kappa\pm\varep}$, gives that
 \begin{equation}\label{5.10}
 	\Big|V(t,x)- e^{\lambda^*(x_1+\frac{N+2}{2\lambda^*}\ln t)} U_{c^*}\Big(x_1+\frac{N+2}{2\lambda^*}\ln t+\zeta_{\kappa\pm\varep}(t)\Big)\Big| \to 0~~~\text{as}~t\to+\infty,
 \end{equation}
 uniformly for $|x_1+\frac{N+2}{2\lambda^*}\ln t|\le  t^\mu$ and $x'$ in any compact, where
 \begin{equation*}
 	\zeta_{\kappa\pm\varep}(t)=-\frac{1}{\lambda^*}\ln(\kappa\pm\varep)+\mathcal{O}( t^{2\mu-1}).
 \end{equation*}
 Since $\varep>0$ is chosen arbitrarily small, one can pass to the limit in \eqref{5.10} by taking $\varep\to 0$ to conclude that
 \begin{equation*}
 	\Big|V(t,x)- e^{\lambda^*(x_1+\frac{N+2}{2\lambda^*}\ln t)}U_{c^*}\Big(x_1+\frac{N+2}{2\lambda^*}\ln t-z_\infty\Big)\Big| \to 0~~~\text{as}~t\to+\infty, 
 \end{equation*}
 uniformly for $|x_1+\frac{N+2}{2\lambda^*}\ln t|\le  t^\mu$ and $x'$ in any compact,  with $z_\infty:=\frac{1}{\lambda^*}\ln\kappa$ depending on $u_0$ (remember that $\kappa>0$ is given in Theorem \ref{t6.1}). 
 This implies that 
 \begin{equation}\label{-1}
 	\max_{0\le x_1+\frac{N+2}{2\lambda^*}\ln t\le t^\mu}\Big|\bu(t,x)-U_{c^*}\Big(x_1+\frac{N+2}{2\lambda^*}\ln t-z_\infty\Big)\Big|\to 0~~\text{as}~t\to+\infty,
 \end{equation}
 uniformly for $x'$ in any compact. This completes the proof. 
 
 \medskip
We end this section by two corollaries whose proofs follow the same lines as above.
\begin{corollary} 
	\label{cor_6.2}
	For $x_1=t^\mu+o(t^\mu)$, and for all $\rho>0$, we have
	\begin{itemize}
		\item[(i)] if
		$ v(t,x_1,0)\le	\kappa_+ x_1 t^{-\frac{N}{2}-1}$ with $\kappa_+>0$, then $$\limsup_{t\to+\infty}\Big(\bu(t,x)- U_{c^*}\Big(x_1+\di\frac{N+2}{2\lambda^*}\ln t-z_+\Big)\Big)\le 0$$ uniformly in $0\le x_1+\di\frac{N+2}{2\lambda^*}\ln t\le t^\mu$ and $\vert x'\vert\leq\rho$, with $z_+=\di\frac{1}{\lambda^*}\ln\kappa_+$;
		\item[(ii)] if	$ v(t,x_1,0)\ge	\kappa_- x_1 t^{-\frac{N}{2}-1}$ with $\kappa_->0$, then  $$\liminf_{t\to+\infty}\Big(\bu(t,x_1,0)- U_{c^*}\Big(x_1+\frac{N+2}{2\lambda^*}\ln t-z_-\Big)\Big)\ge 0$$ uniformly in $0\le x_1+\di\frac{N+2}{2\lambda^*}\ln t\le t^\mu$ and $\vert x'\vert\leq\rho$,  with $z_+=\di\frac{1}{\lambda^*}\ln\kappa_-$.
	\end{itemize}
\end{corollary}
\begin{corollary} 
	\label{cor_6.3}
	For all $\rho>0$, there is $\overline \kappa_\rho>0$ such that, for all $x=(x_1,x')\in[1,+\infty)\times B_\rho(0)$ we have  
	$$\mathbf{u}(t,x)\leq\overline\kappa_\rho\bigl(x_1+\displaystyle\frac{N+2}{2\lambda^*}\ln t\bigl)\frac{e^{-\lambda^*x_1}}{t^{\frac{N}2+1}}.
	$$
	\end{corollary}

\appendix
\section{Proof of Proposition \ref{p3.1}}\label{A1}
Let us consider $\varpi\in(0,\displaystyle\frac12)$ to be chosen later, and set
$$
\begin{array}{rll}
	Q_{*,n}(s,\tau;x)=&\di\frac1{(2\pi)^N}\int_{\vert\xi\vert\leq s^{\varpi -1/2}}e^{\mathbf{i}x\cdot\xi+\tau\left(\widehat{\ck_{*,n}}(\xi)+\mathbf{i}w^*(e) e\cdot\xi-\widehat{\ck_{*,n}}(0)\right)}\widehat{v_s}(\xi)\md\xi\\
	=&\di\frac1{(2\pi)^N}\int_{\vert\xi\vert\leq s^{\varpi -1/2}}e^{\mathbf{i}(x-y)\cdot\xi+\tau\left(\widehat{\ck_{*,n}}(\xi)+\mathbf{i}w^*(e) e\cdot\xi-\widehat{\ck_{*,n}}(0)\right)}{v_s}(y)\md y\md\xi\\
	=&\di\frac{1}{(2\pi)^N}\int_{\vert\zeta\vert\leq s^{\varpi}}\int_{\R^N}e^{\mathbf{i}(\frac{x}{\sqrt s}-z)\cdot\zeta+\tau\left(\widehat{\ck_{*,n}}(\frac{\zeta}{\sqrt s})+\mathbf{i}w^*(e) e\cdot\frac{\zeta}{\sqrt s}-\widehat{\ck_{*,n}}(0)\right)}\bnu(s,z)\md z\md\zeta.
\end{array}
$$
In the last integral we have used the changes of variables $y=\sqrt s z$ and $\zeta=\sqrt s\xi$. Let us also set
$
R_{*,n}(s,\tau;x)=e^{-\tau\cI_*}v_s(x)-Q_{*,n}(s,\tau;x).$ In the sequel, $\widehat\bnu(s,.)$ will denote the Fourier transform of $\bnu$ with respect to its second variable.  Consider an integer $p$ such that $2p>N+1$; we have
$$
\begin{array}{rll}
	R_{*,n}(s,\tau;x)=&\di\frac1{(2\pi)^N}\int_{\vert\xi\vert\geq s^{\varpi -1/2}}e^{\mathbf{i}x\cdot\xi+\tau\left(\widehat{\ck_{*,n}}(\xi)+\mathbf{i}w^*(e) e\cdot\xi-\widehat{\ck_{*,n}}(0)\right)}\widehat{v_s}(\xi)\md\xi\\
	\leq&\di\bigl(\frac{\sqrt s}{2\pi}\bigl)^{N}\di\int_{\vert\xi\vert\geq s^{\varpi -1/2}}\widehat\bnu(s,\sqrt s\xi)\md\xi\\
	=&\di\frac{1}{(2\pi)^N}\di\int_{\vert\zeta\vert\geq s^{\varpi}}\widehat\bnu(s,\zeta)\md\zeta\\
	\leq&\di\frac{1}{2\pi}\biggl(\int_{\vert\zeta\vert\geq s^{\varpi}}\frac{\md\zeta}{(1+\vert \zeta\vert^p)^2}\biggl)^{1/2}\biggl(\int_{\vert\zeta\vert\geq s^\varpi}\bigl(1+\vert\zeta\vert^{p})^2\bigl)\vert\widehat\bnu(s,\zeta)\vert^2\md\zeta\biggl)^{1/2}
\end{array}
$$
We have successively used  the classical identity $\widehat v_s(\xi)=s^{\frac{N}2}\widehat\bnu(s,\sqrt s\xi)$, the fact  that the real part of $\widehat \ck_*(0)-\widehat\ck_*(\xi)$ is nonnegative, and the fact that $v_0$ is in  $H^m(\R^N)$ for all integer $m$.
So, there is   $C_p>0$ such that
\begin{equation}
	\label{eA.1}
	\vert R_{*,n}(s,\tau;x)\vert\leq \frac {C_p\Vert\bnu(s,.)\Vert_{H^p}}{\sqrt{(2p-N-1)s^{2p-N-1}}}.
\end{equation}
Let us now deal with $Q_{*,n}(s,\tau;x)$; for that, it is convenient to set, for $y\in\R^N$ and $\vert\zeta\vert\leq s^{\varpi}$:
$$
\Phi_{*,n}(s,\tau;y,\zeta)=\mathbf{i}y\cdot\zeta+\tau\left(\widehat{\ck_{*,n}}(\frac{\zeta}{\sqrt s})+\mathbf{i}w^*(e) e\cdot\frac{\zeta}{\sqrt s}-\widehat{\ck_{*,n}}(0)\right),
$$
and 
$$
G_{*,n}(s,\tau;y)=\int_{\vert\zeta\vert\leq s^{\varpi}}e^{\Phi_{*,n}(s,\tau;y,\zeta)}\md\zeta.
$$
The function $v_0$ belongs to $L^1(\R^N)$; Fubini's theorem asserts therefore that order of the integrals entering in the computation of $Q_{*,n}$ may be inverted. So we have:
$$
Q_{*,n}(s,\tau;x)=\frac1{(2\pi)^N}\int_{\RR^N}G_{*,n}(s,\tau;\frac{x}{\sqrt s}-z)\bnu(s,z)\md z.
$$
Recall the identities $\nabla\widehat\ck_{*,n}(0)=-\mathbf{i}\displaystyle\int\ck_{*,n}(x)\md x=-\mathbf{i}w^*(e)e$ and $D^2\widehat\ck_{*,n}(0)=-H_n$.  Let now  $q$ be an integer to be chosen later. As $\displaystyle\frac{\vert\zeta\vert}{\sqrt s}\leq s^{\varpi-\frac12}$, one may perform a Taylor expansion of  $\widehat\ck_{*,n}(\displaystyle\frac{\zeta}{\sqrt s})$ around 0 up to the order $q$ and obtain the following expression for $\Phi_{*,n}$:
$$
\Phi_{*,n}(s,\tau;y,\zeta)=\mathbf{i}y\cdot\zeta-\frac\tau{2s}\zeta^TH_n\zeta+\tau\bigl(\sum_{3\leq\vert\alpha\vert\leq q}\frac{a_\alpha\zeta^\alpha}{s^{\frac{\vert\alpha\vert}2}}+\psi_{*,n;q}(\frac{\zeta}{\sqrt s})\bigl),
$$
with $\vert \psi_{*,n;q}(\displaystyle\frac{\zeta}{\sqrt s})\vert\leq C_qs^{-(\frac12-\varpi)q}$  and $C_q$ a constant depending only on $q$ and the data.  
The precise expressions of the $a_\alpha$'s is of no interest to us.  As  $\tau\leq s^\varsigma$ and $\varsigma$ is assumed to be small, the quantity $\displaystyle\frac{\tau}{\sqrt s}$ is small and, from the use of the classical Taylor expansion of the exponential, the following expression of $e^{\Phi_{*,n}}$  is available:
$$
e^{\Phi_{*,n}(s,\tau;y,\zeta)}=\bigl(1+\sum_{3\leq\vert\alpha\vert\leq q}\frac{P_\alpha(\tau)\zeta^\alpha}{s^{\frac{\vert\alpha\vert}2}}+\tilde\psi_{*,n;q}(\tau,\frac{\zeta}{\sqrt s})\bigl)e^{\mathbf{i}y\cdot\zeta-\frac\tau{2s}\zeta^TH_n\zeta}.
$$
In the sum, the  function $\tilde \psi$ is estimated as
\begin{equation}
	\label{A10}
	\vert \tilde\psi_{*,n;q}(\tau,\displaystyle\frac{\zeta}{\sqrt s})\vert\leq C_q\tau^qs^{-(\frac12-\varpi)q}\leq C_q{s^{-(\frac12-2\varpi)q},}
\end{equation}
with a possibly different $C_q$. The function $P_\alpha$ is a polynomial of degree $\leq\vert\alpha\vert$, and we do not seek any precise expression of its coefficients. We may now estimate $Q_{*,n}$: we have 
$$
\begin{array}{rll}
	&(2\pi)^NQ_{*,n}(s,\tau;x)\\
	=&\displaystyle\int_{\RR^N}\biggl(\int_{\vert\zeta\vert\leq s^{\varpi}}\bigl(1+\sum_{3\leq\vert\alpha\vert\leq q}\frac{P_\alpha(\tau)\zeta^\alpha}{s^{\frac{\vert\alpha\vert}2}}+\tilde\psi_{*,n;q}(\tau,\frac{\zeta}{\sqrt s})\bigl)e^{\mathbf{i}(\frac{x}{\sqrt s}-z)\cdot\zeta-\frac\tau{2s}\zeta^TH_n\zeta}\md\zeta\biggl)\bnu(s,z)\md z\\
	=&\displaystyle\int_{\RR^N}\int_{\RR^N}\bigl(1+\sum_{3\leq\vert\alpha\vert\leq q}\frac{P_\alpha(\tau)\zeta^\alpha}{s^{\frac{\vert\alpha\vert}2}}\bigl)e^{\mathbf{i}(\frac{x}{\sqrt s}-z)\cdot\zeta-\frac\tau{2s}\zeta^TH_n\zeta}\bnu(s,z)\md\zeta\md z\\
	&+\displaystyle\int_{\RR^N}\biggl(\int_{\vert\zeta\vert\geq s^{\varpi}}\bigl(1+\sum_{3\leq\vert\alpha\vert\leq q}\frac{P_\alpha(\tau)\zeta^\alpha}{s^{\frac{\vert\alpha\vert}2}}\bigl)e^{\phi_{n,*}(s,\tau;z,x)}\md\zeta+\int_{\vert\zeta\vert\leq s^{\varpi}}\tilde\psi_{*,n;q}(\tau,\frac{\zeta}{\sqrt s})\bigl)e^{\phi_{n,*}(s,\tau;z,x)}\md\zeta\biggl)\bnu(s,z)\md z
\end{array}
$$
where we have denoted, because its precise expression will not matter much in the last two integrals, $\phi_{n,*}(s,\tau;z,x)=\displaystyle \mathbf{i}(\frac{x}{\sqrt s}-z)\cdot\zeta-\frac\tau{2s}\zeta^TH_n\zeta$.  The first integral is 
$$ \bigl(e^{\frac{\tau}{s}A_n}\bnu(s,.)\bigl)(\frac{x}{\sqrt s})+\sum_{3\leq\vert\alpha\vert\leq q}\mathbf{i}^{\vert\alpha\vert}\frac{P_\alpha(\tau)\zeta^\alpha}{s^{\frac{\vert\alpha\vert}2}}\bigl(e^{\frac{\tau}{s}A_n}\partial^\alpha \bnu(s,.)\bigl)(\frac{x}{\sqrt s}),
$$
which is the main term announced in the proposition. We set $Q_\alpha(\tau)=\mathbf{i}^{\vert\alpha\vert}P_\alpha$; as $\bnu(s,.)$ is real valued this polynomial has to be real valued -- one could check it from the expansion of the exponential, but that would be a tedious and not so useful task. As for the remaining terms, we treat them like $R_{*,n}(s,\tau;x)$, that is: consider an integer $p> \displaystyle\frac{N+1}2$ to be chosen later.   We additionally impose $\tau^q\leq s^{\frac18}$, that is, $q\varsigma\leq\displaystyle\frac18$. We have, for a constant $C$ depending on $q$, $m$ $\varpi$, but not on $s$:
$$
\begin{array}{rll}
	&\displaystyle\biggl\vert\int_{\RR^N}\int_{\vert\zeta\vert\geq s^{\varpi}}\bigl(1+\sum_{3\leq\vert\alpha\vert\leq q}\frac{P_\alpha(\tau)\zeta^\alpha}{s^{\frac{\vert\alpha\vert}2}}\bigl)e^{\phi_{n,*}(s,\tau;z,x)}\bnu(s,z)\md z\md\zeta\biggl\vert\\
	=&\displaystyle\biggl\vert\int_{\vert\zeta\vert\geq s^{\varpi}}e^{\mathbf{i}\zeta\cdot\frac{x}{\sqrt s}}\bigl(1+\sum_{3\leq\vert\alpha\vert\leq q}\frac{P_\alpha(\tau)\zeta^\alpha}{s^{\frac{\vert\alpha\vert}2}}\bigl)e^{-\frac\tau{2s}\zeta^TH_n\zeta}\widehat \bnu(s,\zeta)\md\zeta\biggl\vert\\
	\leq&C\displaystyle\int_{\vert\zeta\vert\geq s^{\varpi}}(1+\vert\zeta\vert^q)\vert\widehat \bnu(s,\zeta)\vert\md\zeta\\
	\leq&C\biggl(\displaystyle\int_{\vert\zeta\vert\geq s^{\varpi}}\frac{\md\zeta}{1+\vert\zeta\vert^{2p}}\biggl)^{\frac12}\biggl(\int_{\vert\zeta\vert\geq s^\varpi}(1+\vert\zeta\vert^{2q+2p}\vert \widehat \bnu(s,\zeta)\vert^2\md\zeta\biggl)^{\frac12}\leq \frac{C\Vert \bnu(s,.)\vert_{H^{p+q}}}{s^{(2p-N-1)\varpi}}.
\end{array}
$$
The remaining integral is estimated in a similar sprit; given the estimate \eqref{A10} for $\tilde\psi$,  we may bound it by $Cs^{-(\frac12-\varpi)q+\varpi}\Vert\bnu(s,.)\Vert_{L^1}$ for a constant $C$ depending on $q$ and $\varpi$, but not on $s$. It now remains to choose the integers $p$ and $q$, as well as the constant $\varpi$, such that 
$$
(2p-N-1)\varpi\geq\frac{N}2+1,~~(\frac12-\varpi)q-\varpi\geq\frac{N}2+1,
$$
This is possible with $\varpi=\displaystyle\frac14$, $q=2N+5$, and $p=3N+2$.  Finally we choose $\varsigma>0$ such that  $q\varsigma\leq\displaystyle\frac18$, that is, 
$\varsigma\leq\displaystyle\frac1{8(2N+5)}$, and $m=p+q=5N+7$, as claimed. \hfill $\Box$


\end{document}